\documentclass[12pt,a4paper]{article}
\usepackage{amsmath}
\usepackage{amssymb}
\usepackage{amsxtra}
\usepackage{latexsym}
\usepackage{graphicx}
\usepackage{amsthm}

\newcommand{\MRe}{{\rm Re}}
\newcommand{\Ker}{{\rm Ker}}
\newcommand{\MIm}{{\rm Im}}
\newcommand{\dd}{\;{\rm d}}

\newcommand{\Mi}{{\rm i}}

\newcommand{\tg}{{\rm tg}}

\newtheorem{thr}{\rm\bf Theorem}[section]
\newtheorem{lem}[thr]{\rm\bf Lemma}
\newtheorem{ut}[thr]{\rm\bf Proposition}
\newtheorem{col}[thr]{\rm\bf Corollary}
\newtheorem{rem}[thr]{\rm\bf Remark}
\newtheorem{df}[thr]{\rm\bf Definition}

\newtheorem{notat}{\rm\bf Notation}[section]

\date{}
\title{Stability and stabilizability of mixed retarded-neutral type systems}

\author{
R.~Rabah\thanks{IRCCyN/\'Ecole des Mines de Nantes, 4 rue Alfred Kastler, BP 20722, 44307, Nantes, France
({\tt rabah@emn.fr, rabah@irccyn.ec-nantes.fr}).},
G.~M.~Sklyar\thanks{Institute of Mathematics, University of Szczecin,
Wielkopolska 15, 70-451, Szczecin, Poland ({\tt sklar@univ.szczecin.pl, sklyar@univer.kharkov.ua}).}
and P.~Yu.~Barkhayev\thanks{IRCCyN/\'Ecole  Centrale  de Nantes. Permanent address: Department of Differential Equations and Control,
Kharkov National University, 4 Svobody sqr., 61077, Kharkov, Ukraine ({\tt pbarhaev@inbox.ru}).}}

\begin{document}

\voffset-2.9in
\thispagestyle{empty}
\begin{center}

\begin{tabular}{lr}
\includegraphics[scale=0.17]{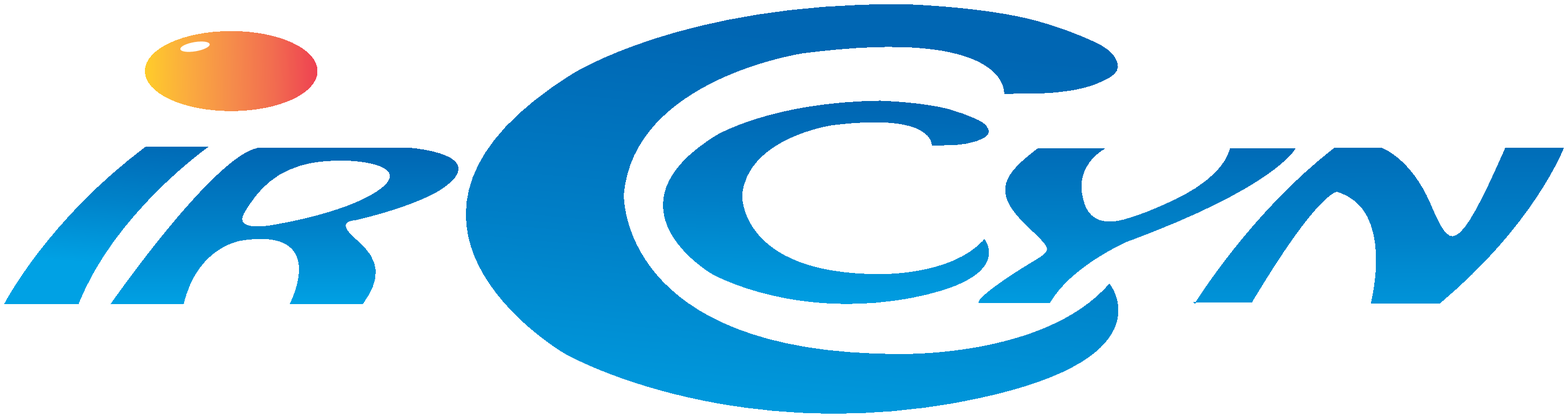} & $\qquad\qquad\qquad$
\begin{small}
\begin{tabular}{c}
Institut de Recherche en Communications\\
 et en Cybern\'{e}tique de Nantes\\
 UMR CNRS 6597
\end{tabular}
\end{small}
\end{tabular}

\vskip 20ex {\huge \bf Stability and stabilizability\\
of mixed retarded-neutral type systems}
\end{center}
\vskip 15ex
\begin{center}
{\Large \bf R. Rabah \hskip 5 mm  G.~M.~Sklyar \hskip 5 mm P.~Yu.~Barkhayev}
\vskip 5 cm
{\bf \large IRCCyN, Internal Report}
\vskip 10ex
\today
\end{center}
\vskip 15ex

\begin{tiny}
\begin{tabular}{cc}
$\qquad$ &
\begin{tabular}{c}
{\bf \'Etablissement de rattachement administratif et adresse:}\\
{\bf IRCCyN $\bullet$ \'{E}cole Centrale de Nantes}\\
 1, rue de la No\"{e} $\bullet$ BP 92101 $\bullet$ 44321 Nantes Cedex 3 $\bullet$ France\\
T\'{e}l. +33 (0)2 40 37 16 00 $\bullet$ Fax. +33 (0)2 40 37 69 30 \\
\\
\hline
\\
{\bf Unit\'{e} Mixte: \'{E}cole Centrale de Nantes, Universit\'{e} de Nantes, \'{E}cole des Mines de Nantes, CNRS.}
\end{tabular}
\end{tabular}
\end{tiny}

\setlength{\textheight}{25cm}
\voffset-1.4in
\newpage
\maketitle
\begin{abstract}{
We analyze the stability and stabilizability properties of mixed
retarded-neutral type systems when the neutral term is allowed to be
singular. Considering an operator model of the system in a Hilbert
space we are interesting in the critical case when there exists a
sequence of eigenvalues with real parts approaching to zero. In this
case the exponential stability is not possible and we are studying
the strong asymptotic stability property. The behavior of spectra of
mixed retarded-neutral type systems does not allow to apply directly
neither methods of retarded system nor the approach of neutral type
systems for analysis of stability. In this paper two technics are
combined to get the conditions of asymptotic non-exponential
stability: the existence of a Riesz basis of invariant
finite-dimensional subspaces and the boundedness of the resolvent in
some subspaces of a special decomposition of the state space. For
unstable systems the technics introduced allow to analyze the
concept of regular strong stabilizability for mixed retarded-neutral
type systems. The present paper extends the results on stability
obtained in [R.~Rabah, G.M.~Sklyar,  A. V.~Rezounenko, Stability
analysis of neutral type systems in Hilbert space. J. of
Differential Equations, 214(2005), No.~2, 391--428] and the results
on stabilizability from [R.~Rabah, G.M.~Sklyar,  A. V.~Rezounenko,
On strong regular stabilizability for linear neutral type systems.
J. of Differential Equations, 245(2008), No.~3, 569--593]. Comparing
with the mentioned papers, here we avoid a restrictive assumption of
non-singularity of the main neutral term.


}
\end{abstract}
\section{Introduction}

The interest in considering delay differential equations and
corresponding infinite-dimen\-sional dynamical systems is caused by a
huge amount of applied problem which can be described by these
equations. The stability theory of such type of systems was studied
intensively (see e.g. \cite{Bellman_Cooke_1963, Hale_Verduyn_1993, Kolmanovskii_Nosov_1986}). Number of
results was obtained for retarded systems, however an analysis of
neutral type systems is much  more complicated and these systems are
still studied not so deeply.
We consider neutral type systems given by the following functional differential
equation:
\begin{equation}\label{syst_delay_gen_kt}
\frac{\rm d}{{\rm d}t}[z(t)-Kz_t]=Lz_t + Bu(t),\qquad t \ge 0,
\end{equation}
where $z_t: [-1,0]\rightarrow \mathbb{C}^n$ is the history of $z$ defined by $z_t(\theta)=z(t+\theta)$.
We assume the delay operator $L: H^1([-1,0], \mathbb{C}^n)\rightarrow  \mathbb{C}^n$ to be linear and bounded, thus, it has the following form:
\begin{equation}\label{syst_delay_gen_operators_l}
Lf=\int_{-1}^0 A_2(\theta)f'(\theta)\;{\rm d}\theta +\int_{-1}^0 A_3(\theta)f(\theta) \;{\rm d}\theta,\qquad f\in H^1([-1,0], \mathbb{C}^n),
\end{equation}
where $A_2$, $A_3$ are $n\times n$-matrices whose elements belong to $L_2([-1,0], \mathbb{C})$.
We take the difference operator $K$ in the form
\begin{equation}\label{syst_delay_gen_operators_k}
Kf=A_{-1}f(-1),
\end{equation}
where $A_{-1}$ is a constant $n\times n$-matrix.
The form~(\ref{syst_delay_gen_operators_k}) may be considered as a particular case of the operator
$K: C([-1,0],\mathbb{C}^n)\rightarrow \mathbb{C}^n$ given by $K f=\int^0_{-1}{\rm d}\mu(\theta)f(\theta)$,
where $\mu(\cdot):[-1,0]\to \mathbb{C}^{n\times n}$ is of bounded variation and continuous at zero.
However, the considered model~(\ref{syst_delay_gen_operators_k}) is sufficiently general. Its analysis is difficult enough
and the results obtained are derived, in part, from the properties of the matrix $A_{-1}$.

The well-known approach, when studying systems of the form~(\ref{syst_delay_gen_kt}),
is to consider a corresponding infinite-dimensional model $\dot{x}={\mathcal A}x$, where ${\mathcal A}$ is
the infinitesimal generator of a $C_0$-semigroup.
For systems~(\ref{syst_delay_gen_kt})--(\ref{syst_delay_gen_operators_k}) the resolvent of
the operator ${\mathcal A}$ allows an explicit representation (see \cite{Rabah_Sklyar_Rezounenko_2003,Rabah_Sklyar_Rezounenko_2005}).
Such a representation is an effective tool for analyzing the exponential stability property since the last is equivalent to
the uniform boundedness of the resolvent on the complex right half-plane.
The resolvent boundedness approach is exhaustive when one considers stability of pure retarded type systems ($A_{-1}=0$)
since such systems may be exponentially stable or unstable only.
This fact is due to that exponential growth of the semigroup $\{{\rm e}^{t{\mathcal A}}\}_{t\ge 0}$ is determinated by spectrum's location
and there are only a finite number of eigenvalues of ${\mathcal A}$ in any half-plane $\{\lambda:\; {\rm Re} \lambda \ge C\}$.

For neutral-type systems ($A_{-1}\not=0$) in addition to the notion of exponential stability,
which is characterized by the condition that the spectrum is bounded away from the imaginary axis
(see \cite[Theorem~6.1]{Henry_1974}, \cite{Hale_Verduyn_1993}),
one meets the notion of {\it strong} asymptotic non-exponential stability.
This type of stability may happen in some critical case when the exponential stability is not possible (see e.g.~\cite{Brumley_1970}).
Thus, strong stability cannot be described in terms of the resolvent boundedness.
In \cite{Rabah_Sklyar_Rezounenko_2003,Rabah_Sklyar_Rezounenko_2005} for neutral type systems with a nonsingular
neutral term ($\det A_{-1}\not=0$)
this type of stability was precisely investigated for systems of the form~(\ref{syst_delay_gen_kt})--(\ref{syst_delay_gen_operators_k})
and some necessary and sufficient conditions of strong stability and instability had been proved.
The proofs are based on such a powerful tool as existence of a Riesz basis of $\mathcal A$-invariant finite-dimensional subspaces of the state space
and on further application of the results on strong stability in Banach spaces that had been originated in \cite{Sklyar_Shirman_1982}
and later developed in \cite{Arendt_Batty_1988,Lyubich_Phong_1988,Sklyar_Rezounenko_JMAA_2001,Sklyar_Rezounenko_ASP_2001} and many others
(see e.g. \cite{Neerven_1996} for a review).

In the case of neutral type systems with a singular neutral term ($\det A_{-1}=0$ and $A_{-1}\not=0$),
which we call {\it mixed retarded-neutral} type systems, the strong stability may also happen.
However, the approach given in \cite{Rabah_Sklyar_Rezounenko_2003,Rabah_Sklyar_Rezounenko_2005} cannot be directly applied
to such systems, since the existence of a Riesz basis of $\mathcal A$-invariant finite-dimensional subspaces of the whole state space
cannot be guarantied.
Moreover, mixed retarded-neutral type systems, in general, cannot be decomposed onto systems of pure neutral and pure retarded types.
Therefore, the analysis of strong stability for mixed retarded-neutral type systems
poses a hard problem which requires bringing in essential ideas in addition.

The method presented in this paper is based on a decomposition of the initial infinite-dimensional model $\dot{x}={\mathcal A}x$
onto two systems $\dot{x}_0={\mathcal A}_0x_0$ and $\dot{x}_1={\mathcal A}_1x_1$ in such a way that
the spectra of ${\mathcal A}_0$, ${\mathcal A}_1$ satisfy: ${\rm Re}\:\sigma({\mathcal A}_0)\le -\varepsilon$ and
$ -\varepsilon<{\rm Re}\:\sigma({\mathcal A}_1)<0$ for some $\varepsilon>0$.
Generally speaking, the operators ${\mathcal A}_0$ and ${\mathcal A}_1$ are not the models of delay systems (retarded or neutral),
what, in particular, implies that the relation between their exponential growth and spectrum's location is unknown a priori.
We prove the exponential stability of the operator ${\mathcal A}_0$ analyzing the boundedness of its resolvent.
This direct analysis requires subtle estimates and the proof is technically complicated.
For the analysis of the subsystem $\dot{x}_1={\mathcal A}_1x_1$ we apply methods of strong stability
introduced in \cite{Rabah_Sklyar_Rezounenko_2003,Rabah_Sklyar_Rezounenko_2005}.
Finally, the introduced approach allows us to prove for mixed retarded-neutral type systems
the results on strong stability formulated in \cite{Rabah_Sklyar_Rezounenko_2005}.

Besides, for control systems the proposed approach allows to analyze
the notion of {\it regular} asymptotic stabilizability \cite{Rabah_Sklyar_Rezounenko_2008}
which is closely related to the strong stability notion.
The technic of the regular asymptotic stabilizability were introduced in \cite{Rabah_Sklyar_Rezounenko_2008}
and the sufficient condition for the system~(\ref{syst_delay_gen_kt})--(\ref{syst_delay_gen_operators_k}) to be stabilizable had been proved
in the case $\det A_{-1}\not=0$. In the present paper, using the same framework as for stability,
we show that these results hold for mixed retarded-neutral type systems also.

The general framework which we use is the theory of $C_0$-semigroups of linear bounded operators (see e.g. \cite{Neerven_1996}).
In order to precise the main contribution of our paper let us first give the operator model of the
system~(\ref{syst_delay_gen_kt})--(\ref{syst_delay_gen_operators_k}). We use the model introduced by Burns et al.
\cite{Burns_Herdman_Stech_1983} in a Hilbert state space. The state operator is given by
\begin{equation}\label{syst_operator}
{\cal A}x(t)={\cal A} \left(
\begin{array}{c}
y(t)\\
z_t(\cdot)
\end{array}
\right) = \left(
\begin{array}{c}
\int_{-1}^0 A_2(\theta)\dot{z}_t(\theta) \dd\theta +\int_{-1}^0 A_3(\theta)z_t(\theta) \dd\theta \\
\dd z_t(\theta) / \dd \theta
\end{array}
\right),
\end{equation}
with the domain
\begin{equation}\label{syst_domain}
{\cal D}({\cal A})=\{(y, z(\cdot))^T: \; z\in H^1(-1, 0; \mathbb{C}^n), y=z(0)-A_{-1}z(-1)\}\subset M_2,
\end{equation}
where $M_2 {\stackrel{\rm def}{=}}\mathbb{C}^n\times L_2(-1, 0; \mathbb{C}^n)$ is the state space.
The operator ${\cal A}$ is the infinitesimal generator of a $C_0$-semigroup.
Studying stability problem, we consider the model
\begin{equation}\label{syst_delay_opmodel}
\dot{x}={\cal A}x, \qquad x(t)=\left(
\begin{array}{c}
y(t)\\
z_t(\cdot)
\end{array}
\right),
\end{equation}
corresponding to the equation~(\ref{syst_delay_gen_kt})--(\ref{syst_delay_gen_operators_k}) with the control $u\equiv 0$, i.e.
to the equation
\begin{equation}\label{syst_delay_gen}
\dot{z}(t)=A_{-1}\dot{z}(t-1)+\int_{-1}^0
A_2(\theta)\dot{z}(t+\theta)\;{\rm d}\theta +\int_{-1}^0
A_3(\theta){z}(t+\theta) \;{\rm d}\theta, \qquad t \ge 0.
\end{equation}
The solutions of (\ref{syst_delay_gen}) and (\ref{syst_delay_opmodel}) are related as $z_t(\theta)=z(t+\theta)$, $\theta\in [-1,0]$.

For the control the equation~(\ref{syst_delay_gen_kt})--(\ref{syst_delay_gen_operators_k}) which we rewrite as
\begin{equation}\label{syst_delay_gen_cont}
\dot{z}(t)=A_{-1}\dot{z}(t-1)+\int_{-1}^0
A_2(\theta)\dot{z}(t+\theta)\;{\rm d}\theta +\int_{-1}^0
A_3(\theta){z}(t+\theta) \;{\rm d}\theta+Bu,
\end{equation}
we consider the model
\begin{equation}\label{syst_delay_opmodel_stab}
\dot{x}={\mathcal A}x + {\mathcal B}u,
\end{equation}
where the operator ${\mathcal B}: \mathbb{C}^p\rightarrow M_2$ is defined by $n\times p$-matrix $B$ as follows:
${\mathcal B}u{\stackrel{\rm def}{=}}\left(Bu,\; 0\right)^T$.

The operator $\mathcal A$ given by~(\ref{syst_operator})
possesses only discrete spectrum $\sigma({\mathcal A})$, and,
moreover, the growth of the semigroup $\{{\rm e}^{t{\mathcal
A}}\}_{t\ge 0}$ is determinated by spectrum's location. Namely,
denoting by $\omega_s=\sup\{{\rm Re} \lambda: \; \lambda\in
\sigma({\mathcal A})\}$ and by $\omega_0=\inf\{\omega:\; \|{\rm
e}^{{\mathcal A}t}x\|\le M {\rm e}^{\omega t} \|x\|\}$, we have the
relation $\omega_0=\omega_s$ (see e.g. \cite{Hale_Verduyn_1993}).

For stability problem, the last fact implies that the semigroup $\{{\rm e}^{t{\mathcal
A}}\}_{t\ge 0}$ is exponentially stable if and only if the spectrum
of $\mathcal A$ satisfies $\omega_s<0$. However,  this type of
stability is not the only possible one  for systems of the
form~(\ref{syst_delay_gen}) (e.g. the same
situation can also happen for some hyperbolic partial differential
equation). Namely, if $\omega_s=0$ (and $A_{-1}\not=0$), then there
exists a sequence of eigenvalues with real parts approaching to zero
and imaginary part tending to infinity. In this critical case the
exponential stability is not possible: $\|{\rm e}^{t{\mathcal
A}}\|\not\rightarrow 0$ when $t\rightarrow \infty$, but asymptotic
non-exponential stability may occur: $\lim\limits_{t\rightarrow
+\infty} {\rm e}^{t{\mathcal A}}x=0$ for all $x\in M_2$.
For systems~(\ref{syst_delay_opmodel}), satisfying the assumption $\det A_{-1}\not=0$,
the problem of strong stability was analyzed in \cite{Rabah_Sklyar_Rezounenko_2003,Rabah_Sklyar_Rezounenko_2005}.
The main result on stability obtained there may be formulated as follows.
\begin{thr}[{\cite[R.~Rabah,~G.M.~Sklyar,~A.V.~Rezounenko]{Rabah_Sklyar_Rezounenko_2005}}]\label{thr_stability_intro}
Consider the system~(\ref{syst_delay_opmodel}) such that $\det A_{-1}\not=0$.
Let us put $\sigma_1=\sigma(A_{-1})\cap \{\mu: |\mu|= 1 \}$.
Assume that $\sigma({\cal A})\subset \{\lambda:\; \MRe \lambda<0\}$ (necessary condition).
The following three mutually exclusive possibilities hold true:

{\em (i)} $\sigma_1$ consists of simple eigenvalues only, i.e. an one-dimensional eigenspace corresponds to each eigenvalue
and there are no root vectors. Then system~(\ref{syst_delay_opmodel}) is asymptotically stable.

{\em (ii)} The matrix $A_{-1}$ has a Jordan block, corresponding to an eigenvalue $\mu\in \sigma_1$.
Then (\ref{syst_delay_opmodel}) is unstable.

{\em (ii)} There are no Jordan blocks, corresponding to eigenvalues from $\sigma_1$, but there exists
an eigenvalue $\mu\in \sigma_1$ whose eigenspace is at least two
dimensional. In this case system~(\ref{syst_delay_opmodel}) can be
either stable or unstable. Moreover, there exist two systems with
the same spectrum, such that one of them is stable while the other
one is unstable.
\end{thr}

Let us discuss the importance of the assumption $\det A_{-1}\not=0$.
The proof of Theorem~\ref{thr_stability_intro} given in \cite{Rabah_Sklyar_Rezounenko_2005} is
based on the following facts.
Firstly, if $\det A_{-1}\not=0$, then the spectrum of $\mathcal A$ is located in a vertical strip $d_1\le{\rm Re}\; \sigma(\mathcal A)\le d_2$.
Namely, in \cite{Rabah_Sklyar_Rezounenko_2003,Rabah_Sklyar_Rezounenko_2005} it was shown that
$\sigma({\mathcal A})=\{\ln |\mu_m|+ {\rm i}(\arg \mu_m +2\pi k)+\overline{\rm o}(1/k):\; \mu_m\in\sigma(A_{-1}), k\in\mathbb{Z}\}$.
From the last it also follows the necessary condition for the system to be asymptotically stable: $\sigma(A_{-1})\subset\{\mu:\; |\mu|\le 1\}$.

Secondly, such location of the spectrum had allowed to prove the existence of a Riesz basis of generalized eigenvectors for the operator
${\mathcal A}=\widetilde{{\mathcal A}}$ corresponding to the case $A_2(\theta)\equiv A_3(\theta)\equiv 0$.
For a general operator ${\mathcal A}$ the generalized eigenvectors may not constitute a basis of the state space
(see an example in \cite{Rabah_Sklyar_Rezounenko_2003} and some general conditions in \cite{Verduyn_Yakubovich_1997}).
However, in \cite{Rabah_Sklyar_Rezounenko_2003,Rabah_Sklyar_Rezounenko_2005}
it was proved the existence of a Riesz basis of $\mathcal A$-invariant finite-dimensional subspaces of the space $M_2$ (see also \cite{Vlasov_2003}).
Such a basis is a powerful tool that had been applied for the analysis of strong stability.

If we allow the matrix $A_{-1}$ to be singular, then the described above location of the spectrum does not hold anymore.
Generally speaking, in this case for any $\alpha\in\mathbb{R}$ there exists an infinite number of eigenvalues which are situated
on the left of the vertical line ${\rm Re}\lambda=\alpha$.
Thus, the existence of a Riesz basis of ${\mathcal A}$-invariant finite-dimensional subspaces for the whole space $M_2$ cannot be guarantied.
As a consequence, the proof of the item~(i) given in \cite{Rabah_Sklyar_Rezounenko_2005},
which is essentially based on the Riesz basis technic, is no longer satisfactory and one needs another way of the analysis of stability.

However, it can be asserted that nonzero $\mu_m\in\sigma(A_{-1})$ define the spectral set
$\{\ln |\mu_m|+ {\rm i}(\arg \mu_m +2\pi k)+\overline{\rm o}(1/k):\; \mu_m\in\sigma(A_{-1}), \mu_m\not=0, k\in\mathbb{Z}\}\subset \sigma({\mathcal A})$
which belongs to a vertical strip. In particular, this can be asserted for $\mu_m\in\sigma_1$.
The fact that Theorem~\ref{thr_stability_intro} is formulated in terms of $\sigma_1$ and the last remark
give us the idea to decompose the initial system~(\ref{syst_delay_opmodel}) into two systems
\begin{equation}\label{decomp_intro}
\dot{x}={\mathcal A}x \Leftrightarrow
\left\{
\begin{array}{l}
\dot{x_0}={\mathcal A}_0 x_0\\
\dot{x_1}={\mathcal A}_1 x_1
\end{array}
\right.
\end{equation}
in such way that $\sigma({\mathcal A}_0)=\sigma({\mathcal A})\cap \{\lambda:\;-\infty<{\rm Re}\lambda\le -\varepsilon\}$ and
$\sigma({\mathcal A}_1)=\sigma({\mathcal A})\cap \{\lambda:\;-\varepsilon<{\rm Re}\lambda\le \omega_s=0\}$, for some $\varepsilon>0$.

To obtain the representation~(\ref{decomp_intro}) we construct a special spectral decomposition of the state space:
$M_2= M_2^0 \oplus M_2^1$, where $M_2^0$, $M_2^1$ are ${\mathcal A}$-invariant subspaces.
We define ${\mathcal A}_0={\mathcal A}|_{M_2^0}$ and ${\mathcal A}_1={\mathcal A}|_{M_2^1}$.

The spectrum of the system $\dot{x_1}={\mathcal A}_1 x_1$ is such that
the corresponding eigenvectors form a Riesz basis of the subspace $M_2^1$.
The strong stability of the semigroup $\{{\rm e}^{t {\mathcal A}}|_{M_2^1}\}_{t\ge 0}$ is being proved
using the methods of \cite{Rabah_Sklyar_Rezounenko_2005}.

The semigroup $\{{\rm e}^{t {\mathcal A}}|_{M_2^0}\}_{t\ge 0}$ is exponentially stable. We prove this
fact using the equivalent condition consisting in the uniform boundedness of the resolvent $R(\lambda, {\mathcal A})|_{M_2^0}$
on the set $\{\lambda: {\rm Re} \lambda \ge 0\}$.
Thus, we prove that the initial system $\dot{x}={\mathcal A}x$ is asymptotically stable.
The mentioned scheme requires complicated technics.

To complete the stability analysis we revisit the example showing the item~(iii) with a simpler formulation
than in \cite{Rabah_Sklyar_Rezounenko_2005}, where it was given using the Riesz basis technic.
The analysis of the spectrum being carried out in our example is essentially based on the deep results on
transcendental equations obtained by L.~Pontryagin \cite{Pontryagin_function_zeros}.
We notice also that the proof of the item~(ii) given in \cite{Rabah_Sklyar_Rezounenko_2005}
does not involve the Riesz basis technic and, thus, it remains the same for the case $\det A_{-1}=0$ .

The technics of the direct spectral decompositions and the resolvent boundedness presented above
allow us to extend the results on the stabilizability problem given in \cite{Rabah_Sklyar_Rezounenko_2008} for the case of
singular matrix $A_{-1}$.

The general problem of stabilizability of control system is to find a feedback $u={\mathcal F}x$
such that the closed-loop system
$$
\dot{x}=({\mathcal A}+{\mathcal B}{\mathcal F})x
$$
is asymptotically stable in some sense. For the system~(\ref{syst_delay_gen_cont})
the result of exponential stabilizability may be derived
from those obtained for some particular cases (see e.g. \cite{Hale_Verduyn_2002,OConnor_Tarn_1983,Pandolfi_1976}).
The needed feedback for our system is of the form
\begin{equation}\label{eq_control2_intro}
F(z(t+\cdot))=F_{-1} \dot{z}(t-1)
+ \int_{-1}^0 F_2(\theta)\dot{z}(t+\theta) \dd\theta +\int_{-1}^0 F_3(\theta)z(t+\theta) \dd\theta.
\end{equation}
Our purpose is to obtain, as in \cite{Rabah_Sklyar_Rezounenko_2008}, the condition of asymptotic non-exponential
stabilizability of the system~(\ref{syst_delay_gen_cont}) with the regular feedback
\begin{equation}\label{eq_control_intro}
F(z(t+\cdot))=\int_{-1}^0 F_2(\theta)\dot{z}(t+\theta) \dd\theta +\int_{-1}^0 F_3(\theta)z(t+\theta) \dd\theta,
\end{equation}
where $F_2(\cdot), F_3(\cdot)\in L_2(-1, 0; \mathbb{C}^{n\times p})$.
The motivation is that this kind of feedback is relatively bounded with respect to the state operator $\mathcal A$
and does not change the domain of $\mathcal A$: ${\mathcal D}({\mathcal A})={\mathcal D}({\mathcal A}+{\mathcal B}{\mathcal F})$.
The natural necessary condition regular stabilizability is $\sigma(A_{-1})\subset\{\mu:\; |\mu|\le 1\}$
because $A_{-1}$ is not modified by the feedback.
Under the same restrictive condition $\det A_{-1}\not=0$ in \cite{Rabah_Sklyar_Rezounenko_2008}
was obtained the following result on stabilizability.
\begin{thr}[{\cite[R.~Rabah,~G.M.~Sklyar,~A.V.~Rezounenko]{Rabah_Sklyar_Rezounenko_2008}}]\label{thr_stabilizability_intro}
Let the system~(\ref{syst_delay_gen_cont}) verifies the following assumptions:

{\em (1)}  All the eigenvalues of the matrix $A_{-1}$ satisfy $|\mu|\le 1$.

{\em (2)} All the eigenvalues $\mu\in \sigma_1$ are simple.

\noindent Then the system is regularly asymptotically stabilizable if

{\em (3)} ${\rm rank}(\triangle(\lambda),\; B)=n$ for all $\lambda: {\rm Re} \lambda \ge 0$.

{\em (4)} ${\rm rank}(\mu I - A_{-1},\; B)=n$ for all $\mu\in\sigma_1$.
\end{thr}

The proof of this theorem given in \cite{Rabah_Sklyar_Rezounenko_2008} uses the existence of the
Riesz basis of the whole state space $M_2$ and, thus, it requires the assumption $\det A_{-1}\not=0$.
To avoid this assumption, we construct and prove another spectral decomposition
which takes into account the unstable part of the system.
By means of this decomposition we separate a subsystem which is generated
by the part of the spectrum corresponding to the zero eigenvalues, i.e. the singularities of the matrix $A_{-1}$.
Proving the resolvent boundedness, we show the exponential stability of this subsystem.
The main ``critical'' part of the system is in $\mathcal A$-invariant subspaces, where we apply the same methods that
were given in \cite{Rabah_Sklyar_Rezounenko_2008}, namely, the theorem on
infinite pole assignment, introduced there, and a classical pole assignment result in finite dimensional spaces.

The paper is organized as follows. In Sections~2 we recall the results on the spectrum, eigenvectors and the resolvent of the operator $\mathcal A$
obtained in \cite{Rabah_Sklyar_Rezounenko_2005,Rabah_Sklyar_Rezounenko_2008}. Besides we prove some properties
of eigenvectors.
In Section~3 we construct and prove two direct spectral decomposition of the state space.
One of them is used to prove the main result on stability and another one for the proof the result on stabilizability.
Section~4 is devoted to the proof of the uniform boundedness of the restriction of the resolvent on some invariant subspaces.
Finally, in Section~5 and Section~6 we give the formulation and the proof of our main results on
stability and stabilizability.
Besides, in Section~5 we give an explicit example of two systems having the same spectrum in the open left half-plane
but one of these systems is asymptotically stable while the other one is unstable.
\section{Preliminaries}
In this section we recall several results on the location of the spectrum of the operator $\mathcal A$,
on the explicit form of its resolvent and on the form of eigenvectors of $\mathcal A$ and ${\mathcal A}^*$.
We prove some properties of eigenvectors of $\mathcal A$ and ${\mathcal A}^*$.

\subsection{The resolvent and the spectrum}

The results given in this subsection have been presented and proved in
\cite{Rabah_Sklyar_Rezounenko_2003,Rabah_Sklyar_Rezounenko_2005,Rabah_Sklyar_Rezounenko_2008}.
Some formulations of the propositions are adapted for the  case $\det A_{-1}=0$.

\begin{ut}[{\cite[Proposition~1]{Rabah_Sklyar_Rezounenko_2005}}]\label{ut_ut001}
The resolvent of the operator ${\cal A}$ has the following form:
\begin{equation}\label{resolvent}
R(\lambda, {\mathcal A}) \left(
\begin{array}{c}
z\\
\xi(\cdot)
\end{array}
\right) \equiv \left(
\begin{array}{c}
{\rm e}^{-\lambda}A_{-1}\int\nolimits_{-1}^0 {\rm e}^{-\lambda s}
\xi(s){\rm d} s + (I-{\rm e}^{-\lambda}A_{-1})
\triangle^{-1}(\lambda)
D(z, \xi, \lambda)\\[2ex]
\int\nolimits_0^\theta {\rm e}^{\lambda(\theta- s)} \xi(s) {\rm d} s
+ {\rm e}^{\lambda\theta}\triangle^{-1}(\lambda) D(z, \xi, \lambda)
\end{array}
\right),
\end{equation}
where $z\in\mathbb{C}^n, \xi(\cdot)\in L_2(-1,0; \mathbb{C}^n)$;
$\triangle(\lambda)$ is the matrix function defined by
\begin{equation}\label{eq_delta}
\triangle(\lambda)=\triangle_{\mathcal A}(\lambda)=-\lambda
I+\lambda {\rm e}^{-\lambda} A_{-1} + \lambda\int\nolimits_{-1}^0
{\rm e}^{\lambda s} A_2(s) {\rm d} s + \int\nolimits_{-1}^0 {\rm
e}^{\lambda s} A_3(s) {\rm d} s,
\end{equation}
and $D(z, \xi, \lambda)$ is the following vector-function acting to $\mathbb{C}^n$:
\begin{equation}\label{eq_D}
D(z, \xi, \lambda)=z+ {\lambda}
{\rm e}^{-\lambda}A_{-1}\int_{-1}^0 {\rm e}^{-{\lambda} \theta} \xi(\theta) \dd
\theta -\int_{-1}^0 A_2(\theta) \xi(\theta) \dd \theta$$
$$- \int_{-1}^0 {\rm e}^{{\lambda} \theta} [\lambda A_2(\theta)+A_3(\theta)]
\left[\int_{0}^\theta {\rm e}^{-{\lambda} s}\xi(s)\dd s \right] \dd
\theta.
\end{equation}
\end{ut}

From (\ref{resolvent}) one may see that the resolvent does not exist in the points of singularity of
the matrix $\triangle(\lambda)$, i.e. the equation $\det \triangle(\lambda)=0$ defines the eigenvalues of the operator $\mathcal A$.
Now let us characterize the spectrum of ${\mathcal A}$ more precisely.

We denote by $\mu_1,\ldots, \mu_\ell$ the set of distinct
eigenvalues of the matrix $A_{-1}$ and by $p_1,\ldots, p_\ell$ their
multiplicities.
We recall the notation $\sigma_1=\sigma(A_{-1})\cap \{\mu: |\mu|= 1
\}$ and assume that $ \sigma_1= \{\mu_1,\ldots, \mu_{\ell_1}\}$,
$\ell_1\le \ell$. We notice that one of the eigenvalues $\mu_{\ell_1+1},\ldots, \mu_\ell$ may be zero.

Further, studying stability and stabilizability problems, we consider mainly the situations when
the eigenvalues from $\sigma_1$ are simple. This gives us a motivation to assume below (if the opposite is not mentioned) that
$p_1=\ldots=p_{\ell_1}=1$.
Besides, without loss of generality, we assume that the matrix $A_{-1}$ is in the following Jordan form:
\begin{equation}\label{eq_eq34}
A_{-1}= \left(
\begin{array}{cccccc}
\mu_1 & \ldots & 0 & 0 & \ldots & 0\\
\vdots & \ddots & \vdots & \vdots & \ddots  & \vdots \\
0 & \ldots & \mu_{\ell_1} & 0 & \ldots & 0\\
0 & \ldots & 0 & J_{\ell_1+1} & \ldots & 0\\
\vdots & \ddots & \vdots & \vdots & \ddots  & \vdots \\
0 & \ldots & 0 & 0 & \ldots & J_{\ell}\\
\end{array}
\right),
\end{equation}
where $J_{\ell_1+1},\ldots,J_{\ell}$ are Jordan blocks corresponding to the eigenvalues
$\mu_{\ell_1+1},\ldots,\mu_{\ell}$.

Let us denote by $\widetilde{{\mathcal A}}$ the state operator in the case when
$A_2(\theta)\equiv A_3(\theta)\equiv 0$. It is not difficult to see that the
spectrum of $\widetilde{{\mathcal A}}$ has the following structure
$$\sigma(\widetilde{\mathcal A})=\{\widetilde{\lambda}_m^{k}=\ln |\mu_m|+ {\rm i}(\arg \mu_m +2\pi k):\;
m=1,\ldots,\ell, \mu_m\not=0, k\in\mathbb{Z}\}\cup \{0\}.$$
We denote by $L_m^{k}(r^{(k)})$ circles centered at $\widetilde{\lambda}_m^{k}$ with radii $r^{(k)}$.

\begin{ut}\label{ut_ut1}
Let $\sigma_1=\{\mu_1,\ldots,\mu_{\ell_1}\}$ consists of simple eigenvalues only.
There exists $N_1\in\mathbb{N}$ such that the total multiplicity of the roots of the equation $\det \triangle(\lambda)=0$, contained in
the circles $L_m^{k}(r^{(k)})$, equals $p_m=1$ for all $m=1,\ldots,\ell_1$ and $k: |k|\ge N_1$,
and the radii $r^{(k)}$ satisfy the relation
$\sum\limits_{k\in\mathbb{Z}}(r^{(k)})^2 < \infty$.
\end{ut}
This proposition is a particular case of \cite[Theorem~4]{Rabah_Sklyar_Rezounenko_2008} which have been
formulated and proved there under the assumption $\det A_{-1}\not=0$. The proof for the case $\det A_{-1}=0$ remains the same.

\begin{notat}\label{notat_notat1}
We denote the eigenvalues of ${\mathcal A}$ mentioned in Proposition~\ref{ut_ut1} by $\lambda_m^k$, $m=1,\ldots,\ell_1$, $|k|\ge N_1$.
\end{notat}

\begin{rem}
Proposition~\ref{ut_ut1} is formulated for $m=1,\ldots,\ell_1$, however, it also holds for all those indices $m=1,\ldots,\ell$
which correspond to nonzero eigenvalues $\mu_m\in\sigma(A_{-1})$.
\end{rem}

\begin{rem}
In the case $\det A_{-1}\not=0$ the spectrum of $\mathcal A$ belongs to a vertical strip which is bounded from the left and from the
right.
However, in the case $\det A_{-1}=0$, in addition to the eigenvalues mentioned in Proposition~\ref{ut_ut1},
the operator $\mathcal A$ may also possess an infinite sequence of eigenvalues with real parts tending to $-\infty$.
\end{rem}

Similar results hold for the operator ${\mathcal A}^*$.
The spectra of ${\mathcal A}$ and ${\mathcal A}^*$ are related as $\sigma({\mathcal A}^*)=\overline{\sigma({\mathcal A})}$.
Eigenvalues of ${\mathcal A}^*$ are the roots
of the equation $\det \triangle^*(\lambda)=0$, where
\begin{equation}\label{eq_delta_conj}
\triangle^*(\lambda)=\triangle_{\cal A^*}(\lambda)=-\lambda
I+\lambda {\rm e}^{-\lambda} A_{-1}^* + \lambda\int_{-1}^0 {\rm e}^{\lambda s}
A_2^*(s) \dd s + \int_{-1}^0 {\rm e}^{\lambda s} A_3^*(s) \dd s,
\end{equation}
and the relation $(\triangle(\lambda))^*=\triangle^*(\overline{\lambda})$ holds.
The eigenvalues $\overline{\lambda_m^k}$, $m=1,\ldots,\ell_1$, $|k|\ge N_1$ may be described as in Proposition~\ref{ut_ut1}.

\subsection{Eigenvectors of $\mathcal A$ and ${\mathcal A}^*$}

First we give the explicit form of eigenvectors which has been proved in \cite{Rabah_Sklyar_Rezounenko_2005,Rabah_Sklyar_Rezounenko_2008}.

\begin{ut}[{\cite[Theorem~2]{Rabah_Sklyar_Rezounenko_2005}, \cite[Theorem~7]{Rabah_Sklyar_Rezounenko_2008}}]\label{ut_ut2}
Eigenvectors $\varphi$: $({\mathcal A}-\lambda I)\varphi=0$ and
$\psi$: $({\mathcal A}^*-\overline{\lambda} I)\psi=0$, $\lambda\in\sigma(\mathcal A)$ are of the form:
\begin{equation}\label{eq_eq27}
\varphi=\varphi(\lambda)=\left(
\begin{array}{c}
(I-{\rm e}^{-{\lambda}}A_{-1})x\\
{\rm e}^{{\lambda} \theta}x
\end{array}
\right),
\end{equation}
\begin{equation}\label{eq_eq28}
\psi=\psi(\overline{\lambda})=\left(
\begin{array}{c}
y\\
\left[\overline{\lambda} {\rm e}^{-\overline{\lambda} \theta}-A_2^*(\theta)+
{\rm e}^{-{\overline{\lambda}} \theta}\int\limits_0^\theta {\rm e}^{{\overline{\lambda}} s} (A_3^*(s)
+ \overline{\lambda} A_2^*(s)){\rm d} s
\right] y
\end{array}
\right),
\end{equation}
where $x=x(\lambda)\in \Ker\triangle({\lambda})$, $y=y(\overline{\lambda})\in \Ker\triangle^*({\overline{\lambda}})$.
\end{ut}

Below we give several properties of the sets of eigenvectors and we begin with the calculation of the
scalar product between eigenvectors of ${\mathcal A}$ and ${\mathcal A}^*$.

\begin{lem}\label{lm_lm4}
Let $\lambda_0, \lambda_1\in\sigma({\mathcal A})$ and $\varphi=\varphi(\lambda_0)$,
$\psi=\psi(\overline{\lambda_1})$ are corresponding eigenvectors:
$({\cal A}-{\lambda_0} I)\varphi=0$, $({\cal A}^*-\overline{\lambda_1} I)\psi=0$.
The scalar product $\langle \varphi, \psi \rangle_{M_2}$
equals to the following value:
\begin{equation}\label{eq_eq26}
\langle \varphi, \psi\rangle_{M_2}=\left\{
\begin{array}{cl}
0, & \lambda_0\not=\lambda_1\\
-\langle \triangle'({\lambda_0})x, y\rangle_{\mathbb{C}^n}, & \lambda_0=\lambda_1
\end{array}
\right.,
\end{equation}
where  $\triangle'({\lambda})=\frac{{\rm d}}{{\rm d}\lambda}\triangle(\lambda)$ and
$x=x(\lambda_0)$, $y=y(\overline{\lambda_1})$ are defined by (\ref{eq_eq27}) and
(\ref{eq_eq28}).
\end{lem}

\begin{proof} First let $\lambda_0\not=\lambda_1$ and we compute directly the
scalar product $\langle \varphi, \psi \rangle_{M_2}$ using
the representations (\ref{eq_eq27}) and (\ref{eq_eq28}):
\begin{equation}\label{eq_eq36}
\begin{array}{rcl}
\langle \varphi, \psi\rangle_{M_2} & = & \langle
(I-{\rm e}^{-{\lambda_0}}A_{-1})x, y \rangle_{\mathbb{C}^n} + \int\limits_{-1}^0 \langle {\rm e}^{{\lambda_0} \theta}x,
{\overline{\lambda_1}} {\rm e}^{-{\overline{\lambda_1}} \theta}y \rangle_{\mathbb{C}^n} \dd \theta - \int\limits_{-1}^0 \langle
{\rm e}^{{\lambda_0} \theta}x, A_2^*(\theta)y \rangle_{\mathbb{C}^n} \dd \theta \\
& &  + \int\limits_{-1}^0 \left\langle {\rm e}^{{\lambda_0} \theta}x,
{\rm e}^{-{\overline{\lambda_1}} \theta}\int\limits_0^\theta
{\rm e}^{{\overline{\lambda_1}} s} (A_3^*(s)+ \overline{\lambda_1} A_2^*(s)) \dd s\cdot y \right\rangle_{\mathbb{C}^n} \dd \theta\\
& = & \left\langle (I-{\rm e}^{-{\lambda_0}}A_{-1})x, y \right\rangle
+ \left\langle \int\limits_{-1}^0 \lambda_1 {\rm e}^{({\lambda_0}-\lambda_1) \theta} \dd \theta\cdot x, y \right\rangle
- \left\langle \int\limits_{-1}^0 {\rm e}^{{\lambda_0} \theta}A_2(\theta) \dd \theta \cdot x, y \right\rangle \\
& &  + \left\langle \int\limits_{-1}^0 {\rm e}^{({\lambda_0}-\lambda_1) \theta}
\int\limits_0^\theta {\rm e}^{\lambda_1 s} (A_3(s) + \lambda_1 A_2(s)) \dd s \dd \theta \cdot x,\; y\right\rangle \\
& = & \left\langle \Gamma({\lambda_0}, \lambda_1) x,\; y \right\rangle,
\end{array}
\end{equation}
where
\begin{equation}\label{eq_eq37}
\begin{array}{rcl}
\Gamma({\lambda_0}, \lambda_1) & = & I-{\rm e}^{-{\lambda_0}}A_{-1}+
\lambda_1 \int\limits_{-1}^0 {\rm e}^{({\lambda_0}-\lambda_1) \theta} \dd
\theta - \int\limits_{-1}^0 {\rm e}^{{\lambda_0} \theta}A_2(\theta) \dd \theta \\
& & + \int\limits_{-1}^0 {\rm e}^{({\lambda_0}-\lambda_1) \theta} \int\limits_0^\theta
{\rm e}^{\lambda_1 s} (A_3(s) + \lambda_1 A_2(s)) \dd s \dd \theta.
\end{array}
\end{equation}
The last term of $\Gamma({\lambda_0}, {\lambda_1})$, which is the integral over the domain $-1\le\theta\le s\le 0$,
we rewrite using the identity $\int_{-1}^0 \int_0^\theta G(s,\theta) \dd s  \dd \theta = -\int_{-1}^0 \int_{-1}^s G(s,\theta) \dd \theta \dd s$
which holds for any function $G(s,\theta)$.
Taking into account the relation
$\int_{-1}^0 {\rm e}^{({\lambda_0}-\lambda_1) \theta} \dd \theta
=\frac{1}{{\lambda_0}-\lambda_1}(1-{\rm e}^{\lambda_1-{\lambda_0}})$,
we obtain
$$
\begin{array}{c}
\int\limits_{-1}^0 {\rm e}^{({\lambda_0}-\lambda_1) \theta} \int\limits_0^\theta
{\rm e}^{\lambda_1 s} (A_3(s) + \lambda_1 A_2(s)) \dd s \dd \theta= -\int\limits_{-1}^0
{\rm e}^{\lambda_1 s} (A_3(s) + \lambda_1 A_2(s)) \int\limits_{-1}^s {\rm e}^{({\lambda_0}-\lambda_1)
\theta} \dd \theta \dd s \\
=\frac{1}{{\lambda_0}-\lambda_1}\left[{\rm e}^{\lambda_1-{\lambda_0}}
\int\limits_{-1}^0 {\rm e}^{\lambda_1 s} (A_3(s) + \lambda_1 A_2(s)) \dd s - \int\limits_{-1}^0
{\rm e}^{{\lambda_0} s} (A_3(s) + \lambda_1 A_2(s)) \dd s\right].
\end{array}
$$
Finally, we have
$$
\begin{array}{rcl}
\Gamma({\lambda_0}, \lambda_1) & = & \frac{1}{{\lambda_0}-\lambda_1}\left[({\lambda_0}-\lambda_1)I-({\lambda_0}-\lambda_1)
{\rm e}^{-{\lambda_0}}A_{-1} +\lambda_1(1-{\rm e}^{\lambda_1-{\lambda_0}})I\right. \\
& & - ({\lambda_0}-\lambda_1)\int\limits_{-1}^0 {\rm e}^{{\lambda_0} \theta}A_2(\theta) \dd \theta-
\int\limits_{-1}^0 {\rm e}^{{\lambda_0} s} (A_3(s) + \lambda_1 A_2(s)) \dd s \\
& & \left. + {\rm e}^{\lambda_1-{\lambda_0}} \int\limits_{-1}^0 {\rm e}^{\lambda_1 s} (A_3(s) + \lambda_1 A_2(s)) \dd s \right]\\
& = & \frac{1}{{\lambda_0}-\lambda_1}\left[{\lambda_0} I - {\lambda_0} {\rm e}^{-{\lambda_0}}A_{-1}
-\int\limits_{-1}^0 {\rm e}^{{\lambda_0} s} (A_3(s) + {\lambda_0} A_2(s)) \dd s  \right. \\
& & \left. - \lambda_1 {\rm e}^{\lambda_1-{\lambda_0}}I + \lambda_1
{\rm e}^{-{\lambda_0}}A_{-1} +{\rm e}^{\lambda_1-{\lambda_0}}
\int\limits_{-1}^0 {\rm e}^{\lambda_1 s} (A_3(s) + \lambda_1 A_2(s)) \dd s \right] \\
& = & \frac{1}{{\lambda_0}-\lambda_1}\left[-\triangle({\lambda_0})
+ {\rm e}^{\lambda_1-{\lambda_0}}\triangle(\lambda_1)\right].
\end{array}
$$

Taking into account that $x\in \Ker\triangle({\lambda_0})$,
$y\in \Ker\triangle^*({\overline{\lambda_1}})$ and
$(\triangle(\lambda_1))^*=\triangle^*({\overline{\lambda_1}})$,
we conclude that
\begin{equation}\label{eq_eq29}
\langle \varphi, \psi \rangle_{M_2}= \left\langle
\frac{1}{{\lambda_0}-\lambda_1} \left[-\triangle({\lambda_0})
+ {\rm e}^{\lambda_1-{\lambda_0}}\triangle(\lambda_1)\right]x,
y\right\rangle_{\mathbb{C}^n}=
\frac{{\rm e}^{\lambda_1-{\lambda_0}}}{{\lambda_0}-\lambda_1}
\langle x, \triangle^*({\overline{\lambda_1}})
y\rangle_{\mathbb{C}^n}=0.
\end{equation}

\vskip2ex

Let us now consider the case $\lambda_0=\lambda_1$.
From (\ref{eq_eq36}), (\ref{eq_eq37}) we have:
$$
\langle \varphi, \psi\rangle_{M_2} =\langle \Gamma({\lambda_0})
x, y\rangle_{\mathbb{C}^n},
$$
where
\begin{equation}\label{eq_eq38}
\Gamma({\lambda_0})=I-{\rm e}^{-{\lambda_0}}A_{-1}+{\lambda_0} I -
\int_{-1}^0 {\rm e}^{{\lambda_0} \theta}A_2(\theta) \dd \theta
+ \int_{-1}^0 \int_0^\theta {\rm e}^{{{\lambda_0}} s} (A_3(s)+{\lambda_0}A_2(s)) \dd s \dd \theta.
\end{equation}

The last term of $\Gamma({\lambda_0})$, which is the integral over the
domain $-1\le\theta\le s\le 0$,
we rewrite using the identity $\int_{-1}^0 \int_0^\theta G(s,\theta) \dd s  \dd \theta = -\int_{-1}^0 \int_{-1}^s G(s,\theta) \dd \theta \dd s$.
Thus, we obtain:
$$
\begin{array}{rcl}
\Gamma({\lambda_0}) & = & I-{\rm e}^{-{\lambda_0}}A_{-1}+{\lambda_0} I -
\int\limits_{-1}^0 {\rm e}^{{\lambda_0} \theta}A_2(\theta) \dd \theta
- \int\limits_{-1}^0 {\rm e}^{{{\lambda_0}} s} (A_3(s)+{\lambda_0} A_2(s)) \int\limits_{-1}^s \dd \theta \dd s\\
& = & \left(I-{\rm e}^{-{\lambda_0}}A_{-1} - \int\limits_{-1}^0 {\rm e}^{{{\lambda_0}} s} (s A_3(s)+ s \lambda_0 A_2(s) + A_2(s))\dd s\right)\\
& & +\left({\lambda_0} I
- \int\limits_{-1}^0 {\rm e}^{{{\lambda_0}} s} (A_3(s)+\lambda_0 A_2(s)) \dd s \right)\\
& = & -\triangle'({\lambda_0})-\triangle({\lambda_0}).
\end{array}
$$
Taking into account the relation $x\in \Ker\triangle({\lambda_0})$, we conclude that
\begin{equation}\label{eq_eq25}
\langle \varphi, \psi\rangle_{M_2}=-\langle
\triangle'({\lambda_0})x, y\rangle_{\mathbb{C}^n}.
\end{equation}
The last completes the proof of the lemma.
\end{proof}

For $\varphi(\lambda_m^k)$ and $\psi(\overline{\lambda_m^k})$ we will use the notation $\varphi_m^k$ and $\psi_m^k$ respectively.
Besides, we use $x_m^k$ and $y_m^k$ instead of $x(\lambda_m^k)$ and $y(\overline{\lambda_m^k})$.

\begin{lem}\label{lm_lm333}
Let $\sigma_1=\{\mu_1,\ldots,\mu_{\ell_1}\}$ consists of simple eigenvalues only.
The eigenvectors $\varphi_m^k$, $m=1,\ldots,\ell_1$, $k: |k|\ge N_1$
constitute a Riesz basis of the closure of their linear span.
The same holds for eigenvectors $\psi_m^k$, $m=1,\ldots,\ell_1$, $k: |k|\ge N_1$.
\end{lem}
A more general formulation of this proposition have been given in \cite[Theorem~7,~Theorem~15]{Rabah_Sklyar_Rezounenko_2003}
under the assumption $\det A_{-1}\not=0$.
We give a sketch of the proof in our case.

The families of functions $\{{\rm e}^{\widetilde{\lambda}_m^k \theta}\}_{k\in\mathbb{Z}}$ form an orthogonal basis of
the space $L_2([-1,0], \mathbb{C})$ for each $m=1,\ldots,\ell_1$,
where $\widetilde{\lambda}_m^{k}={\rm i}(\arg \mu_m +2\pi k)$ are eigenvalues of the operator $\widetilde{\mathcal A}$.
Thus, the functions $\{{\rm e}^{\widetilde{\lambda}_m^k \theta}\}_{|k|\ge N}$, $N\in\mathbb{N}$ form a basis of
the closure of their linear span.

Since we have chosen the matrix $A_{-1}$ in the form~(\ref{eq_eq34}) and due to~(\ref{eq_eq27}), the eigenvectors
$\widetilde{\varphi}_m^k$ of $\widetilde{\mathcal A}$ are of the form
$\widetilde{\varphi}_m^k=\left(
\begin{array}{c}
0\\
{\rm e}^{\widetilde{\lambda}_m^{k} \theta}e_m
\end{array}
\right)$, $e_m=(0,\ldots,1,\ldots,0)^T$. Therefore, the family $\{\widetilde{\varphi}_m^k:\; m=1,\ldots,\ell_1: |k|\ge N_1\}$
is a basis of the closure of its linear span.

The eigenvectors $\varphi_m^k=\left(
\begin{array}{c}
(I-{\rm e}^{-{\lambda_m^k}}A_{-1})x_m^k\\
{\rm e}^{{\lambda_m^k} \theta}x_m^k
\end{array}
\right)$ of ${\mathcal A}$ are quadratically close to $\widetilde{\varphi}_m^k$.
To prove this fact we should argue similar to Theorem~15 given in \cite{Rabah_Sklyar_Rezounenko_2003} (see also \cite{Kato_1972}).
Thus, eigenvectors $\varphi_m^k$, $m=1,\ldots,\ell_1$, $k: |k|\ge N_1$
constitute a Riesz basis of the closure of their linear span.

\vskip1ex

From Lemma~\ref{lm_lm4} and Lemma~\ref{lm_lm333} we conclude the following.

\begin{col}\label{col_col23}
The sequences $\varphi_m^k$ and $\psi_m^k$, $m=1,\ldots,\ell_1$, $k: |k|\ge N_1$ are biorthogonal after
normalization and $\langle \varphi_m^k, \psi_m^k\rangle_{M_2}= -\langle \triangle'({\lambda_m^k})x_m^k, y_m^k\rangle_{\mathbb{C}^n}$.
\end{col}

The following relation will be essentially used in the analysis of the boundedness of the resolvent in Section~4.

\begin{lem}\label{lm_lm1_plelim}
Let $\psi=\psi(\overline{\lambda_0})$, $\lambda_0\in\sigma(\mathcal A)$ be an eigenvector of the operator $\mathcal A^*$
and let $g=(z, \xi(\cdot))\in M_2$ be orthogonal to $\psi$: $g\bot\psi$.
Then the following relation holds:
\begin{equation}\label{orthogonality_col}
D(z, \xi, \lambda_0) \in  {\rm Im}\triangle(\lambda_0),
\end{equation}
where $D(z, \xi, \lambda)$ is defined by (\ref{eq_D}).
\end{lem}

\begin{proof}
We show the relation $D(z, \xi, \lambda_0)\bot \Ker\triangle^*(\overline{\lambda_0})$
which is equivalent to~(\ref{orthogonality_col}).
The eigenvector $\psi$ is of the form~(\ref{eq_eq28}):
$$
\psi=\left(
\begin{array}{c}
y\\
\left[\overline{\lambda_0} {\rm e}^{-\overline{\lambda_0}
\theta}-A_2^*(\theta)+ {\rm e}^{-\overline{\lambda_0}
\theta}\int_0^\theta {\rm e}^{\overline{\lambda_0} s} A_3^*(s) \dd s +
\overline{\lambda_0} {\rm e}^{-\overline{\lambda_0}
\theta}\int_0^\theta {\rm e}^{\overline{\lambda_0} s} A_2^*(s) \dd s
\right]y
\end{array}
\right),
$$
where $y=y(\overline{\lambda_0})\in \Ker\triangle^*(\overline{\lambda_0})$.
For any $g=(z, \xi(\cdot))$, which is orthogonal to $\psi$, we obtain:
\begin{equation}\label{eq_eq1234}
\begin{array}{rcl}
0=\langle g, \psi\rangle_{M_2} & = & \left\langle z,\; y\right\rangle_{\mathbb{C}^n} + \int\limits_{-1}^0
\left\langle \xi(\theta), \overline{\lambda_0}
{\rm e}^{-\overline{\lambda_0} \theta}y\right\rangle_{\mathbb{C}^n} \dd \theta -
\int\limits_{-1}^0 \left\langle \xi(\theta),\; A_2^*(\theta)y \right\rangle_{\mathbb{C}^n} \dd \theta \\
& & + \int\limits_{-1}^0 \left\langle \xi(\theta),\; {\rm e}^{-\overline{\lambda_0} \theta}
\int\limits_0^\theta {\rm e}^{\overline{\lambda_0} s} (A_3^*(s) + \overline{\lambda_0} A_2^*(s)) \dd s\cdot y\right\rangle_{\mathbb{C}^n} \dd \theta \\
& = & \left\langle z, y\right\rangle_{\mathbb{C}^n} + \left\langle \int\limits_{-1}^0 \lambda_0
{\rm e}^{-\lambda_0 \theta} \xi(\theta) \dd \theta, y\right\rangle_{\mathbb{C}^n} -
\left\langle \int\limits_{-1}^0 A_2(\theta)\xi(\theta) \dd \theta, y \right\rangle_{\mathbb{C}^n} \\
& & + \left\langle \int\limits_{-1}^0 \left[ {\rm e}^{-\lambda_0 \theta}
\int\limits_0^\theta {\rm e}^{\lambda_0 s} (A_3(s)+ \lambda_0 A_2(s)) \dd s \right] \xi(\theta) \dd \theta,\; y\right\rangle_{\mathbb{C}^n}.
\end{array}
\end{equation}
Using the identity $\int_{-1}^0 \int_0^\theta G(s,\theta) \dd s\dd \theta =
-\int_{-1}^0 \int_{-1}^s G(s,\theta) \dd \theta \dd s$ which holds for any function $G(s,\theta)$, we rewrite the the last term of~(\ref{eq_eq1234}),
and, finally, we obtain the relation:
\begin{equation}\label{eq_eq1}
\begin{array}{rcl}
0=\langle g,\; \psi\rangle_{M_2} & = & \left\langle z+ \int\limits_{-1}^0 \lambda_0
{\rm e}^{-\lambda_0 \theta} \xi(\theta) \dd \theta +\int\limits_{-1}^0
A_2(\theta)\xi(\theta) \dd \theta \right.\\
& & \left. -\int\limits_{-1}^0 {\rm e}^{\lambda_0 s} \left[A_3(s)+\lambda_0 A_2(s)\right] \int\limits_{-1}^s
{\rm e}^{-\lambda_0 \theta} \xi(\theta)\dd \theta \dd s,\; y\right\rangle_{\mathbb{C}^n}.
\end{array}
\end{equation}
Since  $y\in \Ker\triangle^*(\overline{\lambda_0})$, then for
any $x\in \mathbb{C}^n$ we have:
$$
0=\langle x, \triangle^*(\overline{\lambda_0})y \rangle_{\mathbb{C}^n} =
\langle \triangle(\lambda_0) x,\; y \rangle_{\mathbb{C}^n}.
$$
Therefore, for any $\theta$ the relation $\langle {\rm e}^{-\lambda_0 \theta} \triangle(\lambda_0) \xi(\theta), y \rangle_{\mathbb{C}^n}=0$ holds,
and, integrating it by $\theta$ from $-1$ to $0$, we obtain:
\begin{equation}\label{eq_eq2}
\begin{array}{rcl}
0=\left\langle \int\limits_{-1}^0 {\rm e}^{-\lambda_0 \theta} \triangle(\lambda_0) \xi(\theta) \dd \theta,\; y \right\rangle & = &
\left\langle -\int\limits_{-1}^0 \lambda_0 {\rm e}^{-\lambda_0 \theta}
\xi(\theta) \dd \theta + \int\limits_{-1}^0 \lambda_0 {\rm e}^{-\lambda_0
\theta} A_{-1} \xi(\theta) \dd \theta \right.\\
& & \left. + \int\limits_{-1}^0  {\rm e}^{\lambda_0 s} \left[A_3(s) + \lambda_0 A_2(s)\right] \dd s \int\limits_{-1}^0
{\rm e}^{-\lambda_0 \theta} \xi(\theta) \dd \theta,\; y \right\rangle_{\mathbb{C}^n}.\\
\end{array}
\end{equation}
Let us sum up the left-hand sides and the right-hand sides of the relations (\ref{eq_eq2}) and (\ref{eq_eq1}).
In the obtained relation the term $\int_{-1}^0 \lambda_0 {\rm e}^{-\lambda_0 \theta} \xi(\theta) \dd \theta$ is cancelled.
The last terms of (\ref{eq_eq2}) and (\ref{eq_eq1}) we sum up according to the identity
$-\int_{-1}^0 \int_{-1}^s G(s,\theta) \dd \theta \dd s + \int_{-1}^0 \int_{-1}^0 G(s,\theta) \dd \theta \dd s
= -\int_{-1}^0 \int_0^s G(s,\theta) \dd \theta \dd s = -\int_{-1}^0 \int_0^\theta G(\theta, s) \dd s \dd \theta$
which holds true for any function $G(s,\theta)$.
Finally, we obtain:
$$
\begin{array}{rcl}
0 & = & \left\langle z+ \lambda_0 {\rm e}^{-\lambda_0}A_{-1}\int\limits_{-1}^0
{\rm e}^{-\lambda_0 \theta} \xi(\theta) \dd \theta -\int\limits_{-1}^0
A_2(\theta) \xi(\theta) \dd \theta \right.\\
& & \left.- \int\limits_{-1}^0 {\rm e}^{\lambda_0 \theta} \left[A_3(\theta) + \lambda_0 A_2(\theta)\right]
\left[\int\limits_{0}^\theta {\rm e}^{-\lambda_0 s}\xi(s)\dd s \right] \dd \theta,\; y \right\rangle_{\mathbb{C}^n}\\
& \equiv & \left\langle D(z, \xi, \lambda_0),\; y \right\rangle_{\mathbb{C}^n}.
\end{array}
$$

Since $y\in \Ker\triangle^*(\overline{\lambda_0})$,
we conclude that $D(z, \xi, \lambda_0)\bot \Ker\triangle^*(\overline{\lambda_0})$,
what completes the proof of the lemma.
\end{proof}

\begin{rem}
We emphasize the fact that $\det \triangle(\lambda_0)=0$ and, therefore,
the matrix $\triangle^{-1}(\lambda_0)$ does not exist.
However, the proved relation $D(z, \xi, \lambda_0) \in  {\rm Im}\triangle(\lambda_0)$ means
that there exists the inverse image of the vector $D(z, \xi, \lambda_0)$ with respect to the matrix $\triangle(\lambda_0)$.
\end{rem}
\section{Spectral decompositions of the state space}\label{sect_decomposition}

We recall that we consider the operator $\mathcal A$ in the case when
all eigenvalues from $\sigma_1\subset\sigma(A_{-1})$ are simple.
In this section we construct construct two spectral decompositions of the state space $M_2$.
Assuming that $\sigma(\mathcal A)\subset \{\lambda:\; {\rm Re} \lambda <0\}$, in the first subsection
we construct a decomposition which we further us in Section~5 for the stability analysis.
In the second subsection we assume only $|\mu|\le 1$ for all $\mu\in \sigma(A_{-1})$
(i.e. a part of the spectrum of $\mathcal A$ may belongs to the closed right half-plane) and construct a decomposition
needed in Section~6 for the stabilizability analysis.
The structures of these decomposition are very similar.
In the third subsection we prove some technical results used in the proofs of validity of the decompositions.

\subsection{Spectral decomposition for the stability problem}

For the stability analysis our aim is to divide the system onto exponentially stable part and
strongly asymptotically stable part. To do this we construct a decomposition of the state space $M_2$ onto the direct sum of two
${\cal A}$-invariant subspaces and prove its validity.

We divide the spectrum of ${\mathcal A}$ onto two parts. For some $N\ge N_1$ we define
\begin{equation}\label{eq_eq69}
\Lambda_1=\Lambda_1(N)=\{\lambda_m^k\in \sigma({\mathcal A}), \; m=1,\ldots,\ell_1, |k|\ge N\},
\end{equation}
and represent the spectrum as follows:
$$\sigma({\mathcal A})=\Lambda_0\cup\Lambda_1.$$

\begin{rem} The set $\Lambda_1$ is determined by $N\in \mathbb{N}$.
For any small $\varepsilon>0$ there exists big enough $N$ such that
$\Lambda_1$ belongs to the vertical strip $\{\lambda:\; -\varepsilon <{\rm Re} \lambda<0\}$.
\end{rem}
The following figure illustrates our idea:
\begin{center}
\includegraphics[scale=0.7]{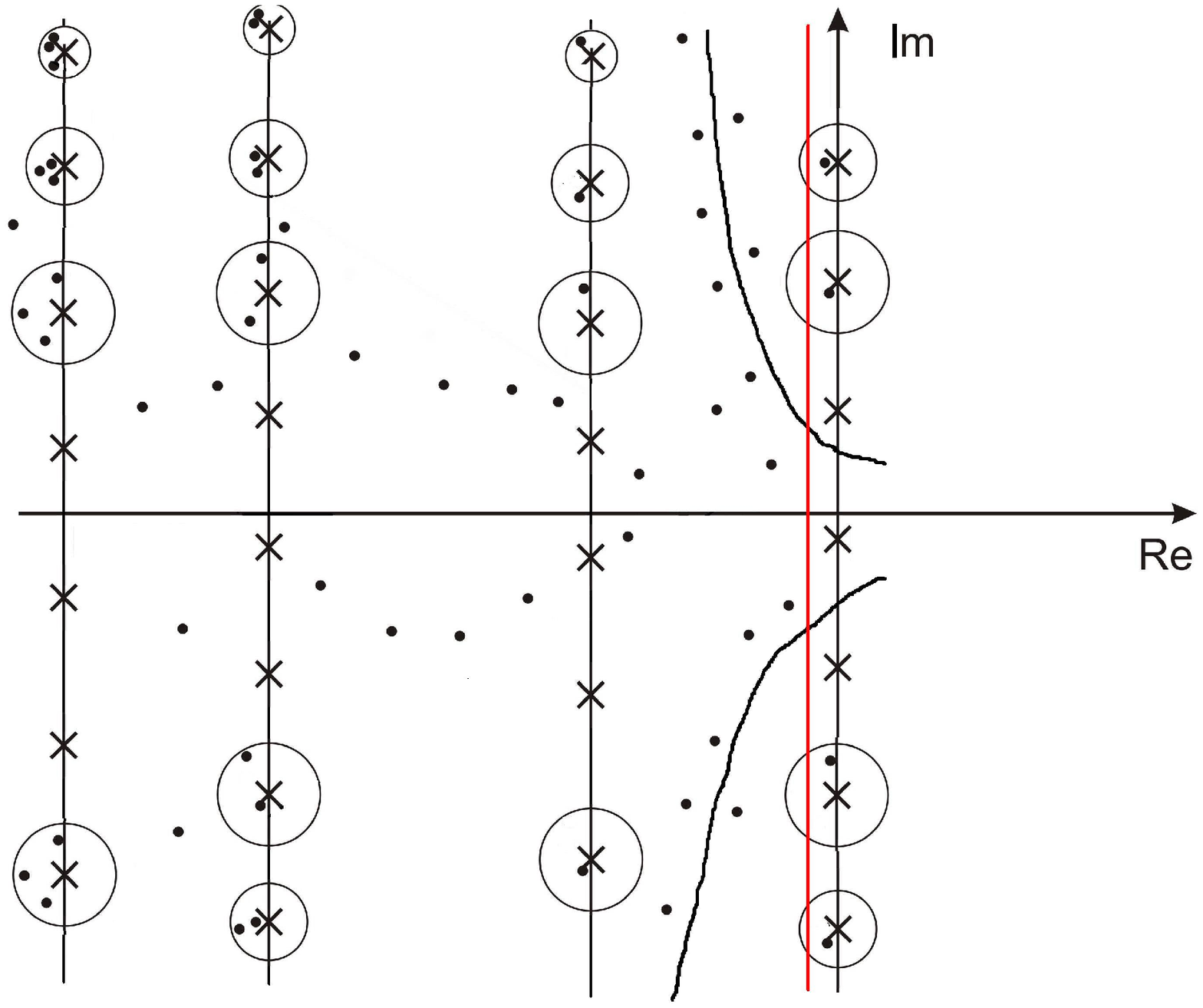}
\vskip 1ex
Figure 1.
\end{center}
By crosses we denote $\widetilde{\lambda}_m^{k}$, eigenvalues of $\widetilde{\mathcal A}$ and by points
we denote eigenvalues of the operator ${\mathcal A}$.

We introduce two subspaces of $M_2$:
\begin{equation}\label{eq_eq32}
M_2^1=M_2^1(N)={{\rm Cl}\; {\rm Lin}\{\varphi:\; ({\mathcal A}-\lambda I)\varphi=0,\; \lambda\in\Lambda_1\}},
\end{equation}
\begin{equation}\label{eq_eq31}
\widehat{M}_2^1=\widehat{M}_2^1(N)={{\rm Cl}\; {\rm Lin}\{\psi:\; ({\mathcal A}^*-\overline{\lambda} I)\psi=0,\;
\lambda\in\Lambda_1\}}.
\end{equation}
Obviously, $M_2^1$ is ${\mathcal A}$-invariant and $\widehat{M}_2^1$ is ${\mathcal A}^*$-invariant.
We introduce $M_2^0=M_2^0(N)$ which satisfies
\begin{equation}\label{eq_eq23}
M_2= \widehat{M}_2^1 {\stackrel{\bot}{\oplus}} M_2^0.
\end{equation}
Due to the construction, $M_2^0$ is an ${\mathcal A}$-invariant subspace.

\begin{rem}
We recall that due to Lemma~\ref{lm_lm333} eigenvectors $\{\varphi_m^k\}$ of $\mathcal A$, corresponding to $\lambda_m^k\in \Lambda_1$,
form a Riesz basis of the closure of their linear span. The same holds for eigenvectors
$\{\psi_m^k\}$ of $\mathcal A^*$, corresponding to $\overline{\lambda_m^k}$, $\lambda_m^k\in {\Lambda_1}$.
\end{rem}

The main result of this subsection is the following theorem.

\begin{thr}[on direct decomposition]\label{thr_decomposition}
Let $\sigma_1=\{\mu_1,\ldots,\mu_{\ell_1}\}$ consists of simple eigenvalues only.
For any $N\ge N_1$ the subset $\Lambda_1=\Lambda_1(N)\subset\sigma({\mathcal A})$ given by (\ref{eq_eq69})
and the subspaces $M_2^0$, $M_2^1$, $\widehat{M}_2^1$, given by (\ref{eq_eq32}),
(\ref{eq_eq31}) and (\ref{eq_eq23}) define the direct decomposition of the space:
\begin{equation}\label{eq_eq3933}
M_2= M_2^1 \oplus M_2^0
\end{equation}
where the subspaces $M_2^1$, $M_2^0$ are ${\cal A}$-invariant.
\end{thr}

\begin{proof} To prove (\ref{eq_eq3933}) we show that
any element $\xi\in M_2$ allows the following representation:
\begin{equation}\label{eq_eq39}
\xi=\xi_0+\sum_{m=1}^{\ell_1}\sum_{|k|\ge N} c_m^k \varphi_m^k, \quad
\xi_0\in M_2^0,\; \varphi_m^k\in M_2^1,\; \sum_{m=1}^{\ell_1}\sum_{|k|\ge N} |c_m^k|^2<\infty.
\end{equation}

As we have noticed above, the eigenvectors
$\{\varphi_m^k:\; ({\cal A}-{\lambda_m^k} I)\varphi_m^k=0,\;
\lambda_m^k\in\Lambda_1\}$ given by (\ref{eq_eq27}) form a Riesz basis of the
closure of their linear span.
Besides the eigenvectors
$\{\frac{1}{{\overline{\lambda_m^k}}}\psi_m^k:\; ({\cal
A}^*-{\overline{\lambda_m^k}} I)\psi_m^k=0,\;
\lambda_m^k\in\Lambda_1\}$ form a Riesz basis of the
closure of their linear span, where eigenvectors
$\psi_m^k$ are given by (\ref{eq_eq28}).
We use the notation $\widehat{\psi}_m^k=\frac{1}{\overline{\lambda_m^k}}\psi_m^k$.

Using the representations (\ref{eq_eq27}) and (\ref{eq_eq28}), we choose eigenvectors with $\|x_m^k\|_{\mathbb{C}^n}=1$ and
$\|y_m^k\|_{\mathbb{C}^n}=1$.
Due to Lemma~\ref{lm_lm5} given in Subsection~\ref{subsect_33}
there exists $C>0$ such that $\|\varphi_m^k\|_{M^2}\le
C$ and
$\|\widehat{\psi}_m^k\|_{M^2}\le C$ for all $m=1,\ldots,\ell_1$, $|k|\ge N$.

Applying the decomposition~(\ref{eq_eq23}) to vectors $\varphi_m^k$, we
obtain
$$
\varphi_m^k=\gamma_m^k + \sum_{i=1}^{\ell_1}\sum_{|j|\ge N} a_i^j
\widehat{\psi}_i^j, \quad \gamma_m^k\in M_2^0.
$$
Since $\langle \varphi_m^k, \widehat{\psi}_i^j\rangle=0$ for $(m,k)\not=(i,j)$ (Corollary~\ref{col_col23}),
the last representation may be rewritten as follows:
\begin{equation}\label{eq_eq43}
\varphi_m^k=\gamma_m^k+a_m^k \widehat{\psi}_m^k, \quad \gamma_m^k\in
M_2^0,
\end{equation}
moreover, due to (\ref{eq_eq26}) we have the relation
\begin{equation}\label{eq_eq44}
a_m^k= \frac{\langle \varphi_m^k, \widehat{\psi}_m^k\rangle_{M_2}}
{\|\widehat{\psi}_m^k\|^2_{M_2}} =
\frac{\frac{1}{\lambda_m^k}\langle \varphi_m^k, \psi_m^k
\rangle_{M_2}} {\|\widehat{\psi}_m^k\|^2_{M_2}}
=\frac{-\frac{1}{{\lambda_m^k}} \langle
\Delta'({\lambda_m^k})x_m^k, y_m^k\rangle_{\mathbb{C}^n}} {
\|\widehat{\psi}_m^k\|^2_{M_2}}.
\end{equation}
From (\ref{eq_eq43}) and (\ref{eq_eq44}) it also follows that
\begin{equation}\label{eq_eq46}
\|\gamma_m^k\|\le \|\varphi_m^k\|+|a_m^k| \|\widehat{\psi}_m^k\| \le
C+\sqrt{|\frac{1}{{\lambda_m^k}} \langle
\Delta'({\lambda_m^k})x_m^k, y_m^k\rangle|}.
\end{equation}

Using the decomposition~(\ref{eq_eq23}) and the relation (\ref{eq_eq43}),
we represent each vector $\xi\in M_2$ as follows:
\begin{equation}\label{eq_eq47}
\begin{array}{rcl}
\xi & = & \widehat{\xi}_0+\sum\limits_{m=1}^{\ell_1}\sum\limits_{|k|\ge N} b_m^k
\widehat{\psi}_m^k =\widehat{\xi}_0 - \sum\limits_{m=1}^{\ell_1}\sum\limits_{|k|\ge N} \frac{b_m^k}{a_m^k} \gamma_m^k
+ \sum\limits_{m=1}^{\ell_1}\sum\limits_{|k|\ge N} \frac{b_m^k}{a_m^k}  \varphi_m^k \\
& = & \xi_0+\sum\limits_{m=1}^{\ell_1}\sum\limits_{|k|\ge N} c_m^k \varphi_m^k,
\end{array}
\end{equation}
where $\widehat{\xi}_0\in M_2^0$, $\sum\limits_{m=1}^{\ell_1}\sum\limits_{|k|\ge N} |b_m^k|^2<\infty$,
$\xi_0=\widehat{\xi}_0 - \sum\limits_{m=1}^{\ell_1}\sum\limits_{|k|\ge N} \frac{b_m^k}{a_m^k} \gamma_m^k \in M_2^0$,
$c_m^k=\frac{b_m^k}{a_m^k}$.

To prove the validity of the decomposition~(\ref{eq_eq47}) it is
enough to show that $\left| \frac{1}{a_m^k}\right|\le C_1$ and
$\|\gamma_m^k\|\le C_2$ for some $0<C_1, C_2<+\infty$.
Taking into account (\ref{eq_eq44}) and (\ref{eq_eq46}), the last means to give the estimate
\begin{equation}\label{eq_eq48}
0<C_1\le \left|\frac{1}{{\lambda_m^k}} \langle
\Delta'({\lambda_m^k})x_m^k, y_m^k\rangle\right|\le C_2, \qquad
{\lambda_m^k}\in \Lambda_1.
\end{equation}
This estimate is proved by Lemma~\ref{lm_lm7} given in Subsection~\ref{subsect_33}.

Thus, the representation~(\ref{eq_eq47}) holds for any $\xi\in
M_2$, what completes the proof of the theorem.
\end{proof}

\subsection{Spectral decomposition for the stabilizability problem}\label{subsect_decomp_stabilizability}

We recall that for the stabilizability problem we assume only that $|\mu|\le 1$ for all $\mu\in \sigma(A_{-1})$.
In this case an infinite number of eigenvalues may belong to the right-half plane.
On the other hand, only a finite number of eigenvalues may be located on the right of
a vertical line ${\rm Re}\lambda = \varepsilon$ for any $\varepsilon>0$. For the analysis of stabilizability
it is convenient to construct a decomposition of the state space onto three ${\mathcal A}$-invariant subspaces.

We divide the spectrum of $\mathcal{A}$ onto three parts:
\begin{equation}\label{eq_eq22-2}
\sigma({\mathcal A})=\Lambda_0\cup \Lambda_1\cup \Lambda_2,
\end{equation}
where the subsets $\Lambda_0$, $\Lambda_1$, $\Lambda_2$ are constructed by the following procedure.

Let $N_0$ be such that $\lambda_m^k\in L_m^{k}(r)$, $r\le \frac13 |\widetilde{\lambda}_m^{k}-\widetilde{\lambda}_i^{j}|$,
$(m,k)\not=(i,j)$ for all $k\ge N_0$ and for all $m$ such that $\mu_m\not=0$.
First, we construct an auxiliary division
$$\sigma({\mathcal A})=\chi_1\cup \chi_0,\qquad \chi_1=\{\lambda_m^k\in \sigma({\mathcal A}): \; |k|\ge N_0, m=1,\ldots, \ell, \mu_m\not=0\}.$$

Due to the construction, any vertical strip ${\rm St}(\delta_1, \delta_2)=\{\lambda:\; \delta_1< {\rm Re} \lambda < \delta_2\}$ contains
only a finite number of eigenvalues from $\chi_0$.
We also recall that $\omega_s=\sup\{{\rm Re} \lambda: \; \lambda\in
\sigma({\mathcal A})\}<+\infty$.

If $\sigma_1\not=\emptyset$ then for any $r>0$ the strip ${\rm St}(-r, r)$ contains an infinite number of eigenvalues from  $\chi_1$
and, as we have noticed above, only a finite number of eigenvalues from $\chi_0$.
Let us fix some $r>0$ and consider the value
$$\varepsilon=\min\limits_{\lambda\in {\rm St}(-r, r)\cap \chi_0} |{\rm Re} \lambda|.$$

If $\varepsilon>0$, then the vertical strip ${\rm St}(-\varepsilon, \varepsilon)$ does not contain eigenvalues from
$\chi_0$ and contains an infinite number of eigenvalues from  $\chi_1$.
Moreover, the strip ${\rm St}(\varepsilon, r)$ contains only a finite number of eigenvalues from  $\chi_1$.
Thus, the strip ${\rm St}(\varepsilon, \omega_s)$ contains a finite number of eigenvalues of the operator $\mathcal A$ and,
therefore, we conclude that these eigenvalues are located in
a rectangle $\{\lambda: \; \varepsilon\le {\rm Re} \lambda\le \omega_0,\; |{\rm Im} \lambda|<M\}$ for some $M>0$.
Finally, we put
\begin{equation}\label{eq_eq22-3}
\begin{array}{l}
\Lambda_0= \sigma({\mathcal A}) \cap \{\lambda: \; {\rm Re} \lambda\le -\varepsilon \},\\
\Lambda_1= \sigma({\mathcal A}) \cap {\rm St}(-\varepsilon, \varepsilon),\\
\Lambda_2= \sigma({\mathcal A}) \cap {\rm St}(\varepsilon, \omega_s).
\end{array}
\end{equation}

We illustrate the mentioned construction by the following figure:
\begin{center}
\includegraphics[scale=0.7]{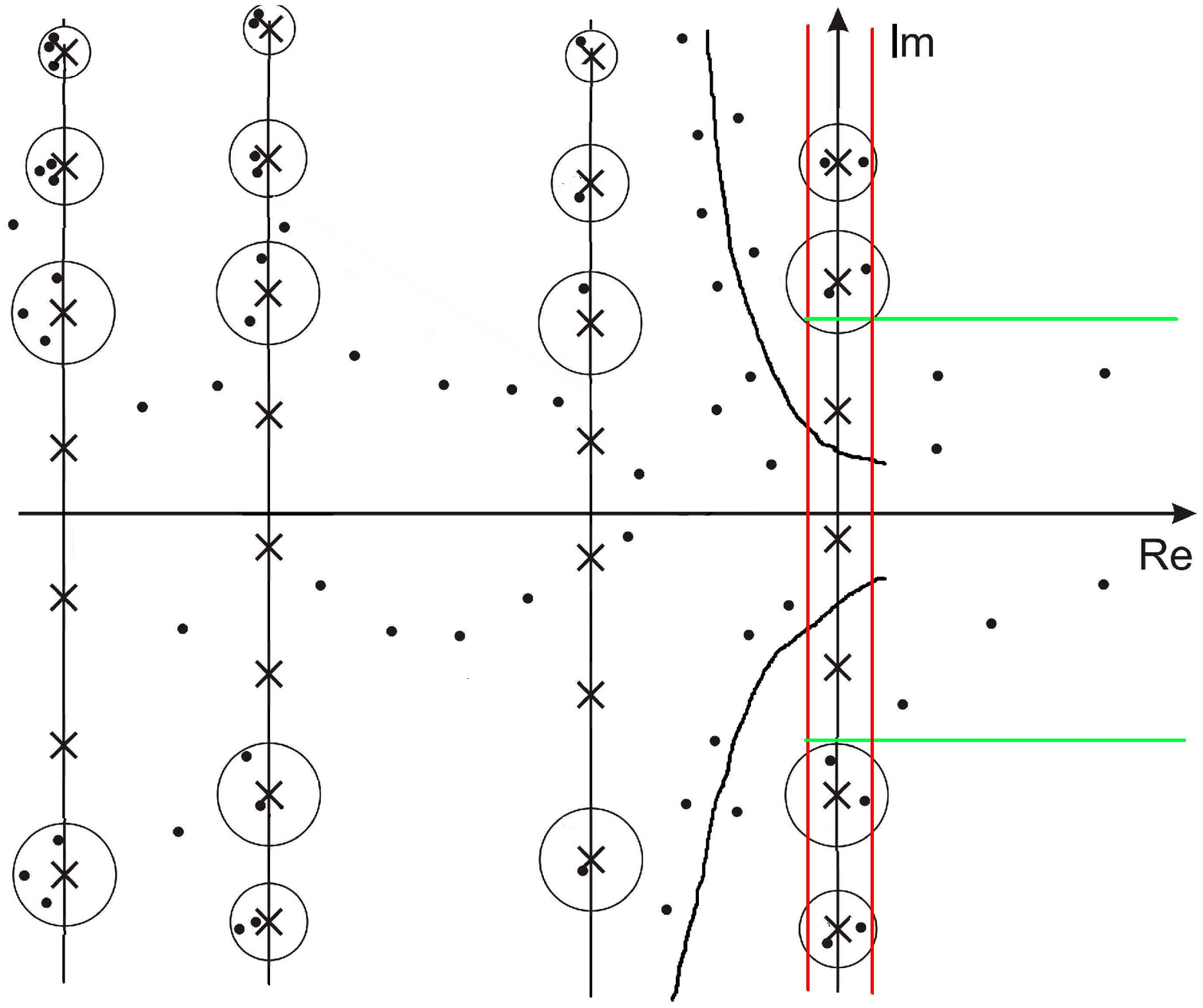}
\vskip 1ex
Figure 2.
\end{center}
Again by crosses we denote $\widetilde{\lambda}_m^{k}$, eigenvalues of $\widetilde{\mathcal A}$ and by points
we denote eigenvalues of the operator ${\mathcal A}$.

We notice that the relation $\varepsilon=0$ means that there exists eigenvalues with zero real part.
In these case we calculate $\min\limits_{\lambda\in {\rm St}(-r, r)\cap \chi_0} |{\rm Re} \lambda|$
without taking these eigenvalues into consideration and after constructing (\ref{eq_eq22-3})
we add these eigenvalues to $\Lambda_2$.

The obtained sets of eigenvalues may be described as follows:
$\Lambda_0$ belongs to the left half-plane and is separated from the imaginary axis;
\noindent $\Lambda_1$ consists of infinite number of simple eigenvalues which may be as stable as unstable,
the corresponding eigenvectors form a Riesz basis of the closure of their linear span;
$\Lambda_2$ consists of finite number of unstable eigenvalues.

Passing over to the construction of invariant subspaces,  let us denote the elements of the finite set $\Lambda_2$ as $\lambda_i$, $i=1,\ldots,r$.
We denote the corresponding generalized eigenvectors  by $\varphi_{i,j}$: $({\mathcal A}-\lambda_i I)^j\varphi_{i,j}=0$, $j=0,\ldots,s_i-1$.
As before, the eigenvalues from $\Lambda_1$ we denote as $\lambda_m^k$ and the corresponding eigenvectors
we denote as $\varphi_m^k$, $m=1,\ldots,\ell_1$, $|k|\ge N$.

We introduce the following two infinite-dimensional subspaces of eigenvectors:
\begin{equation}\label{eq_eq32-2}
M_2^{1}={{\rm Cl}\; {\rm Lin}\{\varphi_m^k:\; ({\mathcal A}-\lambda_m^k I)\varphi_m^k=0,\; \lambda_m^k\in\Lambda_1\}},
$$
$$
\widehat{M}_2^{1}={{\rm Cl}\; {\rm Lin}\{\psi_m^k:\; ({\mathcal A}^*-\overline{\lambda_m^k} I)\psi_m^k=0,\;
\lambda_m^k\in\Lambda_1\}},
\end{equation}
two finite-dimensional subspaces of eigenvectors and root-vectors:
\begin{equation}\label{eq_eq33-2}
M_2^{2}={\; {\rm Lin}\{\varphi_{i,j}:\; ({\mathcal A}-\lambda_i I)^j\varphi_{i,j}=0,\; \lambda_i\in\Lambda_2, \; j=0,\ldots, s_i-1\}},
$$
$$
\widehat{M}_2^{2}={\; {\rm Lin}\{\psi_{i,j}:\; ({\mathcal A}^*-\overline{\lambda_i} I)^j\psi_{i,j}=0,\;
\lambda_i\in\Lambda_2, \; j=0,\ldots, s_i-1\}}
\end{equation}
and the subspace $M_2^0$, which satisfies
\begin{equation}\label{eq_eq20-2}
M_2=(\widehat{M}_2^1 \oplus \widehat{M}_2^2) {\stackrel{\bot}{\oplus}} M_2^0.
\end{equation}
Thus, we have constructed three ${\mathcal A}$-invariant subspaces: $M_2^0$, $M_2^{1}$ and $M_2^{2}$.
The main result of this subsection is the following theorem.

\begin{thr}\label{thr_decomposition_stab}
Let $\sigma_1=\{\mu_1,\ldots,\mu_{\ell_1}\}$ consists of simple eigenvalues only.
For any $N\ge N_1$ the decomposition of the spectrum (\ref{eq_eq22-3})
and the subspaces given by (\ref{eq_eq32-2}), (\ref{eq_eq33-2}) and (\ref{eq_eq20-2})
define the direct decomposition of the space $M_2$:
$$
M_2= M_2^0 \oplus M_2^{1}\oplus M_2^{2},
$$
where the subspaces $M_2^0$, $M_2^{1}$, $M_2^{2}$ are ${\cal A}$-invariant.
\end{thr}

\begin{proof}
The proof of this proposition is very similar to the proof of Theorem~\ref{thr_decomposition}.
We prove that any element $\xi\in M_2$ allows the representation:
\begin{equation}\label{eq_eq39-st}
\xi=\xi_0+\sum\limits_{m=1}^{\ell_1}\sum\limits_{|k|\ge N} c_m^k \varphi_m^k + \sum_{i=1}^r\sum_{j=0}^{s_i-1} c_{i,j} \varphi_{i,j},
\quad \xi_0\in M_2^0,\; \sum\limits_{m=1}^{\ell_1}\sum\limits_{|k|\ge N} |c_m^k|^2<\infty.
\end{equation}

The eigenvectors $\{\varphi_m^k:\; \lambda_m^k\in \Lambda_1\}$ form a Riesz basis of the
closure of their linear span. The finite set of the generalized eigenvectors $\{\varphi_{i,j}:\; \lambda_i\in \Lambda_2\}$
is also a basis of their linear span. These vectors are linearly independent from eigenvectors $\{\varphi_m^k\}$.
Thus eigenvectors $\{\varphi_m^k\}\cup \{\varphi_{i,j}\}$ form a basis of the closure of their linear span.
Arguing the same way, we conclude that the set $\{\widehat{\psi}_m^k\}\cup \{\widehat{\psi}_{i,j}\}$ is also a basis of the closure of its linear span,
where $\widehat{\psi}_m^k=\frac{1}{\overline{\lambda_m^k}}\psi_m^k$ and $\widehat{\psi}_{i,j}=\frac{1}{\overline{\lambda_i}}\psi_{i,j}$.
Moreover, without loss of generality, we may assume that the sets $\{\varphi_m^k\}\cup \{\varphi_{i,j}\}$
and $\{\widehat{\psi}_m^k\}\cup \{\widehat{\psi}_{i,j}\}$ are biorthogonal after the normalization.

Using the representations (\ref{eq_eq27}) and (\ref{eq_eq28}), we choose eigenvectors with $\|x_m^k\|_{\mathbb{C}^n}=1$ and
$\|y_m^k\|_{\mathbb{C}^n}=1$.
Due to Lemma~\ref{lm_lm5} there exists $C>0$ such that $\|\varphi_m^k\|_{M^2}\le
C$ and
$\|\widehat{\psi}_m^k\|_{M^2}\le C$ for all $m=1,\ldots,\ell_1$, $|k|\ge N$.

Applying the decomposition~(\ref{eq_eq23}) to vectors $\varphi_m^k$, we
obtain
$$
\varphi_m^k=\gamma_m^k +
\sum\limits_{i=1}^{\ell_1}\sum\limits_{|j|\ge N} a_m^k \widehat{\psi}_m^k +
\sum_{i=1}^r\sum_{j=0}^{s_i-1} a_{i,j} \widehat{\psi}_{i,j},
 \quad \gamma_m^k \in M_2^0.
$$
Since the sets $\{\varphi\}$ and $\{\psi\}$ are biorthogonal, the last representation can be rewritten as follows:
\begin{equation}\label{eq_eq43-st}
\varphi_m^k=\gamma_m^k+a_m^k \widehat{\psi}_m^k, \quad \gamma_m^k\in M_2^0,
\end{equation}
and, moreover, due to Lemma~\ref{lm_lm4} we have
\begin{equation}\label{eq_eq44-st}
a_m^k= \frac{\langle \varphi_m^k, \widehat{\psi}_m^k\rangle_{M_2}}
{\|\widehat{\psi}_m^k\|^2_{M_2}} =
\frac{\frac{1}{\lambda_m^k}\langle \varphi_m^k, \psi_m^k
\rangle_{M_2}} {\|\widehat{\psi}_m^k\|^2_{M_2}}
=\frac{-\frac{1}{{\lambda_m^k}} \langle
\Delta'({\lambda_m^k})x_m^k, y_m^k \rangle_{\mathbb{C}^n}} {
\|\widehat{\psi}_m^k\|^2_{M_2}}.
\end{equation}
Besides, arguing the same we obtain:
$$
\varphi_{i,j}=\gamma_{i,j}+a_{i,j} \widehat{\psi}_{i,j_1}, \quad \gamma_{i,j}\in M_2^0, j_1=0,\ldots,s_i-1.
$$
From (\ref{eq_eq43-st}) and (\ref{eq_eq44-st}) it also follows that
\begin{equation}\label{eq_eq46-st}
\|\gamma_m^k\|\le \|\varphi_m^k\|+|a_m^k| \|\widehat{\psi}_m^k\| \le
C+\sqrt{|\frac{1}{{\lambda}} \langle
\Delta'({\lambda_m^k})x_m^k, y_m^k\rangle|}.
\end{equation}

Using the decomposition~(\ref{eq_eq23}) and the relation (\ref{eq_eq43-st}),
we represent each vector $\xi\in M_2$ as follows:
\begin{equation}\label{eq_eq47-st}
\xi=\widehat{\xi}_0+\sum\limits_{m=1}^{\ell_1}\sum\limits_{|k|\ge N} b_m^k \widehat{\psi}_m^k +
\sum_{i=1}^r\sum_{j=0}^{s_i-1} b_{i,j} \widehat{\psi}_{i,j}
$$
$$
=\widehat{\xi}_0 - \sum\limits_{m=1}^{\ell_1}\sum\limits_{|k|\ge N} \frac{b_m^k}{a_m^k} \gamma_m^k -
\sum_{i=1}^r\sum_{j=0}^{s_i-1} \frac{b_{i,j}}{a_{i,j}} \gamma_{i,j}
+  \sum\limits_{m=1}^{\ell_1}\sum\limits_{|k|\ge N} \frac{b_\lambda}{a_m^k}  \varphi_m^k
+ \sum_{i=1}^r\sum_{j=0}^{s_i-1} \frac{b_{i,j}}{a_{i,j}}  \varphi_{i,j}
$$
$$
=\xi_0+\sum\limits_{m=1}^{\ell_1}\sum\limits_{|k|\ge N} c_m^k \varphi_m^k + \sum_{i=1}^r\sum_{j=0}^{s_i-1} c_{i,j} \varphi_{i,j},
\end{equation}
where $\widehat{\xi}_0\in M_2^0$, $\sum\limits_{m=1}^{\ell_1}\sum\limits_{|k|\ge N} |b_m^k|^2<\infty$,
$\xi_0=\widehat{\xi}_0- \sum\limits_{m=1}^{\ell_1}\sum\limits_{|k|\ge N} \frac{b_m^k}{a_m^k} \gamma_m^k
- \sum\limits_{i=1}^r\sum\limits_{j=0}^{s_i-1} \frac{b_{i,j}}{a_{i,j}} \gamma_{i,j}\in M_2^0$,
$c_m^k=\frac{b_m^k}{a_m^k}$.

To prove the validity of the decomposition~(\ref{eq_eq47-st}) it is
enough to show that $\left| \frac{1}{a_m^k}\right|\le C_1$ and
$\|\gamma_m^k\|\le C_2$. Taking into account (\ref{eq_eq44-st}) and
(\ref{eq_eq46-st}), the last means to give the estimate
\begin{equation}\label{eq_eq48-st}
0<C_1\le \left|\frac{1}{\lambda_m^k} \langle
\Delta'(\lambda_m^k)x_m^k, y_m^k\rangle\right|\le C_2, \qquad
\lambda_m^k\in \Lambda_1,
\end{equation}
This estimate is proved by Lemma~\ref{lm_lm7}.
Therefore, the representation~(\ref{eq_eq47-st}) holds for any $\xi\in M_2$,
what completes the proof of the theorem.
\end{proof}

\subsection{Auxiliary results}\label{subsect_33}

In this subsection we prove several estimates which have been used in the proofs of Theorem~\ref{thr_decomposition}
and Theorem~\ref{thr_decomposition_stab}.

\begin{lem}\label{lm_lm5}
Let us consider eigenvectors $\varphi_m^k=\varphi(\lambda_m^k)$ and $\psi_m^k=\psi(\overline{\lambda_m^k})$
and their representations (\ref{eq_eq27}) and (\ref{eq_eq28}).
Let us assume that $\|x_m^k\|_{\mathbb{C}^n}=\|y_m^k\|_{\mathbb{C}^n}=1$ in these representations.
Then for any $N\in\mathbb{N}$ there exists a constant $C>0$ such that
\begin{equation}\label{eq_eq45}
\|\varphi_m^k\|\le C,\quad
\frac{1}{|\lambda_m^k|}\|\psi_m^k\|\le C, \qquad {m=1,\ldots,\ell_1,\; |k|\ge N}.
\end{equation}
In other words the two families of eigenvectors $\{\varphi_m^k:\; m=1,\ldots,\ell_1, |k|\ge N\}$ with $\|x_m^k\|_{\mathbb{C}^n}=1$
and $\left\{\frac{1}{{\overline{\lambda_m^k}}}\psi_m^k:\; m=1,\ldots,\ell_1, |k|\ge N\right\}$
with $\|y_m^k\|_{\mathbb{C}^n}=1$ are bounded.
\end{lem}

\begin{proof} Using (\ref{eq_eq27}) and the relation $\|x_m^k\|_{\mathbb{C}^n}=1$ we obtain:
$$
\begin{array}{rcl}
\|\varphi_m^k\|^2 & = & \|(I-{\rm e}^{-{\lambda_m^k}}A_{-1})x_m^k\|^2 +
\int\limits_{-1}^0 \|{\rm e}^{{\lambda_m^k} \theta} x_m^k\|^2 \dd \theta \\
& \le & \|I-{\rm e}^{-{\lambda_m^k}}A_{-1}\|^2 + \int\limits_{-1}^0 {\rm e}^{2\MRe{\lambda_m^k} \theta} \dd \theta\\
& \le & 1+ {\rm e}^{2\MRe{\lambda_m^k}}
\|A_{-1}\|^2+\frac{1-{\rm e}^{-2\MRe{\lambda_m^k}}}{2\MRe{\lambda_m^k}}
\le 1+\|A_{-1}\|^2+\frac{1-{\rm e}^{-2r}}{2r} \le C,
\end{array}
$$
where $r=\mathop{\max}\limits_{k\in\mathbb{N}} r^{(k)}$ and $r^{(k)}$ are the radii of the circles $L_m^{k}(r^{(k)})$.
The last inequality holds since the real function
$\frac{1-{\rm e}^{-y}}{y}$ decreases monotone from $+\infty$ to $1$ when
the variable $y$ runs from $\infty$ to $-0$.

From (\ref{eq_eq28}) and since $\|y_m^k\|_{\mathbb{C}^n}=1$, we have:
$$
\begin{array}{rcl}
\|\frac{1}{{\overline{\lambda_m^k}}} \psi_m^k\|^2 & = & \frac{1}{|\lambda_m^k|^2}\|y_m^k\|^2
 + \int\limits_{-1}^0 \left \|({\rm e}^{-{\overline{\lambda_m^k}} \theta}-\frac{1}{{\overline{\lambda_m^k}}}A_2^*(\theta)+ \right.\\
& &+
\left.\frac{1}{{\overline{\lambda_m^k}}} {\rm e}^{-{\overline{\lambda_m^k}}
\theta} \int\limits_0^\theta {\rm e}^{{\overline{\lambda_m^k}} s} A_3^*(s) \dd s
+ {\rm e}^{-{\overline{\lambda_m^k}} \theta}\int\limits_0^\theta
{\rm e}^{{\overline{\lambda_m^k}} s} A_2^*(s) \dd s) y_m^k\right\|^2\dd \theta\\
& \le &  \|y_m^k\|^2\left(\frac{1}{|\lambda_m^k|^2}+\frac{{\rm e}^{2\MRe \lambda_m^k}-1}{2\MRe \lambda_m^k}
+\frac{1}{|\lambda_m^k|^2}\|A_2^*(\theta)\|^2_{L_2}\right)\\
& & +\int\limits_{-1}^0 {\rm e}^{-2\MRe\lambda_m^k \theta} \dd\theta
\int\limits_{-1}^0 {\rm e}^{2\MRe\lambda_m^k s}\left(\frac{1}{|\lambda_m^k|^2} \|A_3^*(s)\| + \|A_2^*(s)\|\right)\dd s\\
& \le & \frac{1}{|\lambda_m^k|^2}+\frac{1}{|\lambda_m^k|^2}\|A_2^*(\theta)\|^2_{L_2} +\frac{{\rm e}^{2\MRe \lambda_m^k}-1}{2\MRe \lambda_m^k}
\left(1+\frac{1}{|\lambda_m^k|^2}\|A_3^*(\theta)\|^2_{L_2} +\|A_2^*(\theta)\|^2_{L_2}\right)\\
& \le &
\left(\frac{1}{|\lambda_m^k|^2}+1\right)(\|A_2^*(\theta)\|^2_{L_2}+1) +\frac{1}{|\lambda_m^k|^2}\|A_3^*(\theta)\|^2_{L_2}\le C,
\end{array}
$$
where $\|A_i^*(\theta)\|_{L_2}=\|A_i^*(\theta)\|_{L_2(-1,0; \mathbb{C}^{n\times n})}$.
Here we used the fact that the real function $\frac{{\rm e}^{y}-1}{y}$ increases monotone from $0$ to $1$
when the variable $y$ runs from $\infty$ to $-0$.
\end{proof}

\begin{rem}\label{rm_rm1}
We notice that the norm of eigenvectors $\psi_m^k$ (assuming $\|y_m^k\|=1$) increases infinitely when $k\rightarrow\infty$.
This could be seen on the example of eigenvectors $\widetilde{\psi}_m^k$ of
the operator $\widetilde{{\cal A}}^*$ ($A_2^*(\theta)=A_2^*(\theta)\equiv 0$):
$$
\|\widetilde{\psi}_m^k\|^2=\|y_m^k\|^2 +\int_{-1}^0
\|{\overline{\lambda_m^k}} {\rm e}^{-{\overline{\lambda_m^k}} \theta}
y_m^k\|^2 \dd \theta =\|y_m^k\|^2\left(1+|\lambda_m^k|^2
\frac{{\rm e}^{2\MRe \lambda_m^k}-1}{2\MRe \lambda_m^k}\right)
$$
$$
\ge (1+C|\lambda_m^k|^2)\rightarrow +\infty, \qquad k \rightarrow \infty.
$$
\end{rem}

To formulate the next proposition we introduce the matrices
$$
R_m=\left(
\begin{array}{cc}
\widehat{R}_m & 0\\
0 & I
\end{array}
\right),\quad \quad
\widehat{R}_m= \left(
\begin{array}{ccccc}
0 & 0 & \ldots & 0 & 1\\
0 & 1 & \ldots & 0 & 0\\
\vdots  & \vdots  & \ddots & \vdots  & \vdots \\
0 & 0 & \ldots & 1 & 0\\
1 & 0 & \ldots & 0 & 0
\end{array}
\right)\in \mathbb{C}^{m\times m}, \quad m=1,\ldots,\ell_1,
$$
where $I=I_{n-m}$ is the is the identity matrix of the dimension $n-m$.
Obviously, $R_1=I$ and $R_m^{-1}=R_m^*=R_m$ for all $m=1,\ldots,\ell_1$.

\begin{lem}\label{lm_lm6}
Assume that $\sigma_1=\{\mu_1,\ldots,\mu_{\ell_1}\}$ consists of simple eigenvalues only.
There exists $N\in\mathbb{N}$ such that for any ${\lambda_m^k}\in \Lambda_1$, $|k|\ge N$
and the corresponding matrix $\Delta({\lambda_m^k})$  there exist matrices
$P_{m,k}$, $Q_{m,k}$ of the form
\begin{equation}\label{eq_eq49}
P_{m,k}= \left(
\begin{array}{cccc}
1 & -p_2 & \ldots & -p_n \\
0 & 1 & \ldots & 0\\
\vdots & \vdots  & \ddots & \vdots \\
0 & 0 & \ldots & 1
\end{array}
\right),\quad Q_{m,k}= \left(
\begin{array}{cccc}
1 & 0 & \ldots & 0 \\
-q_2 & 1 & \ldots & 0\\
\vdots & \vdots  & \ddots & \vdots \\
-q_n & 0 & \ldots & 1
\end{array}
\right)
\end{equation}
such that the product $\frac{1}{{\lambda_m^k}} P_{m,k} R_m \Delta({\lambda_m^k}) R_m Q_{m,k}$ has the following form:
\begin{equation}\label{eq_eq50}
\frac{1}{{\lambda_m^k}} P_{m,k} R_m \Delta({\lambda_m^k}) R_m
Q_{m,k}= \left(
\begin{array}{cc}
0 & \begin{array}{ccc} 0 & \ldots & 0 \end{array} \\
\begin{array}{c}
0\\
\vdots \\
0
\end{array} & S_{m,k}
\end{array}
\right), \quad \det S_{m,k}\not=0,
\end{equation}

Moreover, for any $\varepsilon>0$ there exists $N\in \mathbb{Z}$
such that for any $|k|\ge N$ the components $p_i=p_i(m,k)$, $q_i=q_i(m,k)$ of the matrices~(\ref{eq_eq49})
may be estimated as follows:
\begin{equation}\label{eq_eq51}
|p_i|\le \varepsilon, \; |q_i|\le \varepsilon, \quad
i=2,\ldots,n.
\end{equation}
\end{lem}

\begin{proof} We begin with the analysis of the structure of the matrix
$$
\frac{1}{{\lambda_m^k}}R_m\Delta({\lambda_m^k})R_m= -I+
{\rm e}^{-{\lambda_m^k}} R_m A_{-1}R_m +
\int_{-1}^0 {\rm e}^{{\lambda_m^k} s} R_m \left(A_2(s) + \frac{1}{\lambda_m^k}A_3(s)\right) R_m \dd s.
$$
Since the matrix $A_{-1}$ is in Jordan form~(\ref{eq_eq34}), then
the multiplication of $A_{-1}$ on $R_m$ from the left and from the right
changes the places of the one-dimensional Jordan blocks $\mu_1$ and $\mu_m$:
$$
R_m A_{-1}R_m= \left(
\begin{array}{cc}
\mu_m & 0\\
0 & S
\end{array}
\right),\qquad S\in C^{(n-1)\times (n-1)}.
$$
We introduce the notation
\begin{equation}\label{eq_eq1888}
\int_{-1}^0 {\rm e}^{{\lambda} s} R_m \left(A_2(s) + \frac{1}{\lambda}A_3(s)\right) R_m \dd s=
\left(
\begin{array}{ccc}
\varepsilon_{11}({\lambda}) & \ldots & \varepsilon_{1n}({\lambda}) \\
\vdots & \ddots & \vdots \\
\varepsilon_{n1}({\lambda}) & \ldots & \varepsilon_{nn}({\lambda})
\end{array}
\right).
\end{equation}
According to Proposition~\ref{ut_u13}, elements of the matrix~(\ref{eq_eq1888})
tends to zero when $|{\rm Im}\lambda|\rightarrow\infty$ (and $|{\rm Re}\lambda|\le C<\infty$).
Thus, $|\varepsilon_{ij}(\lambda_m^k)|\rightarrow 0$ when $k\rightarrow\infty$.

In the introduced notation the singular matrix $\frac{1}{\lambda_m^k}R_m\Delta(\lambda_m^k) R_m$ has the
following form:
\begin{equation}\label{eq_eq24}
\frac{1}{\lambda_m^k}R_m\Delta(\lambda_m^k)R_m= \left(
\begin{array}{cc}
-1+{\rm e}^{-\lambda_m^k}\mu_m + \varepsilon_{11}(\lambda_m^k) &
\begin{array}{ccc}\varepsilon_{12}(\lambda_m^k)& \ldots & \varepsilon_{1n}(\lambda_m^k)
\end{array} \\
\begin{array}{c}
\varepsilon_{21}(\lambda_m^k)\\
\vdots \\
\varepsilon_{n1}(\lambda_m^k)
\end{array} & S_{m,k}
\end{array}
\right),
\end{equation}
where
\begin{equation}\label{eq_eq19}
S_{m,k} =
\left(-I_{n-1}+ {\rm e}^{-\lambda_m^k}S + \left(
\begin{array}{ccc}
\varepsilon_{22}(\lambda_m^k) & \ldots & \varepsilon_{2n}(\lambda_m^k) \\
\vdots & \ddots & \vdots \\
\varepsilon_{n2}(\lambda_m^k) & \ldots &
\varepsilon_{nn}(\lambda_m^k)
\end{array}
\right)\right)
\end{equation}
and $I_{n-1}$ is the identity matrix of the dimension $n-1$.
Let us prove that
\begin{equation}\label{eq_eq1999}
\det S_{m,k} \not=0.
\end{equation}
Consider the identity $-I_{n-1}+{\rm e}^{-\lambda_m^{k}} S =
-I_{n-1}+{\rm e}^{-\widetilde{\lambda}_m^{k}} S + ({\rm e}^{-\lambda_m^{k}}-{\rm e}^{-\widetilde{\lambda}_m^{k}})S$,
where $\widetilde{\lambda}_m^k=\Mi(\arg \mu_m +2\pi k)$ is an eigenvalue of the operator $\widetilde{\cal A}$.
Since ${\rm e}^{\widetilde{\lambda}_m^k}=\mu_m$ we have
$$
-I+{\rm e}^{-\widetilde{\lambda}_m^{k}} R_mA_{-1}R_m= \left(
\begin{array}{cc}
0 & 0\\
0 & -I_{n-1}+{\rm e}^{-\widetilde{\lambda}_m^{k}} S
\end{array}
\right),
$$
and since the multiplicity of the eigenvalue $\mu_m\in \sigma_1$ equals $1$,
we conclude that $\det(-I_{n-1}+{\rm e}^{-\widetilde{\lambda}_m^{k}} S)\not=0$.
Since $|\lambda_m^k - \widetilde{\lambda}_m^{k}|\rightarrow 0$
when $k\rightarrow\infty$, then for any $\varepsilon>0$ there exists
$N>0$ such that for $k:\; |k|\ge N$ the estimates
$|{\rm e}^{-\lambda_m^k}-{\rm e}^{-\widetilde{\lambda}_m^{k}}|\|S\|\le
\frac{\varepsilon}{2}$ and $|\varepsilon_{ij}({\lambda_m^k})|\le \frac{\varepsilon}{2}$ hold.
Thus, we have that $\det S_{m,k}= \det(-I_{n-1}+{\rm e}^{-\widetilde{\lambda}_m^{k}}S + B_{m,k})$, where the absolute value of each
component of $B_{m,k}$ is less than $\varepsilon$. Therefore, there exists $N>0$ such that $S_{m,k}$ is invertible and
we obtain the relation~(\ref{eq_eq1999}).

Since $\det S_{m,k}\not=0$ then from the relation~(\ref{eq_eq24}) we conclude that
the first row of the matrix $\frac{1}{\lambda_m^k}R_m\Delta(\lambda_m^k)R_m$
is a linear combination of all other rows and the first column is a linear combination of all other columns:
\begin{equation}\label{eq_eq16}
\varepsilon_{1i}(\lambda_m^k)=p_2 s_{2i}+\ldots + p_n s_{ni}, \quad
i=2,\ldots,n
$$
$$
\varepsilon_{j1}(\lambda_m^k)=q_2 s_{j2}+\ldots + q_n s_{jn}, \quad
j=2,\ldots,n.
\end{equation}
where $s_{ij}=s_{ij}(m,k)$, $2\le i,j \le n$ are the components of the matrix $S_{m,k}$.

Let us consider the matrices $P_{m,k}$, $Q_{m,k}$ of the form~(\ref{eq_eq49}) with the coefficients
$p_2,\ldots,p_n$ and $q_2,\ldots,q_n$ defined by (\ref{eq_eq16}).
Direct computations gives us that $\frac{1}{\lambda_m^k} P_{m,k} R_m \Delta(\lambda_m^k) R_m Q_{m,k}$ is
of the form~(\ref{eq_eq50}), i.e.:
$$
\frac{1}{\lambda_m^k} P_{m,k} R_m \Delta(\lambda_m^k) R_m
Q_{m,k}= \left(
\begin{array}{cc}
0 & \begin{array}{ccc} 0 & \ldots & 0 \end{array} \\
\begin{array}{c}
0\\
\vdots \\
0
\end{array} & S_{m,k}
\end{array}
\right).
$$

Let us estimate the coefficients $p_2,\ldots,p_n$ and $q_2,\ldots,q_n$. The equations (\ref{eq_eq16}) may be rewritten in the form:
$$
v_1=(S_{m,k})^T w_1, \qquad v_2=S_{m,k} w_2,
$$
where $v_1=(\varepsilon_{12}(\lambda_m^k), \ldots,
\varepsilon_{1n}(\lambda_m^k))^T$, $w_1=(p_2, \ldots, p_n)^T$,
$v_2=(\varepsilon_{21}(\lambda_m^k), \ldots,
\varepsilon_{n1}(\lambda_m^k))^T$, $w_2=(q_2, \ldots, q_n)^T$. Since
$\det S_{m,k}\not=0$ then
$$
w_1=(S_{m,k}^{-1})^T v_1, \qquad w_2=S_{m,k}^{-1} v_2
$$
and since the values $\varepsilon_{ij}(\lambda_m^k)$ are small then to show (\ref{eq_eq51})
we have to prove that the estimate $\|S_{m,k}^{-1}\|\le C$, $C>0$ holds for all $k: |k|\ge N$.

As we have shown above $S_{m,k}=-I_{n-1}+\frac{1}{\mu_m}S+B_{m,k}$, where elements of the
matrices $B_{m,k}$ tend to zero when $k\rightarrow\infty$. Thus, there
exists $N\in\mathbb{Z}$ such that for all $k: |k|\ge N$ the norm of the matrix
$\widetilde{B}_{m,k}{\stackrel{\rm def}{=}}-\left(-I_{n-1}+\frac{1}{\mu_m}S\right)^{-1}B_{m,k}$
is small enough, say $\|\widetilde{B}_{m,k}\|<\frac12$. Thus, there exists the
inverse matrix of $I_{n-1}-\widetilde{B}_{m,k}$ for every $|k|\ge N$, and these inverse matrices are bounded uniformly by $k$:
$$
\|(I_{n-1}-\widetilde{B}_{m,k})^{-1}\|=\|\sum_{i=0}^{\infty}(\widetilde{B}_{m,k})^i\|\le C_1, \quad |k|\ge N.
$$
Thus, we obtain the estimate
$$
\|S_{m,k}^{-1}\|=\left\|(I_{n-1}-\widetilde{B}_{m,k})^{-1}\left(-I_{n-1}+\frac{1}{\mu_m}S\right)^{-1}\right\|
\le C_1\left\|\left(-I_{n-1}+\frac{1}{\mu_m}S\right)^{-1}\right\| \le C,
$$
what completes the proof of the lemma.
\end{proof}

\begin{col}\label{col_col1}
The matrix function $\widehat{\Delta}_{m,k}({\lambda})\stackrel{{\rm def}}{=}
\frac{1}{{\lambda}} P_{m,k} R_m \Delta({\lambda}) R_m  Q_{m,k}$, where $P_{m,k}$, $Q_{m,k}$ are given by (\ref{eq_eq49}),
allows the following representation in a neighborhood $U({\lambda_m^k})$ of the corresponding eigenvalue ${\lambda_m^k}\in \Lambda_1$, $|k|\ge N$:
\begin{equation}\label{eq_eq56}
\widehat{\Delta}_{m,k}(\lambda)= \left(
\begin{array}{cc}
(\lambda-{\lambda_m^k})r_{11}^{m,k}(\lambda) &
\begin{array}{ccc} (\lambda-{\lambda_m^k})r_{12}^{m,k}(\lambda) & \ldots & (\lambda-{\lambda_m^k})r_{1n}^{m,k}(\lambda)
\end{array}\\
\begin{array}{c}
(\lambda-{\lambda_m^k})r_{21}^{m,k}(\lambda)\\
\vdots \\
(\lambda-{\lambda_m^k})r_{n1}^{m,k}(\lambda)
\end{array} & S_{m,k}(\lambda)
\end{array}
\right),
\end{equation}
where the functions $r_{ij}^{m,k}(\lambda)$ are analytic in $U({\lambda_m^k})$. Moreover,
\begin{equation}\label{eq_eq57}
r_{11}^{m,k}({\lambda_m^k})\not=0, \qquad
|r_{11}^{m,k}({\lambda_m^k})|\rightarrow 1, \; k \rightarrow \infty.
\end{equation}
\end{col}

\begin{proof} Since $\Delta({\lambda})$ is analytic, then all the components of the matrix function
$\widehat{\Delta}_{m,k}({\lambda})=\frac{1}{{\lambda}} P_{m,k} R_m \Delta({\lambda}) R_m Q_{m,k}$ are
analytic in a neighborhood of the point $\lambda_m^k$.
Moreover, since the matrix $\widehat{\Delta}_{m,k}(\lambda_m^k)$ has the form (\ref{eq_eq50}),
then we conclude that $\widehat{\Delta}_{m,k}({\lambda})$ is of the form (\ref{eq_eq56}).

Let us prove the relation~(\ref{eq_eq57}). If we assume that $r_{11}^{m,k}(\lambda_m^k)=0$, then
$(\lambda-\lambda_m^k)r_{11}^{m,k}(\lambda)=(\lambda-\lambda_m^k)^2\widehat{r}_{11}^{m,k}(\lambda)$,
where $\widehat{r}_{11}^{m,k}(\lambda)$ is analytic.
The last implies that the multiplicity of the root $\lambda=\lambda_m^k$ of the equation $\det \widehat{\Delta}_{m,k}(\lambda) = 0$
is greater or equal to $2$, i.e. $\det \widehat{\Delta}_{m,k}(\lambda) = (\lambda-\lambda_m^k)^2 r(\lambda)$,
where $r(\lambda)$ is an analytic function.
Indeed, decomposing $\det \widehat{\Delta}_{m,k}(\lambda)$ by the
elements of the first row, we see that all the term of this
decomposition have the common multiplier $(\lambda-\lambda_m^k)^2$.
Thus, we obtain that the multiplicity of $\lambda=\lambda_m^k$ as the root of the equation
$\det \Delta(\lambda) = 0$ is greater or equal to $2$, what contradicts to the assumption that $\lambda_m^k$ is an eigenvalue of
the multiplicity one of the operator ${\cal A}$.

Taking into account (\ref{eq_eq24}) and the form of the
transformations (\ref{eq_eq50}), we see that
\begin{equation}\label{eq_eq5553}
(\lambda-\lambda_m^k)r_{11}^{m,k}(\lambda)=\left(-1+{\rm e}^{-\lambda}\mu_m + \varepsilon_{11}(\lambda)\right)
-\sum_{i=2}^n p_i\varepsilon_{i1}(\lambda)- \sum_{j=2}^n
q_j\varepsilon_{1j}(\lambda).
\end{equation}
Differentiating (\ref{eq_eq5553}) by $\lambda$ and substituting $\lambda=\lambda_m^k$, we obtain
$$
r_{11}^{m,k}(\lambda_m^k)=
-{\rm e}^{-\lambda_m^k}\mu_m+ \left(\varepsilon_{11}(\lambda)-\sum_{i=2}^n
p_i\varepsilon_{i1}(\lambda)- \sum_{j=2}^n
q_j\varepsilon_{1j}(\lambda)\right)'_{\lambda=\lambda_m^k}.
$$
The terms $(\varepsilon_{ij}(\lambda))'$ are of the form
$$
(\varepsilon_{ij}(\lambda))'=
\int_{-1}^0 {\rm e}^{\lambda s} \left(sA_2(s) + \frac{s}{\lambda}A_3(s)-\frac{1}{\lambda^2}A_3(s) \right)_{ij} \dd s,
$$
therefore, due to Proposition~\ref{ut_u13} and Lemma~\ref{lm_lm6}, we conclude that
$$
\left(\varepsilon_{11}(\lambda)-\sum_{i=2}^n
p_i\varepsilon_{i1}(\lambda)- \sum_{j=2}^n
q_j\varepsilon_{1j}(\lambda)\right)'_{\lambda=\lambda_m^k}
\rightarrow 0, \quad k\rightarrow\infty.
$$

Since $-{\rm e}^{-\lambda_m^k}\mu_m\rightarrow -1$ when
$k\rightarrow\infty$, then we obtain the relation~(\ref{eq_eq57})
and, in particular, there exists a constant $C>0$ and an integer $N$
such that for $|k|>N$ we have
\begin{equation}\label{eq_eq5333}
0<C\le \left|r_{11}^{m,k}(\lambda_m^k)\right|.
\end{equation}
The last completes the proof of the proposition.
\end{proof}

\begin{rem}\label{rm_rm_r0}
The same arguments gives us that
\begin{equation}\label{eq_eq5333aw}
|r_{i1}^{m,k}({\lambda_m^k})|\rightarrow 0, \; |r_{1j}^{m,k}({\lambda_m^k})|\rightarrow 0,\quad k \rightarrow \infty.
\end{equation}
for all $i=2,\ldots,n$ and for all $i=2,\ldots,n$.
\end{rem}
\begin{proof}
Indeed, let us consider $r_{1j}^{m,k}({\lambda})$ for $j=2,\ldots,n$ and use the fact that $A_{-1}$ is in a Jordan form:
$$
(\lambda-\lambda_m^k)r_{1j}^{m,k}(\lambda)=\varepsilon_{1j}(\lambda) -\sum_{i=2}^n p_i\varepsilon_{ij}(\lambda)
+p_j \left(-1+{\rm e}^{-\lambda}\mu\right) +p_{j-1} c,
$$
where $\mu\in\sigma(A_{-1})$ and the constant $c=0$ if $\mu$ is geometrically simple or, otherwise, $c=1$.
Thus, we obtain
$$
r_{1j}^{m,k}(\lambda_m^k)=
-p_j{\rm e}^{-\lambda_m^k}\mu+ \left(\varepsilon_{1j}(\lambda)-\sum_{i=2}^n
p_i\varepsilon_{i1}(\lambda)\right)'_{\lambda=\lambda_m^k}
$$
and since $p_i=p_i(m,k)\rightarrow 0$ when $k \rightarrow \infty$ due to Lemma~\ref{lm_lm6}, then
we conclude that $|r_{1j}^{m,k}({\lambda_m^k})|\rightarrow 0$ when $k \rightarrow \infty$.
\end{proof}

\begin{rem}\label{rm_rm2}
Direct computations give us:
\begin{equation}\label{eq_eq53}
P_{m,k}^{-1}= \left(
\begin{array}{cccc}
1 & p_2 & \ldots & p_n \\
0 & 1 & \ldots & 0\\
\vdots & \vdots  & \ddots & \vdots \\
0 & 0 & \ldots & 1
\end{array}
\right),\quad Q_{m,k}^{-1}= \left(
\begin{array}{cccc}
1 & 0 & \ldots & 0 \\
q_2 & 1 & \ldots & 0\\
\vdots & \vdots  & \ddots & \vdots \\
q_n & 0 & \ldots & 1
\end{array}
\right).
\end{equation}
\end{rem}

\begin{lem}\label{lm_lm7}
Let $\sigma_1=\{\mu_1,\ldots,\mu_{\ell_1}\}$ consists of simple eigenvalues only.
There exist constants $0<C_1<C_2$, $N\in\mathbb{Z}$ such that for
any ${\lambda_m^k}\in\Lambda_1$, $|k|\ge N$ the following estimate holds:
\begin{equation}\label{eq_eq62}
0<C_1\le \left|\frac{1}{{\lambda_m^k}} \langle
\Delta'({\lambda_m^k})x_m^k, y_m^k\rangle\right|\le C_2,
\end{equation}
where  $\Delta'({\lambda})=\frac{{\rm d}}{{\rm d}\lambda}\Delta(\lambda)$;
$x_m^k=x(\lambda_m^k)$, $y_m^k=y(\overline{\lambda}_m^k)$ are defined by (\ref{eq_eq27}), (\ref{eq_eq28})
and $\|x_m^k\|=\|y_m^k\|=1$.
\end{lem}

\begin{proof} First, we prove the estimate~(\ref{eq_eq62}) for eigenvalues
${\lambda_1^k}\in\Lambda_1$.
Since $x_1^k\in \Ker \Delta({\lambda_1^k})$, then
\begin{equation}\label{eq_eq52}
0=\frac{1}{{\lambda_1^k}} P_{1,k}^{-1} P_{1,k}
\Delta({\lambda_1^k}) Q_{1,k} Q_{1,k}^{-1}x_1^k =P_{1,k}^{-1}
\widehat{\Delta}_{1,k}({\lambda_1^k}) Q_{1,k}^{-1}x_1^k,
\end{equation}
where $P_{1,k}$, $Q_{1,k}$, $\widehat{\Delta}_{1,k}({\lambda_1^k})$ are
defined by (\ref{eq_eq49}), (\ref{eq_eq50}).
Thus, $Q_{1,k}^{-1}x_1^k\in \Ker \widehat{\Delta}_{1,k}({\lambda_1^k})$
and, taking into account the form~(\ref{eq_eq50}) of the matrix $\widehat{\Delta}_{1,k}({\lambda_1^k})$,
we conclude that $Q_{1,k}^{-1}x_1^k = (\widehat{x}_1, 0,\ldots, 0)^T$, $\widehat{x}_1\not=0$.
On the other hand, multiplying directly $Q_{1,k}^{-1}$ given by (\ref{eq_eq53})
on $x_1^k=((x_1^k)_1, \ldots, (x_1^k)_n)^T$,
we conclude that $\widehat{x}_1=(x_1^k)_1$ and $(x_1^k)_i=-q_i(x_1^k)_1$, $i=2,\ldots,n$.

Due to the relation (\ref{eq_eq51}), for any $\varepsilon>0$ there exists $N\in\mathbb{N}$ such that for all $k: |k|\ge N$ we have:
$$
1=\|x_1^k\|^2=|(x_1^k)_1|^2\left(1+|q_2|^2+\ldots+|q_n|^2\right)\le
|(x_1^k)_1|^2\left(1+(n-1)\varepsilon^2\right)
$$
what implies that $|(x_1^k)_1|\rightarrow 1$ when $k \rightarrow \infty$.
Finally, we obtain
\begin{equation}\label{eq_eq60}
Q_{1,k}^{-1}x_1^k = \left((x_1^k)_1, 0,\ldots, 0\right)^T, \qquad 0< C\le
|(x_1^k)_1|\le 1, \; |k|\ge N.
\end{equation}

Conjugating (\ref{eq_eq50}), we have the relation
\begin{equation}\label{eq_eq54}
\left(
\begin{array}{cc}
0 & \begin{array}{ccc} 0 & \ldots & 0 \end{array} \\
\begin{array}{c}
0\\
\vdots \\
0
\end{array} & S_{1,k}^*
\end{array}
\right) =\left(\frac{1}{{\lambda_1^k}} P_{1,k} \Delta({\lambda_1^k})
Q_{1,k}\right)^* =\frac{1}{\overline{{\lambda_1^k}}} Q_{1,k}^*
\Delta^*(\overline{{\lambda_1^k}}) P_{1,k}^*
=\widehat{\Delta}_{1,k}^*(\overline{{\lambda_1^k}}).
\end{equation}

Let us use the fact that $y_1^k\in \Ker \Delta^*(\overline{{\lambda_1^k}})$:
\begin{equation}\label{eq_eq55}
0=\frac{1}{\overline{{\lambda_1^k}}} (Q_{1,k}^*)^{-1} Q_{1,k}^*
\Delta^*(\overline{{\lambda_1^k}}) P_{1,k}^* (P_{1,k}^*)^{-1} y_1^k
=(Q_{1,k}^*)^{-1} \widehat{\Delta}_{1,k}^*(\overline{{\lambda_1^k}})
(P_{1,k}^*)^{-1} y_1^k.
\end{equation}
From the last we obtain that $(P_{1,k}^*)^{-1} y_1^k\in \Ker \widehat{\Delta}_{1,k}^*(\overline{{\lambda_1^k}})$,
and, taking into account the left-hand side of~(\ref{eq_eq54}),
we conclude that $(P_{1,k}^*)^{-1} y_1^k = (\widehat{y}_1, 0,\ldots, 0)^T$, $\widehat{y}_1\not=0$.
Multiplying $(P_{1,k}^*)^{-1}$ on $y_1^k=((y_1^k)_1,\ldots,(y_1^k)_1)^T$ we obtain the relations
$\widehat{y}_1=(y_1^k)_1$ and $(y_1^k)_i=-\overline{p}_i(y_1^k)_1$, $i=2,\ldots,n$.
Thus, due to (\ref{eq_eq51}), any $\varepsilon>0$ and $k: |k|\ge N$ we have:
$$
1=\|y_1^k\|^2=|(y_1^k)_1|^2\left(1+|\overline{p}_2|^2+\ldots+|\overline{p}_n|^2\right)
\le |(y_1^k)_1|^2\left(1+(n-1)\varepsilon^2\right)
$$
and we conclude that $|(y_1^k)_1|\rightarrow 1$ when $k\rightarrow \infty$. Finally,
\begin{equation}\label{eq_eq61}
(P_{1,k}^*)^{-1} y_1^k = ((y_1^k)_1, 0,\ldots, 0)^T, \qquad 0< C\le
|(y_1^k)_1|\le 1,  \; |k|\ge N.
\end{equation}

Differentiating~(\ref{eq_eq56}) by $\lambda$ and putting $\lambda={\lambda_1^k}$,
we obtain
\begin{equation}\label{eq_eq58}
\left(
\begin{array}{cc}
r_{11}^{1,k}({\lambda_1^k}) &
\begin{array}{ccc} r_{12}^{1,k}({\lambda_1^k}) & \ldots & r_{1n}^{1,k}({\lambda_1^k})
\end{array}\\
\begin{array}{c}
r_{21}^{1,k}({\lambda_1^k})\\
\vdots \\
r_{n1}^{1,k}({\lambda_1^k})
\end{array} & S'_{m,k}({\lambda_1^k})
\end{array}
\right)=\widehat{\Delta}_{1,k}'({\lambda_1^k})
=\left(\frac{1}{\lambda} P_{1,k} \Delta(\lambda) Q_{1,k}\right)'_{\lambda={\lambda_1^k}}
$$
$$
= P_{1,k} \left(\frac{1}{{\lambda_1^k}} \Delta'({\lambda_1^k}) -
\frac{1}{(\lambda_1^k)^2}\Delta({\lambda_1^k})\right) Q_{1,k}.
\end{equation}
Using (\ref{eq_eq58}) and the relation $x_1^k\in \Ker \Delta({\lambda_1^k})$, we obtain
\begin{equation}\label{eq_eq59}
\begin{array}{rcl}
\frac{1}{{\lambda_1^k}} \left\langle \Delta'({\lambda_1^k})x_1^k,\; y_1^k \right\rangle & = &
\left\langle  P_{1,k}^{-1}
P_{1,k}\left(\frac{1}{{\lambda_1^k}} \Delta'({\lambda_1^k})-
\frac{1}{(\lambda_1^k)^2}\Delta({\lambda_1^k})\right)Q_{1,k}
Q_{1,k}^{-1} x_1^k,\; y_1^k\right\rangle\\
& = & \quad\left\langle  P_{1,k}\left(\frac{1}{{\lambda_1^k}}
\Delta'({\lambda_1^k})-
\frac{1}{(\lambda_1^k)^2}\Delta({\lambda_1^k})\right)Q_{1,k}
Q_{1,k}^{-1} x_1^k,\; (P_{1,k}^{-1})^* y_1^k\right\rangle\\
&= &\left\langle \widehat{\Delta}_{1,k}'({\lambda_1^k}) Q_{1,k}^{-1} x_1^k,\; (P_{1,k}^{-1})^* y_1^k \right\rangle.
\end{array}
\end{equation}
Finally, using the representation~(\ref{eq_eq58}) of the matrix $\widehat{\Delta}_{1,k}'({\lambda_1^k})$ and
representations (\ref{eq_eq60}), (\ref{eq_eq61}) of the vectors $Q_{1,k}^{-1} x_1^k$, $(P_{1,k}^{-1})^* y_1^k$,
we conclude that
\begin{equation}\label{eq_eq59999}
\frac{1}{{\lambda_1^k}} \left\langle \Delta'({\lambda_1^k})x_1^k,\; y_1^k\right\rangle
=r_{11}^{1,k}({\lambda_1^k})(x_1^k)_1 (y_1^k)_1.
\end{equation}
Moreover, taking into account the estimate~(\ref{eq_eq57}) of Corollary~\ref{col_col1} and (\ref{eq_eq60}), (\ref{eq_eq61}),
we obtain the estimate~(\ref{eq_eq62}), what proves the lemma for the case of eigenvalues ${\lambda_1^k}$, i.e. for $m=1$.

Let us now prove the estimate~(\ref{eq_eq62}) for ${\lambda_m^k}\in \Lambda_1$, $m=2,\ldots,\ell_1$.
In this case the idea of the proof remains the same but the arguing appears to be more cumbersome.
In the proof we omit some detailed explanations that were given above for the case $m=1$.

Let us consider the product $R_m \Delta({\lambda_m^k}) R_m$.
Using the relation $x_m^k\in \Ker \Delta({\lambda_m^k})$, we have
\begin{equation}\label{eq_eq64}
0=\frac{1}{{\lambda_m^k}} R_m P_{m,k}^{-1} P_{m,k}  R_m
\Delta({\lambda_m^k}) R_m Q_{m,k} Q_{m,k}^{-1} R_m x_m^k =R_m
P_{m,k}^{-1}  \widehat{\Delta}_{m,k}({\lambda_m^k}) Q_{m,k}^{-1} R_m x_m^k.
\end{equation}
Thus, $Q_{m,k}^{-1} R_m x_m^k\in \Ker \widehat{\Delta}_{m,k}({\lambda_m^k})$ and
from the explicit form (\ref{eq_eq50}) of $\widehat{\Delta}_{m,k}({\lambda_m^k})$
we conclude that $Q_{m,k}^{-1}R_m x_m^k = (\widehat{x}_1, 0,\ldots, 0)^T$, $\widehat{x}_1\not=0$.
Multiplying $Q_{m,k}^{-1}$ on $R_m$ from the right, we changes places
the first and the $m$-th column of $Q_{m,k}^{-1}$, therefore,
we obtain:
$$(x_m^k)_m=\widehat{x}_1,\; (x_m^k)_1=-q_m(x_m^k)_m, \; (x_m^k)_i=-q_i(x_m^k)_m,\; i=2,\ldots, n, i\not=m.$$
Thus, taking into account (\ref{eq_eq51}), for any $\varepsilon>0$ there exists $N\in\mathbb{N}$ such that for all $k: |k|\ge N$ we have:
$$
1=\|x_m^k\|^2\le|(x_m^k)_m|^2(1+(n-1)\varepsilon^2)
$$
and, thus, $|(x_m^k)_m|\rightarrow 1$ when $k \rightarrow \infty$. Therefore,
\begin{equation}\label{eq_eq65}
Q_{m,k}^{-1}R_m x_m^k = \left((x_m^k)_m, 0,\ldots, 0\right)^T, \qquad 0< C\le
|(x_m^k)_m|\le 1,\; |k|\ge N.
\end{equation}
The similar arguing gives us that
\begin{equation}\label{eq_eq66}
(P_{m,k}^{-1})^*R_m y_m^k = ((y_m^k)_m, 0,\ldots, 0)^T, \qquad 0<
C\le |(y_m^k)_m|\le 1,\; |k|\ge N.
\end{equation}
Differentiating~(\ref{eq_eq56}) by $\lambda$ and putting $\lambda={\lambda_m^k}$,
we obtain
\begin{equation}\label{eq_eq67}
\left(
\begin{array}{cc}
r_{11}^{m,k}({\lambda_m^k}) &
\begin{array}{ccc} r_{12}^{m,k}({\lambda_m^k}) & \ldots & r_{1n}^{m,k}({\lambda_m^k})
\end{array}\\
\begin{array}{c}
r_{21}^{m,k}({\lambda_m^k})\\
\vdots \\
r_{n1}^{m,k}({\lambda_m^k})
\end{array} & S'_{m,k}({\lambda_m^k})
\end{array}
\right) =\widehat{\Delta}_{m,k}'({\lambda_m^k})
$$
$$
=P_{m,k} R_m\left(\frac{1}{{\lambda_m^k}}
\Delta'({\lambda_m^k}) -
\frac{1}{(\lambda_m^k)^2}\Delta({\lambda_m^k})\right)R_m Q_{m,k}.
\end{equation}

Finally, using  (\ref{eq_eq65}),  (\ref{eq_eq66}), (\ref{eq_eq67}) and the relation $x_1^k\in \Ker \Delta({\lambda_1^k})$, we obtain
$$
\begin{array}{l}
\frac{1}{{\lambda_m^k}} \left\langle \Delta'({\lambda_m^k})x_m^k,\; y_m^k\right\rangle = \\
\qquad\qquad = \left\langle R_m P_{m,k}^{-1} P_{m,k} R_m
\left(\frac{1}{{\lambda_m^k}} \Delta'({\lambda_m^k})-
\frac{1}{(\lambda_m^k)^2}\Delta({\lambda_m^k})\right)R_m Q_{m,k}
Q_{m,k}^{-1}R_m x_m^k,\; y_m^k\right\rangle \\
\qquad\qquad\; = \left\langle  \widehat{\Delta}_{m,k}'({\lambda_m^k}) R_m Q_{m,k}^{-1} x_m^k,\; (P_{m,k}^{-1})^* R_m y_m^k\right\rangle \\
\qquad\qquad = r_{11}^{m,k}({\lambda_m^k})(x_m^k)_m (y_m^k)_m.
\end{array}
$$
To complete the proof of the lemma it remains to apply the estimates~(\ref{eq_eq57}), (\ref{eq_eq65}) and  (\ref{eq_eq66}).
\end{proof}
\section{Boundedness of the resolvent on invariant subspaces}\label{sect_resolvent}

In this section we prove the exponential stability of the restriction of
the semigroup $\{{\rm e}^{t {\mathcal A}}\}_{t\ge 0}$ onto $M_2^0$
(i.e. the semigroup $\{{\rm e}^{t {\mathcal A}}|_{M_2^0}\}_{t\ge 0}$),
where the invariant subspace $M_2^0$ is defined in Section~3 by~(\ref{eq_eq23}).

To show this we use the following well-known equivalent condition of exponential stability
(see e.g. \cite[p.119]{Neerven_1996} or \cite[p.139]{Luo_Guo_Morgul_1999}):

\noindent
{\it Let $T(t)$ be a $C_0$-semigroup on a Hilbert space $H$ with a generator $A$.
Then $T(t)$ is exponentially stable if and only if the following
conditions hold:

1. $\{\lambda: \MRe \lambda \ge 0\}\subset \rho(A)$;

2. $\|R(\lambda, A)\|\le M$ for all $\{\lambda: \MRe \lambda \ge
0\}$ and for some constant $M>0$.}

\vskip1ex
\noindent
Thus, we reformulate our aim as follows.
\begin{thr}[on resolvent boundedness]\label{thr_boundedness}
Let $\sigma_1=\{\mu_1,\ldots,\mu_{\ell_1}\}$ consists of simple eigenvalues only.
On the subspace $M_2^0$ defined by~(\ref{eq_eq23}) the restriction of the resolvent $R(\lambda, {\cal A})|_{M_2^0}$
is uniformly bounded for $\lambda: \MRe\lambda\ge 0$, i.e. there exists $C>0$ such that
\begin{equation}\label{eq_eq33}
\left\|R(\lambda, {\cal A})x\right\|\le C \|x\|, \quad x\in M_2^0.
\end{equation}
\end{thr}

Let us briefly describe the ideas of the proof.
From the explicit form of the resolvent~(\ref{resolvent}) we conclude that the main difficulty
is to prove the uniform boundedness of the term $\triangle^{-1}(\lambda) D(z, \xi, \lambda)$ in neighborhoods
of the eigenvalues of $\mathcal A$ located close to the imaginary axis.
Indeed, since $\det \triangle(\lambda_m^k)=0$ for $\lambda_m^k\in \Lambda_1$ and
${\rm Re}\lambda_m^k\rightarrow 0$ when $k\rightarrow\infty$ then norm of $\triangle^{-1}(\lambda)$ grows infinitely
when ${\rm Re}\lambda \rightarrow 0$ and ${\rm Im}\lambda \rightarrow \infty$ simultaneously.
However, the product $\triangle^{-1}(\lambda) D(z, \xi, \lambda)$ turns out to be bounded for $(z, \xi(\cdot))\in M_2^0$.
\begin{lem}\label{lm_lm2}
The vector-function $\triangle^{-1}(\lambda) D(z, \xi, \lambda): M_2^0\times \mathbb{C}^n\rightarrow  \mathbb{C}^n$ is
uniformly bounded in neighborhoods $U_\delta(\lambda_m^k)$ of
eigenvalues $\lambda_m^k\in\Lambda_1$ for some fixed $\delta>0$,
i.e.:

1.for any $k: |k|>N$ and $m=1,\ldots,\ell_1$ there exists a constant $C_{m,k}$
such that the estimate $\left\|\triangle^{-1}(\lambda) D(z, \xi, \lambda)\right\|\le C_{m,k}\|(z, \xi(\cdot))\|_{M_2}$
holds for all $\lambda\in U_\delta(\lambda_m^k)$ and $(z, \xi(\cdot))\in M_2^0$.

2. there exists a constant $C>0$ such that $C_{m,k}\le C$ for all $m=1,\ldots,\ell_1$, $k: |k|>N$.
\end{lem}
The proof of this lemma is technically difficult. It essentially uses the following relation.
\begin{lem}\label{lm_lm1}
For any vector $g=(z, \xi(\cdot))^T\in M_2^0$ and for any
eigenvalue $\lambda_m^k\in \Lambda_1$ the following relation holds:
\begin{equation}\label{orthogonality_col_rs}
D(z, \xi, \lambda_m^k) \in  {\rm Im}\triangle(\lambda_m^k).
\end{equation}
\end{lem}

The complete proofs of the mentioned propositions are given in the next subsection.

\begin{rem}\label{thr_boundedness_stab}
Theorem~\ref{thr_boundedness} holds also for the subspace $M_2^0$ defined by~(\ref{eq_eq20-2}) under the assumption $\Lambda_2=\emptyset$.
The proof remains the same.
\end{rem}
\subsection{The proof of Theorem~\ref{thr_boundedness}}

We begin with several auxiliary propositions.

\begin{ut}\label{ut_u11}
If the vector $y\in \MIm A$, $A\in\mathbb{C}^{n\times n}$, then for
any two matrices $P$, $Q$ such that $\det Q\not=0$ the relation $Py\in \MIm (PAQ)$ holds.
\end{ut}

\begin{proof} The relation $y\in \MIm A$
means that there exists a vector $x$ such that $Ax=y$. Since $Q$ is
non-singular then there exists a vector $x_1$ such that $x=Qx_1$.
Therefore, $AQx_1=y$ and, multiplying on $P$ from the left we obtain
$PAQx_1=Py$.
\end{proof}

\begin{ut}\label{ut_u13}
Let $L_0\subset\mathbb{C}$ be a bounded closed set and $f(s)\in
L_2[-1,0]$. Denote by $a_k(\lambda)=\int\limits_{-1}^0 {\rm e}^{2\pi {\rm i}ks}
{\rm e}^{\lambda s} f(s) \dd s$, $\lambda\in L_0$, $k\in \mathbb{Z}$. Then
$|a_k(\lambda)|\rightarrow 0$ when $k\rightarrow \infty$ uniformly
on the set $L_0$.
\end{ut}

\begin{proof} Integrals $a_k(\lambda)$
can be considered as Fourier coefficients of the function
${\rm e}^{\lambda s} f(s)$, thus, they converge to zero when $k\rightarrow
\infty$. It remains to prove that they converge uniformly on the set
$L_0$. The last means that for any $\varepsilon>0$ there exists
$n\in \mathbb{N}$ such that for any $|k|\ge n$ and for any
$\lambda\in L_0$ we have $|a_k(\lambda)|<\varepsilon$.

Let us suppose the contrary: $\exists \varepsilon >0$ such that
$\forall n\in \mathbb{N}$, $\exists |k|\ge n$, $\exists \lambda\in
L_0$: $|a_k(\lambda)| \ge\varepsilon$. Thus, there exists a sequence
$k_1<k_2<\ldots$ and a sequence $\{\lambda_{k_i}\}_{i=1}^\infty$
such that $|a_{k_i}(\lambda_{k_i})| \ge\varepsilon$.

Since $L_0$ is a bounded set then there exists a converging
subsequence of $\{\lambda_{k_i}\}_{i=1}^\infty$ which we denote by
$\{\lambda_j\}_{j\in J}$, where $J\subset N$ is a strictly
increasing sequence. Moreover, since $L_0$ is also closed, then the
limit of $\{\lambda_{j}\}_{j\in J}$ belongs to $L_0$:
$\lambda_{j}\rightarrow \lambda_0\in L_0$. Let us show that the
sequence $\{a_k(\lambda_0)\}$ does not converge to zero.

Indeed, choosing big enough $n\in \mathbb{N}$, such that for any
$j>n$, $j\in J$ and any $s\in [-1,0]$: $|{\rm e}^{\lambda_0 s} -
{\rm e}^{\lambda_j s}|\le \frac{\varepsilon}{2}\|f(s)\|$, we obtain
$$
|a_j(\lambda_0)-a_j(\lambda_j)|=\left|\int_{-1}^0 {\rm e}^{2\pi {\rm i} j s}
({\rm e}^{\lambda_0 s} - {\rm e}^{\lambda_j s}) f(s) \dd s\right|
\le \int_{-1}^0 |{\rm e}^{\lambda_0 s} - {\rm e}^{\lambda_j s}| f(s) \dd s \le
\frac{\varepsilon}{2}.
$$
Since $|a_j(\lambda_j)| \ge\varepsilon$ and assuming that
$|a_j(\lambda_0)|\le |a_j(\lambda_j)|$, we obtain
$$
\frac{\varepsilon}{2} \ge |a_j(\lambda_0)-a_j(\lambda_j)| \ge
|a_j(\lambda_j)| - |a_j(\lambda_0)| \ge \varepsilon-
|a_j(\lambda_0)|,
$$
and, thus, $|a_j(\lambda_0)|\ge \frac{\varepsilon}{2}$ for any $j\in
J$, $j>n$.

Thus, $\{a_k(\lambda_0)\}$ does not converge to zero and we have
obtained a contradiction with the fact that they are the
coefficients of the Fourier series of the function ${\rm e}^{\lambda_0 s}
f(s)$. The last completes the proof of the proposition.
\end{proof}

\begin{col}\label{col_col2}
If the sequence $\{\lambda_k\}$ is such that
$\MIm\lambda_k\rightarrow\infty$ and $-\infty<a\le\MRe\lambda_k\le
b<\infty$ then for any $f(s)\in L_2(0,1;\mathbb{C}^{n\times m})$ we
have: $\int_{-1}^0 {\rm e}^{\lambda_k s} f(s) \dd s \rightarrow 0$ when $k
\rightarrow\infty$.
\end{col}

\begin{lem}\label{lm_lm3} The following estimates hold:

1. There exists a constant $C>0$ such that $\left\|\frac{1}{\lambda}\triangle(\lambda)\right\|\le C$
for all $\lambda\in \{\lambda:\; \MRe \lambda\ge 0\}\backslash U(0)$, and $\left\|\triangle(\lambda)\right\|\le C$ for all $\lambda\in U(0)$.

2. There exists a constant $C>0$ such that $\left\|\frac{1}{\lambda} D(z, \xi, \lambda)\right\|\le C\|(z, \xi(\cdot))\|_{M_2}$
for all $\lambda\in \{\lambda:\; \MRe \lambda\ge 0\}\backslash U(0)$, and $\left\|D(z, \xi, \lambda)\right\|\le C\|(z, \xi(\cdot))\|_{M_2}$
for all $\lambda\in U(0)$, $(z, \xi(\cdot))\in {M_2}$.
\end{lem}
\begin{proof}
From the explicit form~(\ref{eq_delta}) of $\triangle(\lambda)$ we have an estimate
$$
\left\|\frac{1}{\lambda}\triangle(\lambda)\right\|\le 1 + \|A_{-1}\|
+ \left\|\int\nolimits_{-1}^0 {\rm e}^{\lambda s} A_2(s) {\rm d} s\right\|
+\frac{1}{|\lambda|} \left\|\int\nolimits_{-1}^0 {\rm e}^{\lambda s} A_3(s) {\rm d} s\right\|
$$
for $\lambda\in \{\lambda:\; \MRe \lambda\ge 0\}\backslash U(0)$,
and, thus, it remains to prove that there exists a constant $C_1>0$
such that $\left\|\int\nolimits_{-1}^0 {\rm e}^{\lambda s} A_i(s) {\rm d} s\right\|\le C_1$, $i=2,3$.
Indeed, if we suppose the contrary, then there exists an unbounded sequence $\{\lambda_j\}_{j=1}^\infty$ such that
$\left\|\int\nolimits_{-1}^0 {\rm e}^{\lambda_j s} A_i(s) {\rm d} s\right\|\rightarrow\infty$ when $j\rightarrow\infty$.
On the other hand, it is easy to see that for any $k\ge 0$: $\int\nolimits_{-1}^0 {\rm e}^{\lambda s} s^k {\rm d} s\rightarrow 0$ when
$|\lambda|\rightarrow\infty$ and $\lambda\in \{\lambda:\; \MRe \lambda\ge 0\}$. Since the set of polynomials is everywhere dense in $L_2(-1,0)$,
then $\left\|\int\nolimits_{-1}^0 {\rm e}^{\lambda s} A_i(s) {\rm d} s\right\|\rightarrow 0$ when
$|\lambda|\rightarrow\infty$, $\lambda\in \{\lambda:\; \MRe \lambda\ge 0\}$ and we have come to a contradiction.

The estimate $\left\|\triangle(\lambda)\right\|\le C$ for all $\lambda\in U(0)$ follows easily from the explicit form~(\ref{eq_delta}).

The estimates for $D(z, \xi, \lambda)$ may be checked directly in the same manner, taking into account that
${\rm e}^{-\lambda}\int\nolimits_{-1}^0 {\rm e}^{-\lambda s} s^k {\rm d} s \rightarrow 0$ when
$|\lambda|\rightarrow\infty$ and $\lambda\in \{\lambda:\; \MRe \lambda\ge 0\}$, $k\ge 0$.
\end{proof}

Now we pass over to the proofs of the main propositions mentioned in the beginning of the section.

\begin{proof}[Proof of Lemma~\ref{lm_lm2}]

Let us introduce the following notation:
\begin{equation}\label{eq_eq5}
f(\lambda){\stackrel{\rm def}{=}}\triangle^{-1}(\lambda) D(z, \xi, \lambda)=
\left(\frac{1}{\lambda}\triangle(\lambda)\right)^{-1}
\left(\frac{1}{\lambda} {D}(z, \xi, \lambda)\right), \quad (z, \xi(\cdot))\in M_2^0.
\end{equation}
We analyze the  behavior of the vector-function $f(\lambda)$ near the imaginary axis.
For the points $\lambda_m^k\in \Lambda_1$, which are the eigenvalues of the operator ${\cal A}$,
the inverse to the matrix $\triangle(\lambda_m^k)$ does not exists.
These eigenvalues approach to the imaginary axis when $k\rightarrow\infty$.
Our first aim is to prove that $f(\lambda)$ is bounded in each neighborhood $U(\lambda_m^k)$ of $\lambda_m^k\in \Lambda_1$,
i.e. that the limit $\mathop{\lim}\limits_{\lambda\rightarrow \lambda_m^k}\triangle^{-1}(\lambda) D(z, \xi, \lambda)$ exists
for all $(z, \xi(\cdot))\in M_2^0$.

Since $\triangle(\lambda)$ and $D(z, \xi, \lambda)$ are analytic and since, by the construction, all eigenvalues $\lambda_m^k\in \Lambda_1$
are simple, then we have that if $\lambda_m^k$ is a pole of $f(\lambda)$ then it is a simple pole.
In other words, in every neighborhood $U(\lambda_m^k)$ the vector-function $f(\lambda)$ may be represented as follows:
\begin{equation}\label{eq_eq5asd}
f(\lambda)=\frac{1}{\lambda-\lambda_m^k}f_{-1}+\sum_{i=0}^\infty (\lambda-\lambda_m^k)^i f_i.
\end{equation}
Thus, our aim is to prove that for each $\lambda_m^k$ the coefficient
$f_{-1}=\mathop{\lim}\limits_{\lambda\rightarrow\lambda_m^k} (\lambda-\lambda_m^k) f(\lambda)$ is equal to zero in the representation~(\ref{eq_eq5asd}),
i.e. that $f(\lambda)$ is analytic.
To prove this, we construct a representation of the matrix
$\left(\frac{1}{\lambda}\triangle(\lambda)\right)^{-1}$
which separates the singularity of this matrix.

According to Lemma~\ref{lm_lm5}, for each $\lambda_m^k\in \Lambda_1$ there exist matrices $P_{m,k}$,
$Q_{m,k}$ such that the value of the matrix-function
$\widehat{\triangle}_{m,k}(\lambda)=
\frac{1}{\lambda} P_{m,k} R_m \triangle(\lambda) R_m Q_{m,k}$ at the
point $\lambda=\lambda_m^k$ has the form~(\ref{eq_eq50}), i.e.
$$
\widehat{\triangle}_{m,k}({\lambda_m^k})=
\frac{1}{{\lambda_m^k}} P_{m,k} R_m \triangle({\lambda_m^k}) R_m
Q_{m,k}= \left(
\begin{array}{cc}
0 & \begin{array}{ccc} 0 & \ldots & 0 \end{array} \\
\begin{array}{c}
0\\
\vdots \\
0
\end{array} & S_{m,k}
\end{array}
\right), \quad \det S_{m,k}\not=0.
$$
We rewrite the representation~(\ref{eq_eq5}) of the function
$f(\lambda)$ in a neighborhood $U(\lambda_m^k)$ as follows:
\begin{equation}\label{eq_eq7}
\begin{array}{rcl}
f(\lambda) & = & \left(\frac{1}{\lambda}R_m P_{m,k}^{-1}P_{m,k} R_m
\triangle(\lambda) R_m Q_{m,k} Q_{m,k}^{-1}R_m\right)^{-1}
\left(\frac{1}{\lambda} {D}(z, \xi, \lambda)\right)\\
& = & R_m Q_{m,k}\left(\frac{1}{\lambda}P_{m,k} R_m \triangle(\lambda)
R_m Q_{m,k} \right)^{-1} P_{m,k} R_m \left(\frac{1}{\lambda} {D}(z, \xi, \lambda)\right) \\
& = & R_m Q_{m,k}\left(\widehat{\triangle}_{m,k}(\lambda) \right)^{-1}
P_{m,k} R_m \left(\frac{1}{\lambda} {D}(z, \xi, \lambda)\right).
\end{array}
\end{equation}
Let us consider the Taylor expansion of the analytic matrix-function $\widehat{\triangle}_{m,k}(\lambda)$ in $U(\lambda_m^k)$:
\begin{equation}\label{eq_eq7asd}
\widehat{\triangle}_{m,k}(\lambda)=
\widehat{\triangle}_{m,k}(\lambda_m^k)+
(\lambda-\lambda_m^k)\widehat{\triangle}_{m,k}'(\lambda_m^k)
+\sum_{i=2}^\infty \frac{1}{i!}(\lambda-\lambda_m^k)^i\widehat{\triangle}_{m,k}^{(i)}(\lambda_m^k).
\end{equation}
Due to Corollary~(\ref{col_col1}), $\widehat{\triangle}_{m,k}(\lambda)$ allows the representation~(\ref{eq_eq56}) in some $U(\lambda_m^k)$, i.e.
$$
\widehat{\triangle}_{m,k}(\lambda)= \left(
\begin{array}{cc}
(\lambda-{\lambda_m^k})r_{11}^{m,k}(\lambda) &
\begin{array}{ccc} (\lambda-{\lambda_m^k})r_{12}^{m,k}(\lambda) & \ldots & (\lambda-{\lambda_m^k})r_{1n}^{m,k}(\lambda)
\end{array}\\
\begin{array}{c}
(\lambda-{\lambda_m^k})r_{21}^{m,k}(\lambda)\\
\vdots \\
(\lambda-{\lambda_m^k})r_{n1}^{m,k}(\lambda)
\end{array} & S_{m,k}(\lambda)
\end{array}
\right),
$$
where $r_{ij}^{m,k}(\lambda)$ are analytic functions, and we note that $S_{m,k}(\lambda_m^k)=S_{m,k}$, where $S_{m,k}$ is defined by~(\ref{eq_eq50}).
Differentiating the last relation by $\lambda$ at $\lambda=\lambda_m^k$, we obtain:
$$
\widehat{\triangle}_{m,k}'(\lambda_m^k)= \left(
\begin{array}{cc}
r_{11}^{m,k}(\lambda_m^k) &
\begin{array}{ccc} r_{12}^{m,k}(\lambda_m^k) & \ldots & r_{1n}^{m,k}(\lambda_m^k)
\end{array}\\
\begin{array}{c}
r_{21}^{m,k}(\lambda_m^k)\\
\vdots \\
r_{n1}^{m,k}(\lambda_m^k)
\end{array} & S_{m,k}'(\lambda_m^k)
\end{array}
\right)=\Gamma_{m,k}^0+\Gamma_{m,k}^1,
$$
$$
\Gamma_{m,k}^0{\stackrel{\rm def}{=}}\left(
\begin{array}{cc}
r_{11}^{m,k}(\lambda_m^k) &
\begin{array}{ccc} r_{12}^{m,k}(\lambda_m^k) & \ldots & r_{1n}^{m,k}(\lambda_m^k)
\end{array}\\
\begin{array}{c}
0\\
\vdots \\
0
\end{array} & 0
\end{array}
\right),\quad
\Gamma_{m,k}^1{\stackrel{\rm def}{=}}\left(
\begin{array}{cc}
0 &
\begin{array}{ccc} 0 & \ldots & 0
\end{array}\\
\begin{array}{c}
r_{21}^{m,k}(\lambda_m^k)\\
\vdots \\
r_{n1}^{m,k}(\lambda_m^k)
\end{array} & S_{m,k}'(\lambda_m^k)
\end{array}
\right).
$$
We introduce the matrix-function $F_{m,k}(\lambda){\stackrel{\rm def}{=}}\widehat{\triangle}_{m,k}(\lambda_m^k)+(\lambda-\lambda_m^k)\Gamma_{m,k}^0$,
which has the following structure:
\begin{equation}\label{eq_eq71}
F_{m,k}(\lambda)= \left(
\begin{array}{cc}
r_{11}^{m,k}(\lambda_m^k)(\lambda-\lambda_m^k) &
\begin{array}{ccc} r_{12}^{m,k}(\lambda_m^k)(\lambda-\lambda_m^k) & \ldots & r_{1n}^{m,k}(\lambda_m^k)(\lambda-\lambda_m^k)
\end{array}\\
\begin{array}{c}
0\\
\vdots \\
0
\end{array} & S_{m,k}
\end{array}
\right).
\end{equation}
The matrix $F_{m,k}(\lambda)$ is non-singular in a neighborhood
$U(\lambda_m^k)\backslash\{\lambda_m^k\}$. Indeed, due to Lemma~\ref{lm_lm5} and Corollary~\ref{col_col1}
we have that $\det S_{m,k}\not=0$, $r_{11}^{m,k}(\lambda_m^k)\not=0$, and, thus
$$\det F_{m,k}(\lambda)=r_{11}^{m,k}(\lambda_m^k)(\lambda-\lambda_m^k) \det S_{m,k}\not=0,\quad \lambda\in U(\lambda_m^k)\backslash\{\lambda_m^k\}.$$
Therefore, there exists the inverse matrix $F_{m,k}^{-1}(\lambda)$, which is of the following form:
\begin{equation}\label{eq_eq13}
F_{m,k}^{-1}(\lambda)=
\left(
\begin{array}{cccc}
\frac{1}{r_{11}^{m,k}(\lambda_m^k)(\lambda-\lambda_m^k)} & F_{21}^{m,k} & \ldots & F_{n1}^{m,k} \\
0 &  F_{22}^{m,k}  & \ldots &  F_{n2}^{m,k} \\
\vdots & \vdots  & \ddots & \vdots \\
0 &  F_{2n}^{m,k}  & \ldots &  F_{nn}^{m,k}
\end{array}
\right),
\end{equation}
where
\begin{equation}\label{eq_eq13sd}
\begin{array}{rclr}
F_{i1}^{m,k} & = & \frac{1}{r_{11}^{m,k}(\lambda_m^k) \det S_{m,k}} \sum\limits_{j=2}^n (-1)^{i+j} r_{1j}^{m,k}(\lambda_m^k) [S_{m,k}(\lambda_m^k)]_{ij},
& i=2,\ldots,n,\\
F_{ij}^{m,k} & = & (-1)^{i+j}[S_{m,k}(\lambda_m^k)]_{ij}, &
i,j=2,\ldots,n,
\end{array}
\end{equation}
and by $[S_{m,k}(\lambda)]_{ij}$ we denote the complementary minor of the element $s_{ij}^{m,k}(\lambda)$, $i,j=2,\ldots,n$
of the matrix $S_{m,k}(\lambda)$. Since the matrix-functions $S_{m,k}(\lambda)$
are analytic and since $S_{m,k}(\lambda_m^k)\rightarrow S$ when $k\rightarrow\infty$,
then $\|S_{m,k}(\lambda)\|$, $\|[S_{m,k}(\lambda)]_{ij}\|$ and $|s_{ij}^{m,k}(\lambda)|$ are uniformly bounded for all $k$
and $\lambda\in U_\delta(\widetilde{\lambda}_m^k)$.
Thus, we conclude that $|F_{ij}^{m,k}|\le C$ for all $k$ and $i,j=2,\ldots,n$.
Moreover, since $r_{1j}^{m,k}(\lambda_m^k)\rightarrow 0$ due to Remark~\ref{rm_rm_r0},
then $F_{i1}^{m,k} \rightarrow 0$, $i=2,\ldots,n$ when $k\rightarrow\infty$.

Let us rewrite the representation~(\ref{eq_eq7asd}) as follows:
\begin{equation}\label{eq_eq8}
\begin{array}{rcl}
\widehat{\triangle}_{m,k}(\lambda) & = & F_{m,k}(\lambda)+(\lambda-\lambda_m^k)\Gamma_{m,k}^1+
\sum\limits_{i=2}^\infty \frac{1}{i!}(\lambda-\lambda_m^k)^i\widehat{\triangle}_{m,k}^{(i)}(\lambda_m^k) \\
& = & F_{m,k}(\lambda)\left(I+(\lambda-\lambda_m^k)F_{m,k}^{-1}(\lambda)\Gamma_{m,k}^1
+ \sum\limits_{i=2}^\infty \frac{1}{i!}(\lambda-\lambda_m^k)^i F_{m,k}^{-1}(\lambda)\widehat{\triangle}_{m,k}^{(i)}(\lambda_m^k)\right)
\end{array}
\end{equation}
and introduce the notation
$$
\Upsilon_{m,k}(\lambda){\stackrel{\rm def}{=}}(\lambda-\lambda_m^k)F_{m,k}^{-1}(\lambda)\Gamma_{m,k}^1+
\sum\limits_{i=2}^\infty \frac{1}{i!}(\lambda-\lambda_m^k)^i F_{m,k}^{-1}(\lambda)\widehat{\triangle}_{m,k}^{(i)}(\lambda_m^k).
$$
Let us prove that for any $\varepsilon>0$ there exist $\delta>0$ and  $N\in \mathbb{N}$
such that for any $k:\;|k|>N$ the following estimate holds:
\begin{equation}\label{eq_eq72}
\|\Upsilon_{m,k}(\lambda)\|\le \varepsilon, \qquad \lambda\in U_\delta(\widetilde{\lambda}_m^k).
\end{equation}

From (\ref{eq_eq8}) we have that $\Upsilon_{m,k}(\lambda)=F_{m,k}^{-1}(\lambda)\widehat{\triangle}_{m,k}(\lambda)-I$,
and we are proving the estimate $\|F_{m,k}^{-1}(\lambda)\widehat{\triangle}_{m,k}(\lambda)-I\|\le \varepsilon$,
$\lambda\in U_\delta(\widetilde{\lambda}_m^k)$.
Using the representations~(\ref{eq_eq56}), (\ref{eq_eq13}) and (\ref{eq_eq13sd}),
we estimate the elements $\{\gamma_{ij}^{m,k}(\lambda)\}_{i,j=1}^n$
of the matrix $\Upsilon_{m,k}(\lambda)$.
For the sake of convenience, we divide these elements onto several groups
and we begin with the element $\gamma_{11}^{m,k}(\lambda)$:
$$
\gamma_{11}^{m,k}(\lambda) = \frac{r_{11}^{m,k}(\lambda)}{r_{11}^{m,k}(\lambda_m^k)}
 + (\lambda-\lambda_m^k)\sum_{i=2}^n F_{i1}^{m,k} r_{i1}^{m,k}(\lambda)-1.
$$
Due to Corollary~\ref{col_col1} and since $r_{11}^{m,k}(\lambda)$ is analytic, there exist $\delta>0$ and  $N\in \mathbb{N}$
such that $\left|\frac{r_{11}^{m,k}(\lambda)}{r_{11}^{m,k}(\lambda_m^k)}\right|<\frac{\varepsilon}{2n}$ for all $k:\;|k|>N$
and $\lambda\in U_\delta(\widetilde{\lambda}_m^k)$. Besides, since
$F_{i1}^{m,k} \rightarrow 0$, $i=2,\ldots,n$ and $r_{i1}^{m,k}(\lambda_m^k)\rightarrow 0$ when $k\rightarrow\infty$,
we obtain the estimate $\left|\gamma_{11}^{m,k}(\lambda)\right|<\frac{\varepsilon}{n}$ for all $k:\;|k|>N$,
$\lambda\in U_\delta(\widetilde{\lambda}_m^k)$.

Let us consider other diagonal elements of the matrix $\Upsilon_{m,k}(\lambda)$:
$$
\gamma_{jj}^{m,k}(\lambda) = \sum_{i=2}^n F_{ij}^{m,k} s_{ij}^{m,k}(\lambda)-1
=  \sum_{i=2}^n (-1)^{i+j} [S_{m,k}(\lambda)]_{ij} (s_{ij}^{m,k}(\lambda)-s_{ij}^{m,k}(\lambda_m^k)),\quad j=2,\ldots,n.
$$
There exists $\delta>0$ such that $\left|\gamma_{jj}^{m,k}(\lambda)\right|<\frac{\varepsilon}{n}$ for all $k:\;|k|>N$,
$\lambda\in U_\delta(\widetilde{\lambda}_m^k)$.
Further, we consider the elements of the first row:
$$
\gamma_{1j}^{m,k}(\lambda)= \frac{r_{1j}^{m,k}(\lambda)}{r_{11}^{m,k}(\lambda_m^k)} +
\sum_{i=2}^n F_{i1}^{m,k} s_{ij}^{m,k}(\lambda), \quad j=2,\ldots,n.
$$
Since $F_{i1}^{m,k} \rightarrow 0$, $i=2,\ldots,n$ and $r_{1j}^{m,k}(\lambda_m^k)\rightarrow 0$ when $k\rightarrow\infty$,
we obtain the estimate $\left|\gamma_{1j}^{m,k}(\lambda)\right|<\frac{\varepsilon}{n}$ for all $k:\;|k|>N$,
$\lambda\in U_\delta(\widetilde{\lambda}_m^k)$.

Finally we consider all other elements:
$$
\gamma_{ij}^{m,k}(\lambda)=\sum_{r=2}^n F_{ri}^{m,k} s_{rj}^{m,k}(\lambda)
= \sum_{r=2}^n (-1)^{i+r} [S_{m,k}(\lambda)]_{ir} (s_{rj}^{m,k}(\lambda)-s_{rj}^{m,k}(\lambda_m^k)),
\; i,j=2,\ldots,n; i\not=j.
$$
They may be estimated as $\left|\gamma_{ij}^{m,k}(\lambda)\right|<\frac{\varepsilon}{n}$ for all $k:\;|k|>N$,
$\lambda\in U_\delta(\widetilde{\lambda}_m^k)$ choosing small enough $\delta>0$.

Finally, we obtain the estimate~(\ref{eq_eq72}) and, therefore, there exist $\delta>0$, $N\in\mathbb{N}$
such that the matrix $I+\Upsilon_{m,k}(\lambda)$ has an inverse for any
$\lambda\in U_\delta(\widetilde{\lambda}_m^k)$, $k:|k|>N$:
\begin{equation}\label{eq_eq9}
(I+\Upsilon_{m,k}(\lambda))^{-1}=I+(\lambda-\lambda_m^k)\Gamma_{m,k}(\lambda),
\end{equation}
where $\Gamma_{m,k}(\lambda)$ is analytic in a neighborhood $U_\delta(\lambda_m^k)$.

Finally, from (\ref{eq_eq7}), (\ref{eq_eq8}) and (\ref{eq_eq9}) we
obtain:
\begin{equation}\label{eq_eq10}
\begin{array}{rcl}
f(\lambda) & = & R_m Q_{m,k} \widehat{\triangle}_{m,k}^{-1}(\lambda)
P_{m,k} R_m \left(\frac{1}{\lambda} {D}(z, \xi, \lambda)\right)\\
& = & R_m Q_{m,k}
\left(F_{m,k}(\lambda)(I+\Upsilon_{m,k}(\lambda))\right)^{-1} R_m P_{m,k} \left(\frac{1}{\lambda} {D}(z, \xi, \lambda)\right)\\
& = & R_m Q_{m,k} \left (I+(\lambda-\lambda_m^k)\Gamma_{m,k}(\lambda)\right)F_{m,k}^{-1}(\lambda) P_{m,k}R_m
\left(\frac{1}{\lambda} {D}(z, \xi, \lambda)\right).
\end{array}
\end{equation}
Let us use the fact that $(z, \xi(\cdot))\in M_2^0$.
In Lemma~\ref{lm_lm1} the important relation ${D}(z, \xi, \lambda_m^k)\in \MIm {\triangle}(\lambda_m^k)$, $(z, \xi(\cdot))\in M_2^0$ is stated.
Thus, due to Proposition~\ref{ut_u11}, we conclude that
\begin{equation}\label{eq_eq73}
\frac{1}{\lambda_m^k} P_{m,k} R_m D(z, \xi, \lambda_m^k)\in \MIm
\widehat{\triangle}_{m,k}(\lambda_m^k).
\end{equation}

Moreover, since the matrix $\widehat{\triangle}_{m,k}(\lambda_m^k)$ is of
the form~(\ref{eq_eq50}), we conclude that the first component of the vector $\frac{1}{\lambda_m^k} P_{m,k} R_m D(z, \xi, \lambda_m^k)$
equals to zero:
\begin{equation}\label{eq_eq11}
\frac{1}{\lambda_m^k} P_{m,k} R_m D(z, \xi, \lambda_m^k)\;{\stackrel{\rm def}{=}}\; \widehat{d}_{m,k}=(0, c_2, \ldots, c_n)^T,
\end{equation}
and since the vector-function $\frac{1}{\lambda} P_{m,k} R_m D(z, \xi, \lambda)$ is analytic in a
neighborhood $U(\lambda_m^k)$, we conclude that
\begin{equation}\label{eq_eq12}
\frac{1}{\lambda} P_{m,k} R_m D(z, \xi, \lambda)= \widehat{d}_{m,k} + d_{m,k}(\lambda),\quad d_{m,k}(\lambda_m^k)=0.
\end{equation}
Finally, we note that  $F_{m,k}^{-1}(\lambda) \widehat{d}_{m,k}$ is a constant vector
and the vector-function $F_{m,k}^{-1}(\lambda)d_{m,k}(\lambda)$ is bounded in $U(\lambda_m^k)$.
Taking into account (\ref{eq_eq10}), (\ref{eq_eq13}) and (\ref{eq_eq12}), we obtain that
\begin{equation}\label{eq_eq14}
\begin{array}{rcl}
\lim\limits_{\lambda\rightarrow\lambda_m^k}(\lambda-\lambda_m^k)f(\lambda) & = &
R_m Q_{m,k}\lim\limits_{\lambda\rightarrow\lambda_m^k} (\lambda-\lambda_m^k)  F_{m,k}^{-1}(\lambda)
P_{m,k} R_m \left(\frac{1}{\lambda} {D}(z, \xi, \lambda)\right) \\
& = & R_m Q_{m,k} \lim\limits_{\lambda\rightarrow\lambda_m^k} (\lambda-\lambda_m^k) F_{m,k}^{-1}(\lambda) (\widehat{d}_{m,k} + d_{m,k}(\lambda))\\
& = & 0.
\end{array}
\end{equation}

Thus, we have proved that $f(\lambda)=\triangle^{-1}(\lambda) D(z,
\xi, \lambda)$ is an analytic vector-function in $U_\delta(\lambda_m^k)$, $k: |k|>N$, $m=1,\ldots,\ell_1$ what
gives the estimate $\left\|\triangle^{-1}(\lambda) D(z, \xi, \lambda)\right\|\le C_{m,k}$, $(z, \xi(\cdot))\in M_2^0$.

It remains to prove that  $f(\lambda)$ is uniformly bounded in the
neighborhoods $U_\delta(\lambda_m^k)$ for all $k: |k|>N$, $m=1,\ldots,\ell_1$. In other words,
our next aim is to prove that the set of vectors
$$
f_0=f_0^{m,k}=f(\lambda_m^k)=(\triangle^{-1}(\lambda) D(z, \xi,
\lambda))_{\lambda=\lambda_m^k}
$$
is bounded. Taking into account the representation (\ref{eq_eq10}),
we obtain:
\begin{equation}\label{eq_eq15}
\begin{array}{rcl}
f_0^{m,k} & = & \left(R_m Q_{m,k} F_{m,k}^{-1}(\lambda) P_{m,k} R_m \left(\frac{1}{\lambda} D(z, \xi, \lambda)\right)\right)_{\lambda=\lambda_m^k} \\
& = & R_m Q_{m,k} F_{m,k}^{-1}(\lambda) (\widehat{d}_{m,k} + d_{m,k}(\lambda))_{\lambda=\lambda_m^k}\\
& = & R_m Q_{m,k}\left(
\frac{d_1^1}{r_{11}^{m,k}(\lambda_m^k)} + \sum\limits_{i=2}^n c_i F_{i1}^{m,k},\;
\sum\limits_{i=2}^n c_i F_{i2}^{m,k},\ldots,\; \sum\limits_{i=2}^n c_i F_{in}^{m,k}\right)^T,
\end{array}
\end{equation}
where $d_1^1=(d'_{m,k}(\lambda_m^k))_1$ is the first component of the derivative of $d_{m,k}(\lambda)$ at the point $\lambda_m^k$.

As we have mentioned above, there exists a constant $C_1>0$ such that $\|F_{ij}^{m,k}\|\le C_1$, for all $k\in\mathbb{N}$, $m=1,\ldots,\ell_1$.
The estimates $\|P_{m,k}\|\le C_1$ and $\|Q_{m,k}\|\le C_1$ follow from the estimate~(\ref{eq_eq51}) of Lemma~\ref{lm_lm6}.
The estimates $|c_i|<C_1$ and $|d_1^1|<C_1$ follow immediately from Lemma~\ref{lm_lm3}.
From the relation~(\ref{eq_eq57}) of the Corollary~\ref{col_col1} it follows that there exists a constant $C_2>0$ such that
$0<C_2\le |r_{11}^{m,k}(\lambda_m^k)|$ for all $k\in\mathbb{N}$, $m=1,\ldots,\ell_1$ and, thus,
$$
\frac{1}{|r_{11}^{m,k}(\lambda_m^k)|}\le \frac{1}{C_2}.
$$

Finally, we conclude that $\|f_0^{m,k}\|\le C$ for all $m=1,\ldots,\ell_1$, $k: |k|\ge N$,
what completes the proof of the lemma.
\end{proof}

\begin{proof}[Proof of Lemma~\ref{lm_lm1}]
Since $g\in M_2^0$ then $g\bot\psi_m^k=\psi(\overline{\lambda_m^k})$ for all $\lambda_m^k\in \Lambda_1$.
Therefore, the proposition follows from Lemma~\ref{lm_lm1_plelim}.
\end{proof}

\begin{proof}[Proof of Theorem~\ref{thr_boundedness}]
Let $\delta>0$ be such that Lemma~\ref{lm_lm2} holds.
We divide the closed right half-plane onto the following two sets:
\begin{equation}\label{eq_eq155a}
\begin{array}{l}
K_1(\delta)=\{\lambda:\;\MRe\lambda \ge 0,\; \lambda\in U_\delta(\widetilde{\lambda}_m^k),\; \lambda_m^k\in\Lambda_1\}\\
K_2(\delta)=\{\lambda:\; \MRe\lambda \ge 0\}\backslash K_1.
\end{array}
\end{equation}

First, let us estimate $\|R(\lambda, {\cal A})x\|$ for
any $\lambda\in K_1(\delta)$ and $x\in M_2^0$. Due to
Lemma~\ref{lm_lm2} we have: $\|\triangle^{-1}(\lambda) D(z, \xi,
\lambda)\|\le C_1\|x\|$, $x=(z, \xi(\cdot))\in M_2^0$.
Due to Corollary~\ref{col_col2} we have the estimate $\|\int_{-1}^0 {\rm e}^{-\lambda s} \xi(s)\dd s\|\le C_2\|x\|$.
Thus, for any $x=(z, \xi(\cdot))\in M_2^0$, $\lambda\in K_1(\delta)$ we obtain:
\begin{equation}\label{eq_eq155avc}
\begin{array}{rcl}
\|R(\lambda, {\cal A})x\| & = & \left\|{\rm e}^{-\lambda}A_{-1}\int\limits_{-1}^0 {\rm e}^{-\lambda s} \xi(s)\dd s
+ (I-{\rm e}^{-\lambda}A_{-1}) \triangle^{-1}(\lambda) D(z, \xi, \lambda)\right\|_{\mathbb{C}^n} \\
& & +\left\|\int\limits_0^\theta {\rm e}^{\lambda(\theta- s)} \xi(s) \dd s
+ {\rm e}^{\lambda\theta}\triangle^{-1}(\lambda) D(z, \xi, \lambda)\right\|_{L_2}\\
& \le & {\rm e}^{\delta} \|A_{-1}\| C_2\|x\| +(1+{\rm e}^{\delta} \|A_{-1}\|)C_1\|x\|\\
& & +\left(\int\limits_{-1}^0 \left\|\int\limits_0^\theta {\rm e}^{\lambda(\theta- s)} \xi(s) \dd s
+ {\rm e}^{\lambda\theta}\triangle^{-1}(\lambda) D(z, \xi, \lambda)\right\|_{\mathbb{C}^n}^2 \dd \theta\right)^{\frac{1}{2}}\\
& \le & \left[{\rm e}^{\delta} \|A_{-1}\| C_2 +(1+{\rm e}^{\delta} \|A_{-1}\|)C_1 +
({\rm e}^{\delta} C_2+ C_1^2)^{\frac{1}{2}}\right]\|x\|= C\|x\|.
\end{array}
\end{equation}

Let us consider $\lambda\in K_2(\delta)$. There exists $\varepsilon>0$
such that $\left|\frac{1}{\lambda}\det\triangle(\lambda)\right|\ge \varepsilon$ for any $\lambda\in K_2(\delta)\backslash U(0)$.
Indeed, if we suppose the contrary then there exists a sequence $\{\lambda_i\}_{i=1}^\infty$ such that
$\left|\frac{1}{\lambda_i}\det\triangle (\lambda_i)\right| \rightarrow 0$, $i\rightarrow\infty$.
If the sequence $\{\lambda_i\}_{i=1}^\infty$ is bounded,
then it possesses a converging subsequence: $\lambda_{i_j}\rightarrow \widehat{\lambda}$
and, thus, $\left|\frac{1}{\widehat{\lambda}}\det\triangle (\widehat{\lambda})\right|=0$.
However, the closure of the set $K_2(\delta)\backslash U(0)$ does not contain zeroes of the function $\det\triangle (\lambda)$
and we have obtained a contradiction.

If the sequence $\{\lambda_i\}_{i=1}^\infty$ is unbounded, i.e. $|\lambda_i| \rightarrow \infty$ when $i \rightarrow \infty$,
then we have $\int_{-1}^0 {\rm e}^{\lambda_i s} A_2(s) \dd s \rightarrow 0$
and $\frac{1}{\lambda_i}\int_{-1}^0 {\rm e}^{\lambda_i s} A_3(s) \dd s\rightarrow 0$, $i \rightarrow \infty$.
Moreover, since $|\det (-I + {\rm e}^{-\lambda} A_{-1})|=|\prod (1+ {\rm e}^{-\lambda}\mu_m)|\ge \prod |1+{\rm e}^{\delta}\mu_m|$,
we conclude that
$$\left|\frac{1}{\lambda}\det\triangle(\lambda_i)\right|
=\det (-I + {\rm e}^{-\lambda_i} A_{-1} + \int_{-1}^0 {\rm e}^{\lambda_i s} A_2(s) \dd s
+ \frac{1}{\lambda_i}\int_{-1}^0 {\rm e}^{\lambda_i s} A_3(s) \dd s \not\rightarrow 0, \quad i \rightarrow \infty.
$$
Thus, we have obtained a contradiction again.

Taking into account the estimates $\|\frac{1}{\lambda} \triangle (\lambda)\|\le C_1$ and
$\left\|\frac{1}{\lambda} D(z, \xi, \lambda)\right\|\le C_2\|x\|$ from Lemma~\ref{lm_lm3},
we conclude that $\|\triangle^{-1}(\lambda) D(z, \xi, \lambda)\|\le C_3\|x\|$ for all $\lambda\in K_2(\delta)\backslash U(0)$.
It is easy to see that
$\|{\rm e}^{-\lambda} \int_{-1}^0 {\rm e}^{-\lambda s} \xi(s)\dd s\|\le C_4\|x\|$ for
all $\lambda\in K_2(\delta)\backslash U(0)$. Finally, similarly to~(\ref{eq_eq155avc}) we obtain the following estimate
$$\|R(\lambda, {\cal A})x\|\le C\|x\|,\quad \lambda\in K_2.$$
The last completes the proof of the theorem.
\end{proof}

\section{Stability analysis}

Basing on the results from Section~\ref{sect_decomposition} and Section~\ref{sect_resolvent},
we prove the main result on stability which does not assume the condition $\det A_{-1}\not=0$.

\begin{thr}[on stability]\label{thr_stability}
If $\sigma({\cal A})\subset \{\lambda:\; \MRe \lambda<0\}$ and
$\sigma_1=\sigma(A_{-1})\cap \{\mu: |\mu|= 1 \}$ consists only of
simple eigenvalues, then system (\ref{syst_delay_gen}) is strongly
asymptotically stable.
\end{thr}

\begin{proof} Let us show that $\|{\rm e}^{t{\cal A}}x\|\rightarrow 0$ when $t\rightarrow+\infty$ for any
$x\in M_2$.
Due to Theorem~\ref{thr_decomposition} each $x\in M_2$
allows the following representation:
$$
x=x_0+x_1, \quad x_0\in M_2^0, \; x_1\in M_2^1,
$$
where $M_2^0$ and $M_2^1$ are defined by (\ref{eq_eq32})--(\ref{eq_eq23})
Moreover, the basis of $M_2^1$ consists of the following eigenvectors:
\begin{equation}\label{eq_eq70}
\{\varphi_m^k:\; ({\cal A}-\lambda_m^k I)\varphi_m^k=0,\;
\lambda_m^k\in\Lambda_1=\Lambda_1(N)\}.
\end{equation}
Thus, for any $x_1\in M_2^1$ we have the representation
$$
x_1=\sum\limits_{m=1}^{\ell_1}\sum\limits_{|k|\ge N} c_m^k \varphi_m^k, \qquad
{\rm e}^{t{\cal A}}x_1=\sum\limits_{m=1}^{\ell_1}\sum\limits_{|k|\ge N} {\rm e}^{\lambda_m^k
t}c_m^k \varphi_m^k, \qquad \sum\limits_{m=1}^{\ell_1}\sum\limits_{|k|\ge N}
|c_m^k|^2<\infty.
$$
Let us consider a norm $\|\cdot\|_1$ in which the Riesz
basis~(\ref{eq_eq70}) is orthogonal, then we have the following
estimate:
\begin{equation}\label{eq_eq35}
\|{\rm e}^{t{\cal A}}x_1\|_1=\left(\sum\limits_{m=1}^{\ell_1}\sum\limits_{|k|\ge N}
{\rm e}^{2\MRe\lambda_m^k t} \|c_m^k \varphi_m^k\|_1^2\right)^{\frac12}\le
\|x_1\|_1.
\end{equation}

Since the series $\sum\limits_{m=1}^{\ell_1}\sum\limits_{|k|\ge N} c_m^k
\varphi_m^k$ converges and since $\|\varphi_m^k\|_1\le C$ for all $k$ and $m=1,\ldots,\ell_1$,
then for any $\varepsilon>0$ there exists $N_1\ge N$ such
that $\sum\limits_{m=1}^{\ell_1}\sum\limits_{|k|\ge N_1} \|c_m^k \varphi_m^k\|_1^2\le \frac{\varepsilon^2}{8}$.
Moreover, since the set $\{(m, k):\; m=1,\ldots,\ell_1,\; N\le |k|\le
N_1\}$ is finite and since $\MRe\lambda_m^k<0$, then there exists
$t_0>0$ such that for any $t\ge t_0$ we have:
$\sum\limits_{m=1}^{\ell_1}\sum\limits_{N\le |k|\le N_1}
{\rm e}^{2\MRe\lambda_m^k t} \|c_m^k \varphi_m^k\|_1^2\le
\frac{\varepsilon^2}{8}$.
Thus, we obtain
\begin{equation}\label{eq_eq41}
\sum\limits_{m=1}^{\ell_1}\sum\limits_{|k|\ge N} {\rm e}^{2\MRe\lambda_m^k t} \|c_m^k
\varphi_m^k\|_1^2\le \sum\limits_{m=1}^{\ell_1}\sum\limits_{N\le |k|\le N_1}
{\rm e}^{2\MRe_m^k\lambda t} \|c_m^k \varphi_m^k\|_1^2+\sum\limits_{m=1}^{\ell_1}\sum\limits_{|k|\ge N_1} \|c_m^k \varphi_m^k\|_1^2 \le
\frac{\varepsilon^2}{4}.
\end{equation}

Due to Theorem~\ref{thr_boundedness} the semigroup ${\rm e}^{t{\cal
A}}|_{M_2^0}$ is exponentially stable, i.e. by definition there
exist some positive constants $M$, $\omega$ such that  $\|{\rm e}^{t{\cal
A}}|_{M_2^0}\|\le M {\rm e}^{-\omega t}$. Thus, for any $x_0\in M_2^0$
there exists $t_0>0$ such that for any $t\ge t_0$ we have an
estimate
\begin{equation}\label{eq_eq40}
\|{\rm e}^{t{\cal A}}x_0\|_1\le M {\rm e}^{-\omega t} \|x_0\|_1\le
\frac{\varepsilon}{2}.
\end{equation}

Finally, from the estimates (\ref{eq_eq35}), (\ref{eq_eq41}) and
(\ref{eq_eq40}) we conclude that for any $x\in M_2$ and for any
$\varepsilon>0$ there exists $t_0>0$ such that for any $t\ge t_0$
the following estimate holds:
\begin{equation}\label{eq_eq42}
\|{\rm e}^{t{\cal A}}x\|_1\le \|{\rm e}^{t{\cal A}}x_0\|_1+ \|{\rm e}^{t{\cal
A}}x_1\|_1\le \varepsilon.
\end{equation}

Therefore, $\lim\limits_{t\rightarrow +\infty}\|{\rm e}^{t{\cal
A}}x\|_1=0$, what implies that the system~(\ref{syst_delay_opmodel})
is strongly asymptotically stable.
\end{proof}

\subsection{An example of dilemma: stable and unstable situations}

In this subsection we give an explicit example illustrating the item~(iii) of Theorem~\ref{thr_stability_intro}.
Namely, we construct two systems having the same spectrum and satisfying the following conditions:
$\sigma({\cal A})\subset\{\lambda: \MRe \lambda<0\}$ and there are no Jordan blocks,
corresponding to eigenvalues from $\sigma_1=\sigma(A_{-1})\cap \{\mu: |\mu|= 1 \}$, but there exists
an eigenvalue $\mu\in \sigma_1$ whose eigenspace is at least two dimensional.
Moreover, one of the constructed systems appears to be stable while the other is unstable.

We consider the system of the form
\begin{equation}\label{syst_delay_ex}
\dot{z}(t)=\left(
\begin{array}{rr}
-1 & 0\\
0 & -1
\end{array}
\right)\dot{z}(t-1)
+ \left(
\begin{array}{rr}
-b & s\\
0 & -b
\end{array}
\right) z(t),\qquad z\in\mathbb{C}^2,\; t\ge 0,
\end{equation}
where $b$ is a real positive number and for the value of $s$ we essentially distinguish two cases: $s=0$ and $s\not=0$.

\begin{rem}\label{rm_1}
The systems of the form $\dot{z}=A_{-1}\dot{z}(t-1)+A_0 z(t)$ is a special case of the systems~(\ref{syst_delay_gen}).
\end{rem}
\begin{proof}
Indeed, we choose $A_2(\theta)=(\theta+1)A_0$ and $A_3(\theta)=A_0$. Thus,
$$
\int_{-1}^0 A_2(\theta)\dot{z}(t+\theta)\;{\rm d}\theta +\int_{-1}^0
A_3(\theta){z}(t+\theta)\;{\rm d}\theta= \int_{-1}^0 (\theta+1)A_0
\dot{z}(t+\theta)\;{\rm d}\theta +\int_{-1}^0 A_0{z}(t+\theta)\;{\rm
d}\theta$$
$$
=A_0\int_{-1}^0  ((\theta+1) z(t+\theta))'\;{\rm d}\theta=A_0 z(t).
$$
\end{proof}

The eigenvalues of the operator $\cal A$ are the roots of the equation
$\det\triangle_{\cal A}(\lambda)=0$, which, in our particular case, has the form:
$$
\det(-\lambda I+\lambda {\rm e}^{-\lambda}A_{-1}+A_0) =\det\left(
\begin{array}{cc}
-\lambda - \lambda {\rm e}^{-\lambda} - b & s\\
0 & -\lambda - \lambda {\rm e}^{-\lambda} - b
\end{array}
\right)=0.
$$
Thus, all the eigenvalues of the operator $\cal A$ satisfy the
equation
\begin{equation}\label{ex_eq1}
\lambda {\rm e}^{\lambda} + \lambda + b {\rm e}^{\lambda}=0
\end{equation}
and the multiplicity of each eigenvalue equals two.
To prove that $\sigma({\cal A})\subset\{\lambda: \MRe \lambda<0 \}$,
we use the results on transcendental equations obtained by L.~Pontryagin \cite{Pontryagin_function_zeros}.

Let us consider the equation $H(z)=0$, where $H(z)=h(z, {\rm e}^z)$ is a
polynomial with respect to $z$ and ${\rm e}^z$.

\begin{df}\label{df_1}
We say that the function $H(z)=\sum\limits_{m,n} a_{mn}z^m {\rm e}^{nz}$
possesses the principal term $a_{rs}z^r {\rm e}^{sz}$, if for all other
terms $a_{mn}z^m {\rm e}^{nz}$ we have that $r\ge m$ and $s\ge n$.
\end{df}

We denote by $F(y): \mathbb{R}\rightarrow \mathbb{R}$ and $G(y): \mathbb{R}\rightarrow \mathbb{R}$,
correspondingly the real and the imaginary parts of the function $H(iy)$, i.e. $H(iy)=F(y)+iG(y)$, $y\in\mathbb{R}$.

\begin{df}\label{df_2}
We say that the zeroes of two real-valued functions of a real
variable alternate if and only if

\noindent a) each of these functions has no multiple roots;

\noindent b) between any two zeroes of one of these functions there exists
at least one zero of the other;

\noindent c) the functions are never simultaneously zero.
\end{df}

\begin{thr}[{\cite[Pontryagin, 1942]{Pontryagin_function_zeros}}]\label{pontr_th7}
Let $H(z)=h(z, {\rm e}^z)$ be a polynomial with a principal term.

\noindent
1. If all the zeroes of the function $H(z)$ belong to the open left half-plane: $\MRe \lambda_k <0$,
then the zeroes of $F(y)$ and $G(y)$ are real, alternating and for all $y\in
\mathbb{R}$ the following inequality holds
\begin{equation}\label{equation_1}
G'(y)F(y)-G(y)F'(y)>0.
\end{equation}

\noindent
2. Any of the conditions below is sufficient for all the zeroes of
the function $H(z)$ to lie in the open left half-plane:

a) all the zeroes of the functions $F(y)$ and $G(y)$ are real,
alternating and the inequality (\ref{equation_1}) is satisfied for
at least one value of $y$;

b) all the zeroes of the function $F(y)$ are real and for each zero
$y=y_0$ the inequality (\ref{equation_1}) is satisfied, i.e.
$G(y_0)F'(y_0)<0$;

c) all the zeroes of the function $G(y)$ are real and for each zero
$y=y_0$ the inequality (\ref{equation_1}) is satisfied, i.e.
$G'(y_0)F(y_0)>0$.
\end{thr}

The following theorem gives a criterion for all zeroes of a function to be real.

\begin{thr}[{\cite[Pontryagin, 1942]{Pontryagin_function_zeros}}]\label{pontr_th3}
Let $F(z)=f(z, \cos z, \sin z)$ be a polynomial with a principal
term $z^r\phi^{(s)}_m(\cos z, \sin z)$, where $\phi^{(s)}_m(\cos z,
\sin z)$ is homogeneous with respect to $\cos z$ and $\sin z$
polynomial.

The function $F(z)$, $z\in \mathbb{C}$ possesses only real zeroes if
and only if for all big enough $k\in \mathbb{Z}$ the function
$F(x)$, $x\in \mathbb{R}$ possesses exactly $4ks+r$ real roots on
the interval $-2\pi k+\varepsilon\le x \le 2\pi k+\varepsilon$ for
some $\varepsilon>0$.
\end{thr}

We use the results mentioned above to analyze the location of the roots of the equation~(\ref{ex_eq1}).

\begin{ut}\label{ut_u14}
For any $b>0$ and any $s\in \mathbb{C}$ the spectrum of the
corresponding operator ${\cal A}$ belongs to the open left half-plane.
\end{ut}

\begin{proof} We represent $H(iy)=iy {\rm e}^{iy} + iy + b {\rm e}^{iy}$ as follows:
$$
\begin{array}{rcl}
H(iy) & = & iy(\cos y + i\sin y)+iy+b(\cos y+i \sin y)\\
& = & (b \cos y - y \sin y) + i (y\cos y + y + è \sin y)\\
& = & F(y)+iG(y).
\end{array}
$$
Since the roots of the equations $y=0$ and $\cos y=0$ are not roots of the equation $F(y)=0$,
we rewrite $F(y)=0$ as $\tg y = \frac{b}{y}$.
Several zeroes of the equation $F(y)=0$ may be seen on the figure below:

\begin{center}
\includegraphics[scale=0.7]{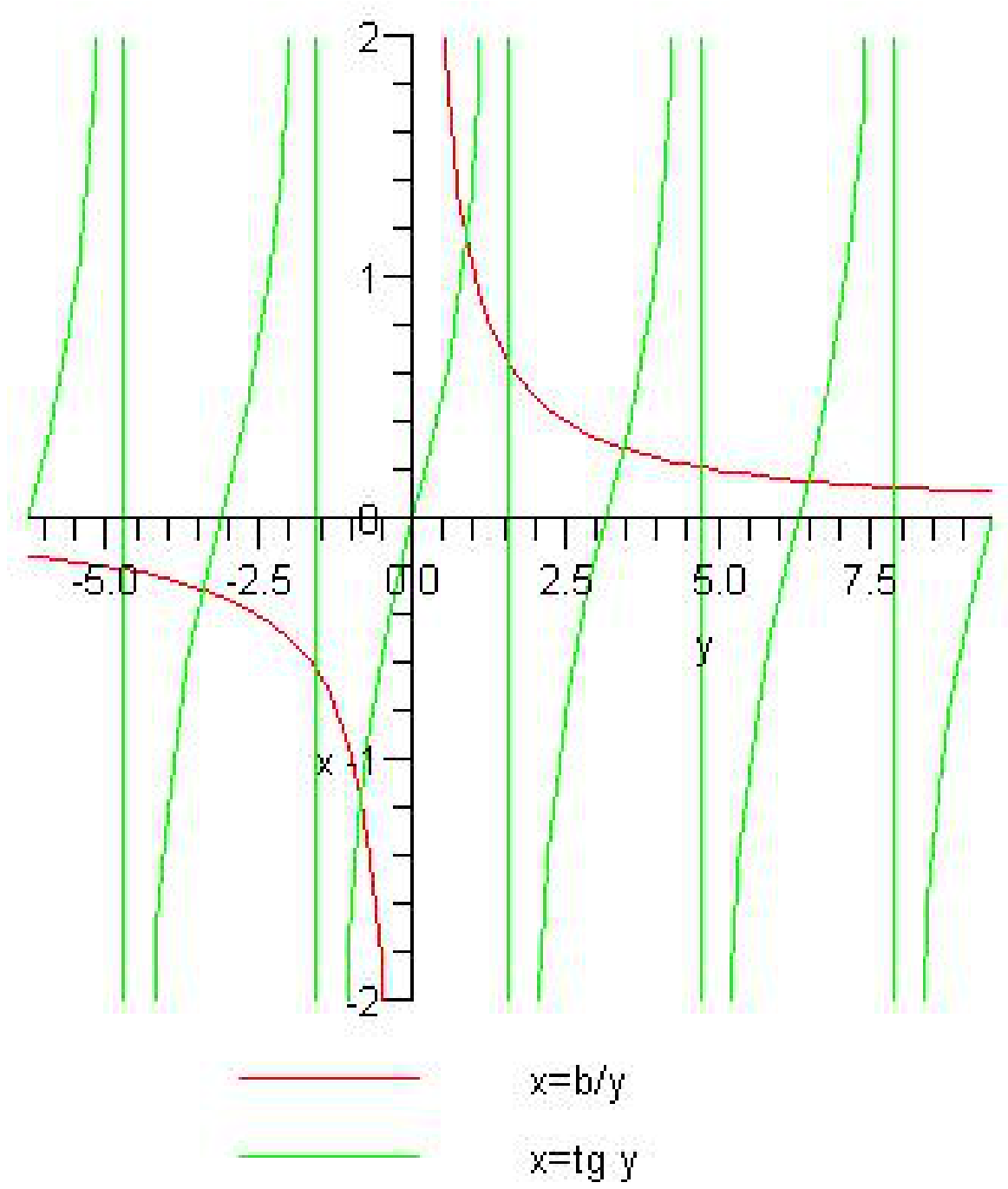}
\vskip 1ex
Figure 3.
\end{center}

We see that exactly 5 zeroes of the equation $F(y)=0$ belong to the interval $-2\pi+\varepsilon\le y\le
2\pi+\varepsilon$ if we choose $\varepsilon>0$ "big" enough such
that the zero on the right from $2\pi$ belongs to the interval.
Adding to this interval from the left and from the right $2\pi$ we
add each time 4 other zeroes. Thus, on each interval $-2\pi
k+\varepsilon\le y\le 2\pi k+\varepsilon$ we have exactly $4k+1$
zeroes of the equation $F(y)=0$ and, therefore, the conditions of
the Theorem~\ref{pontr_th3} are satisfied. Thus, all the zeroes of
the equation $F(y)=0$ are real.

Let us now prove that for each root $y_0$ of the equation $F(y)=0$
the inequality $G(y_0)F'(y_0)<0$ holds. We use the notation $\cos
y=C$, $\sin y=S$:
$$
-GF'=-(-b S - S - y_0C)(y_0 C+y_0 +bS)=(b S + S + y_0 C)(y_0
C+y_0+bS).
$$
We rewrite $\tg y_0 = \frac{b}{y_0}$ as $C=\frac{Sb}{y_0}$ and
substitute $C$ at the last relation:
$$
-GF'=\left(bS + \frac{y_0^2}{b}S+S\right)\left(\frac{y_0^2}{b}S+bS+y_0\right)
=S^2\left(b+\frac{y_0^2}{b}+1\right)\left(\frac{y_0^2}{b}+b+\frac{y_0}{S}\right).
$$
Since $S^2>0$ and $b+\frac{y_0^2}{b}+1>0$, then it remains to prove,
that the third multiplier is greater then zero. From $\tg y_0 =
\frac{b}{y_0}$ we also have $b=\frac{y_0S}{C}$ and substituting we
obtain:
$$
\frac{y_0^2}{b}+b+\frac{y_0}{S} = \frac{y_0C}{S} + \frac{y_0S}{C} +
\frac{y_0}{S} = y_0\frac{C^2+S^2+C}{SC} = y_0\frac{1+C}{SC}=
(1+C)bC^2>0.
$$
We note also that the fact that $y_0SC>0$ can be easily seen from
the picture above.

and use the sufficient condition b) of Theorem~\ref{pontr_th3}.

Applying Theorem~\ref{pontr_th3} we complete the proof of the fact
that all the eigenvalues of the operator ${\cal A}$ belongs to the
open left half-plane.
\end{proof}

\begin{ut}\label{ut_u15}
If $s=0$ then the operator $\cal A$ possesses eigenvectors only,
i.e. it possesses no root vectors; if $s\not=0$ then to each
eigenvalue $\lambda\in\sigma({\cal A})$ there corresponds a pair of
eigenvector and root vector of the operator $\cal A$.
\end{ut}

\begin{proof}
If $s=0$ then $\Delta_{\cal A}(\lambda_k)$ is the zero-matrix for
each eigenvalue $\lambda_k$, and,
therefore, the the space of solutions of the equation
$\Delta_{\cal A}(\lambda_k)z=0$ is two-dimensional. We choose the
following basis of this space: $z^1_k=(1,0)^T$, $z^2_k=(0,1)^T$.

Let us consider the equation for eigenvectors:
$$
({\cal A}-\lambda_k I)\left(
\begin{array}{c}
y\\
z(\theta)
\end{array}
\right)=0 \Leftrightarrow \left(
\begin{array}{c}
A_0z(0) - \lambda_k y\\
\dot{z}(\theta)-\lambda_k z(\theta)
\end{array}
\right)=0.
$$
The solution of the second equation is given by
$z(\theta)={\rm e}^{\lambda_k \theta}z(0)$, thus,
$z(-1)={\rm e}^{-\lambda_k}z(0)$. Taking into account the domain of the
operator ${\cal A}$:
$y=z(0)-A_{-1}z(-1)=(I-{\rm e}^{-\lambda_k}A_{-1})z(0)$, we obtain from
the first equation: $(A_0 - \lambda_k I + \lambda_k
{\rm e}^{-\lambda_k}A_{-1})z(0)=0$, or, in an equivalent form:
$\Delta_{\cal A}(\lambda_k)z(0)=0$.

As we have noted above the last equation has two-dimensional
solution: $z(0)=z^1_k=(1,0)^T$ and $z(0)=z^2_k=(0,1)^T$, and
therefore, there is a two-dimensional eigenspace of the operator
${\cal A}$ corresponding to the eigenvalue $\lambda_k$. Two
eigenvectors of this subspace can be chosen in the following form:
$f^1_k=\left(
\begin{array}{c}
y^1_k\\
z^1_k(\cdot)
\end{array}
\right)$, where $y^1_k=\left(
\begin{array}{c}
1+{\rm e}^{-\lambda_k}\\
0
\end{array}
\right)$, $z^1_k(\theta)=\left(
\begin{array}{c}
{\rm e}^{\lambda_k \theta}\\
0
\end{array}
\right)$ and $f^2_k=\left(
\begin{array}{c}
y^2_k\\
z^2_k(\cdot)
\end{array}
\right)$, where $y^2_k=\left(
\begin{array}{c}
0\\
1+{\rm e}^{-\lambda_k}
\end{array}
\right)$, $z^2_k(\theta)=\left(
\begin{array}{c}
0\\
{\rm e}^{\lambda_k \theta}
\end{array}
\right)$.

Thus, to any eigenvalue $\lambda_k$ of the operator $\cal A$ there
corresponds the two-dimensional eigenspace.

If  $s=1$ (or $s\not=0$) we have that $\Delta_{\cal
A}(\lambda_k)=\left(
\begin{array}{cc}
0 & {\rm e}^{-\lambda_k}\\
0 & 0
\end{array}
\right)$ and, therefore, the space of solutions of the equation
$\Delta_{\cal A}(\lambda_k)z=0$ is one-dimensional:
$z^1_k=(1,0)^T$ (since, obviously, ${\rm e}^{-\lambda_k}\not=0$). Thus,
the equation $\Delta_{\cal A}(\lambda_k)z=z^1_k$ has also
one-dimensional solution which we denote by $z^2_k=(0,1)^T$.

Arguing as above we show that the operator  $\cal A$ possesses one
eigenvector corresponding to the eigenvalue $\lambda_k$:
$f^1_k=\left(
\begin{array}{c}
y^1_k\\
z^1_k(\cdot)
\end{array}
\right)$, where $y^1_k=\left(
\begin{array}{c}
1+{\rm e}^{-\lambda_k}\\
0
\end{array}
\right)$, $z^1_k(\theta)=\left(
\begin{array}{c}
{\rm e}^{\lambda_k \theta}\\
0
\end{array}
\right)$.

Let us show that the operator ${\cal A}$ possesses one root vector.
Each root vector $f$ of the operator ${\cal A}$ satisfies the
following equation: $({\cal A}-\lambda_k I)^2f=0$, i.e.
$\widetilde{f}=({\cal A}-\lambda_k I)f\in \Ker({\cal A}-\lambda_k
I)$. From the last we conclude that $\widetilde{f}=f^1_k=\left(
\begin{array}{c}
(1+{\rm e}^{-\lambda_k}, 0)^T\\
({\rm e}^{\lambda_k \theta}, 0)^T
\end{array}
\right)$ and the root vector satisfies the relation
$$
\left(
\begin{array}{c}
A_0z(-1) - \lambda_k y\\
\dot{z}(\theta)-\lambda_k z(\theta)
\end{array}
\right)=\left(
\begin{array}{c}
(1+{\rm e}^{-\lambda_k}, 0)^T\\
({\rm e}^{\lambda_k \theta}, 0)^T
\end{array}
\right).
$$
The solution of the second equation is given by
$$z(\theta)={\rm e}^{\lambda_k \theta}z(0)+
\int_{0}^\theta {\rm e}^{\lambda_k (\theta-\tau)}{\rm e}^{\lambda_k \tau}y_k^1
\;{\rm d}\tau= {\rm e}^{\lambda_k \theta}z(0)+\theta {\rm e}^{\lambda_k
\theta}y_k^1,$$ what implies
$z(-1)={\rm e}^{-\lambda_k}z(0)-{\rm e}^{-\lambda_k}y_k^1$. Taking into account
the domain of the operator ${\cal A}$, we write down the first
equation in the following form:
$$
A_0z(0)-\lambda_k\left(z(0)-A_{-1}({\rm e}^{-\lambda_k}z(0)-{\rm e}^{-\lambda_k}y_k^1)\right)=(1+{\rm e}^{-\lambda_k})y_k^1
$$
which is equivalent to
$$
\Delta_{\cal A}(\lambda_k)z(0)= (1+{\rm e}^{-\lambda_k})y_k^1+\lambda_k
A_{-1}{\rm e}^{-\lambda_k}y_k^1= \left(
\begin{array}{c}
1+{\rm e}^{-\lambda_k}-\lambda_k {\rm e}^{-\lambda_k}\\
0
\end{array}
\right).
$$

Thus, if $1+{\rm e}^{-\lambda_k}-\lambda_k {\rm e}^{-\lambda_k}\not=0$ then
$z(0)$ is the root vector of the matrix $\Delta_{\cal
A}(\lambda_k)$ and, therefore, $z(0)=z^2_k=(0,1)^T$ and the operator
${\cal A}$ possesses the root vector $f^2_k = \left(
\begin{array}{c}
y^2_k\\
z^2_k(\cdot)
\end{array}
\right)$, where $y^2_k=\left(
\begin{array}{c}
{\rm e}^{-\lambda_k}\\
1+{\rm e}^{-\lambda_k}
\end{array}
\right)$ and $z^2_k(\theta)=\left(
\begin{array}{c}
\theta {\rm e}^{\lambda_k \theta}\\
{\rm e}^{\lambda_k \theta}
\end{array}
\right)$.

It remains to show that
$1+{\rm e}^{-\lambda_k}-\lambda_k {\rm e}^{-\lambda_k}\not=0$. We suppose the
contrary: $\lambda_k {\rm e}^{-\lambda_k}=1+{\rm e}^{-\lambda_k}$ and,
multiplying the last expression onto ${\rm e}^{\lambda_k}$, we obtain
$\lambda_k={\rm e}^{\lambda_k}+1$. Since $\lambda_k$ is an eigenvalue of
${\cal A}$, then we obtain
${\rm e}^{\lambda_k}({\rm e}^{\lambda_k}+1)+{\rm e}^{\lambda_k}+1+{\rm e}^{\lambda_k}=0$ or
${\rm e}^{2\lambda_k}+3{\rm e}^{\lambda_k}+1=0$, and we conclude that
${\rm e}^{\lambda_k}=\frac{-3\pm \sqrt{5}}{2}$. For the root
${\rm e}^{\lambda_k}=\frac{-3- \sqrt{5}}{2}$ we have that
$\MRe\lambda_k>0$, and for the root
 ${\rm e}^{\lambda_k}=\frac{-3+ \sqrt{5}}{2}$ we have:
$\lambda_k=\frac{-3+ \sqrt{5}}{2}+1=\frac{\sqrt{5}-1}{2}>0$. We have
obtained the contradiction what completes the analysis of
eigenvectors and root vectors of the operator ${\cal A}$.
\end{proof}

\begin{ut}\label{ut_u16}
If $s=0$ then the system (\ref{syst_delay_ex}) is strongly asymptotically stable; if
$s\not=0$ then the system is unstable.
\end{ut}

\begin{proof} First we consider the system (\ref{syst_delay_ex}) when
$s=0$ and prove that in this case the system is stable. Let us
evaluate the norm of eigenvectors.
$$
\|f^1_k\|^2=\|f^2_k\|^2=(1+{\rm e}^{-\lambda_k})(\overline{1+{\rm e}^{-\lambda_k}})+ \int_{-1}^0 {\rm e}^{\lambda_k\theta}
\overline{{\rm e}^{\lambda_k\theta}} \;{\rm d}\theta=
|1+{\rm e}^{-\lambda_k}|^2+\int_{-1}^0 {\rm e}^{2\MRe\lambda_k\theta} \;{\rm
d}\theta
$$
$$
=|1+{\rm e}^{-\lambda_k}|^2+\frac{1}{2\MRe\lambda_k}\left(1-{\rm e}^{-2\MRe\lambda_k}\right).
$$
Since $\MRe\lambda_k\rightarrow 0$ when $k\rightarrow \infty$, then
$\mathop{\lim}\limits_{k\rightarrow
\infty}\frac{1}{2\MRe\lambda_k}\left(1-{\rm e}^{-2\MRe\lambda_k}\right)=1$.
Therefore, $0<C_1\le
\frac{1}{2\MRe\lambda_k}\left(1-{\rm e}^{-2\MRe\lambda_k}\right)\le C_2$.
Taking into account that $|{\rm e}^{-\lambda_k}|\le C_3$, we obtain the
following estimates:
\begin{equation}\label{ex_riesz}
C_1\le \|f^i_k\|^2\le (1+C_3)^2+C_2=C_4.
\end{equation}
As it has been shown in \cite{Rabah_Sklyar_Rezounenko_2005}, the
subspaces $V^{(k)}={\rm Lin}\{f^1_k, f^2_k\}$ (and the
finite-dimensional subspace $W_N$) form a Riesz basis of the space
$M_2$. Since we have proved the estimate (\ref{ex_riesz}) then the
eigenvectors $\{f^1_k, f^2_k\}$ (together with vectors from $W_N$)
form a basis of $M_2$. Therefore, we have a Riesz basis of
eigenvectors. Further we just repeat the proof given in
\cite[Theorem 23]{Rabah_Sklyar_Rezounenko_2005}.

We consider a norm $\|\cdot\|_1$ in which the eigenvectors $\{f^1_k,
f^2_k\}_{k\in {\mathbb Z}}$ are orthogonal. Let a vector $x$ belongs
to a closed span of the subspaces $V^{(k)}$, then
$x=\sum\limits_{k\in {\mathbb Z}}(\alpha_k f^1_k + \beta_k f^2_k)$
and we have:
$$
{\rm e}^{{\cal A} t}x= \sum_{k\in {\mathbb Z}}{\rm e}^{\lambda_k t}(\alpha_k
f^1_k + \beta_k f^2_k).
$$
Therefore,
$$\|{\rm e}^{{\cal A} t}x\|_1^2= \sum_{k\in {\mathbb Z}}{\rm e}^{\lambda_k t}(\|\alpha_k f^1_k\|_1^2 + \|\beta_k f^2_k\|_1^2)\le
\sum_{k\in {\mathbb Z}}(\|\alpha_k f^1_k\|_1^2 + \|\beta_k
f^2_k\|_1^2)=\|x\|_1^2$$ and, thus, the family ${\rm e}^{{\cal A} t}$ is
uniformly bounded in the subspace generated by the subspaces
$V^{(k)}$. From the last we conclude that the system is strongly
asymptotically stable.

Let us consider the system (\ref{syst_delay_ex}) when $s=1$.

Let an operator $\cal A$ has a sequence of eigenvalues
$\{\lambda_k\}_{k=1}^\infty$ such that $\MRe \lambda_k<0$ and $\MRe
\lambda_k\rightarrow 0$ when $k\rightarrow \infty$ and to each
$\lambda_k$ there corresponds one eigenvector $v^k$ and at least one
root vector $w^k$. We show that the equation $\dot{x}={\cal A}x$ is
unstable. Let us suppose that $\|v^k\|=\|w^k\|=1$. Since, for each
$w^k$ we have ${\rm e}^{{\cal A}t}w^k={\rm e}^{\lambda_k t}(tv^k+w^k)$ then
$$
\|{\rm e}^{{\cal A}t}w^k\|=|{\rm e}^{\lambda_k t}|\|(tv^k+w^k)\|\ge {\rm e}^{\MRe
\lambda_k t} (t-1).
$$

For any constant $C>0$ we take $t\ge 2C+1$ and for this $t$ we take
big enough $k$ such that ${\rm e}^{\MRe \lambda_k t}\ge\frac12$. Then we
have:
$$
\|{\rm e}^{{\cal A}t}w^k\|\ge\frac12(2C+1-1)=C
$$
and we conclude that $\|{\rm e}^{{\cal A}t}\|\ge C$ for $t\ge 2C+1$.
Therefore, the family of exponents ${\rm e}^{{\cal A}t}$ is not uniformly
bounded and because of Banach-Steinhaus theorem there exists $x\in
D(A)$ such that $\|{\rm e}^{{\cal A}t}x\|\rightarrow \infty$ when $t\rightarrow +\infty$.

Thus, the system (\ref{syst_delay_ex}) is unstable when $s=1$. The
last completes the proof of the proposition.
\end{proof}
\section{Stabilizability analysis}\label{sect_stabilizability}

In this section we prove Theorem~\ref{thr_stabilizability_intro} for control systems~(\ref{syst_delay_gen_cont}) with $\det A_{-1}=0$.
It is convenient to reformulate this theorem as follows.
\begin{thr}[on stabilizability]\label{thr_stabilizability}
Let $b_1, \ldots, b_p\in\mathbb{C}^n$ be the columns of the matrix $B$. Assume that the following four
conditions are satisfied.
\begin{enumerate}
\item[{\em (1)}]  All the eigenvalues of the matrix $A_{-1}$ satisfy $|\mu|\le 1$.
\item [{\em (2)}] All the eigenvalues $\mu\in \sigma_1$ are simple.
\item[{\em (3)}]  $\sum_{i=1}^p |\langle b_i, y \rangle_{\mathbb{C}^n}\not=0$ for all $i=1,\ldots,p$
and all vectors $y$ satisfying $y\in {\rm Ker}\triangle^*_{\mathcal A}(\lambda)$
for roots $\lambda$ of the equation $\det \triangle^*_{\mathcal A}(\lambda)=0$,
such that ${\rm Re} \lambda \ge 0$.
\item[{\em (4)}] $\sum_{i=1}^p |\langle b_i, y_m \rangle_{\mathbb{C}^n}\not=0$ for all vectors $y_m$ of the matrix $A_{-1}^*$,
corresponding to eigenvalues $\overline{\mu}_m$, $\mu_m\in\sigma_1$ and for all $i=1,\ldots,p$.
\end{enumerate}
Then there exists a regular control $u={\mathcal F} x$ of the form
\begin{equation}\label{eq_control}
u={\mathcal F} x=\int_{-1}^0 F_2(\theta)\dot{z}_t(\theta) \dd\theta +\int_{-1}^0 F_3(\theta)z_t(\theta) \dd\theta
\end{equation}
where $x=(y, z(\cdot))\in {\mathcal D}({\mathcal A})$, $F_2(\cdot), F_3(\cdot)\in L_2(-1, 0; \mathbb{C}^{n\times p})$.
And this control stabilizes the system~(\ref{syst_delay_gen_cont}),
i.e. ${\mathcal D}({\mathcal A})={\mathcal D}({\mathcal A}+ {\mathcal B} {\mathcal F})$ and
${\rm e}^{t({\mathcal A}+ {\mathcal B} {\mathcal F})}x_0\rightarrow 0$ as $t\rightarrow\infty$ for all $x_0\in M_2$.
\end{thr}

The controllability conditions (3)--(4) of Theorem~\ref{thr_stabilizability} are equivalent to (3)--(4) of Theorem~\ref{thr_stabilizability_intro}
(for more details concerning these conditions see \cite{Rabah_Sklyar_2007}).
We also note that regular stabilizability for a particular case of the control systems~(\ref{syst_delay_gen_cont})
had been considered in \cite{Rabah_Sklyar_MatFiz_2004}.
Before giving the proof, let us discuss the conditions (1) and (2).
Since the regular feedback does not change the matrix $A_{-1}$ we need the assumption $\sigma(A_{-1})\subset\{\mu:\; |\mu|\le 1\}$.
Moreover, taking into account the results of Theorem~\ref{thr_stability} on strong stability,
we conclude that by means of a regular feedback it is possible to stabilize the system in the case when
the algebraic multiplicity of each eigenvalue $\mu\in\sigma_1$ equals to $1$.
In this case the closed loop system will be asymptotically stable if and only if all the eigenvectors of ${\mathcal A}+{\mathcal B}{\mathcal F}$
are in the left half-plane. Thus, the problem of regular stabilizability for such systems consists
in assigning the spectrum of the system in the left half-plane.
The most intensional problem in the situation appears when an infinite number of the eigenvalues located ``close'' to the imaginary axis
belongs to the right half-plane. This means, in particular, that $\sigma_1\not=\emptyset$.

\begin{proof}[Proof of Theorem~\ref{thr_stabilizability}]
Let us construct the decomposition of the spectrum of $\mathcal A$ introduced in Section~\ref{subsect_decomp_stabilizability} by (\ref{eq_eq22-2}):
\begin{equation}\label{eqn_adf}
\sigma({\mathcal A})=\Lambda_0({\mathcal A})\cup \Lambda_1({\mathcal A})\cup \Lambda_2({\mathcal A}),
\end{equation}
where the subsets are given by (\ref{eq_eq22-3}).
Since $\Lambda_0({\mathcal A})\subset\{\lambda: \; {\rm Re} \lambda\le -\varepsilon\}$, $\varepsilon>0$ then our aim is to construct the
feedback which moves the eigenvalues of $\Lambda_1({\mathcal A})$ and $\Lambda_2({\mathcal A})$ to the left half-plane.

We begin with moving the spectral set $\Lambda_2({\mathcal A})$.
The spectral decomposition of the state space $M_2= M_2^0 \oplus M_2^{1}\oplus M_2^{2}$ (Theorem~\ref{thr_decomposition_stab}),
corresponding to~(\ref{eqn_adf}),
allows us to rewrite the system~(\ref{syst_delay_opmodel_stab}) as
\begin{equation}\label{eqn_new_operator1}
\left\{
\begin{array}{c}
\dot{x_0}=\mathcal{A}_0 x_0 + \mathcal{B}_0 u, \qquad x_0\in M_2^0=M_2^0(\mathcal A)\\
\dot{x_1}=\mathcal{A}_1 x_1 + \mathcal{B}_1 u, \qquad x_1\in M_2^1=M_2^1(\mathcal A)\\
\dot{x_1}=\mathcal{A}_2 x_2 + \mathcal{B}_2 u, \qquad x_1\in M_2^1=M_2^2(\mathcal A),
\end{array}
\right.
\end{equation}
where $\mathcal{A}_0=\mathcal{A}|_{M_2^0}$, $\mathcal{A}_1=\mathcal{A}|_{M_2^{1}}$,  $\mathcal{A}_2=\mathcal{A}|_{M_2^2}$
and $\mathcal{B}_i$ are the projections of $\mathcal{B}$ onto subspaces $M_2^i$.

The spectrum $\Lambda_2({\mathcal A})$ of the operator $\mathcal{A}_2$ defined on
the finite-dimensional subspace $M_2^2$ belongs to the right half-plane. Due to the assumption~(3) of the theorem
all the eigenvalues of $\mathcal{A}_2$ are controllable and, thus,
they may be assigned arbitrary to the half-plane $\{\lambda: \; {\rm Re} \lambda\le -\varepsilon\}$.
Namely, let us consider a feedback $u={\mathcal F}_2x$, $\mathcal F_2: M_2\rightarrow \mathbb{C}^p$ which acts as follows:
$$
{\mathcal F}_2x_0={\mathcal F}_2x_1=0,\; x_0\in M_2^0, x_1\in M_2^1,\quad {\mathcal F}_2x_2=\widehat{\mathcal F}_2 x_2,\; x_2\in M_2^2,
$$
where $\widehat{\mathcal F}: M_2^2\rightarrow \mathbb{C}^p$ is a bounded operator.
It is easy to see that such form of feedback allows to move only eigenvalues of $\Lambda_2({\mathcal A})$,
i.e. $\Lambda_0({\mathcal A})\cup \Lambda_1({\mathcal A}) \subset \sigma(\mathcal{A} + \mathcal{B} \mathcal{F}_2)$.
Thus, we conclude that there exists $\widehat{\mathcal F}$ such that the spectrum of the closed-loop system $\mathcal{A} + \mathcal{B} \mathcal{F}_2$
is of the form
\begin{equation}\label{new_spectrum}
\sigma(\mathcal{A} + \mathcal{B} \mathcal{F}_2)=\sigma(\mathcal{A})\backslash \Lambda_2(\mathcal{A}) \cup \widehat{\Lambda}_2,
\quad \widehat{\Lambda}_2\subset\{\lambda: \; {\rm Re} \lambda\le -\varepsilon\}.
\end{equation}

We choose the control in the form $u={\mathcal F}_2x + v$, denote
$\widehat{\mathcal A}\:{\stackrel{\rm def}{=}}{\mathcal A}+{\mathcal B}{\mathcal F}_0$ and rewrite
the system~(\ref{syst_delay_opmodel_stab}) as
\begin{equation}\label{syst_delay_opmodel_stab_modif}
\dot{x}=\widehat{\mathcal A}x + {\mathcal B}v.
\end{equation}
We emphasize that due to the form of the feedback ${\mathcal F}_2$ the infinite-dimensional system~(\ref{syst_delay_opmodel_stab_modif})
corresponds to the neutral type system~(\ref{syst_delay_gen_cont}) with the same matrix $A_{-1}$.

Let us now construct the decomposition of the spectrum (\ref{eq_eq22-2}) for the operator $\widehat{\mathcal A}$.
Due to~(\ref{new_spectrum}), we obtain $\Lambda_2(\widehat{\mathcal A})=\emptyset$ and, thus,
$$
\sigma(\widehat{\mathcal A})=\Lambda_0(\widehat{\mathcal A})\cup \Lambda_1(\widehat{\mathcal A}),
\qquad \Lambda_1(\widehat{\mathcal A}) = \Lambda_1(\mathcal A).
$$

Applying Theorem~\ref{thr_decomposition_stab}, we construct the spectral decomposition of the state space
$M_2= M_2^0 \oplus M_2^{1}$. The operator model~(\ref{syst_delay_opmodel_stab_modif}) may be rewritten as follows:
\begin{equation}\label{eqn_new_operator}
\left\{
\begin{array}{c}
\dot{x_0}=\widehat{\mathcal{A}}_0 x_0 + \widehat{\mathcal{B}}_0 v, \qquad x_0\in M_2^0=M_2^0(\widehat{\mathcal{A}})\\
\dot{x_1}=\widehat{\mathcal{A}}_1 x_1 + \widehat{\mathcal{B}}_1 v, \qquad x_1\in M_2^1=M_2^1(\widehat{\mathcal{A}}),
\end{array}
\right.
\end{equation}
where $\widehat{\mathcal{A}}_0=\widehat{\mathcal{A}}|_{M_2^0}$, $\widehat{\mathcal{A}}_1=\widehat{\mathcal{A}}|_{M_2^{1}}$
and $\widehat{\mathcal{B}}_i$ are the projections of the vectors of $\mathcal{B}$ onto subspaces $M_2^i$.

Due to Theorem~\ref{thr_boundedness_stab} the restriction of the resolvent
$R(\lambda, \widehat{\mathcal{A}})|_{M_2^0}$ is uniformly bounded
on the set $\{\lambda: {\rm Re} \lambda \ge 0\}$ and, therefore,
the semigroup $\{{\rm e}^{t \widehat{\mathcal A}}|_{M_2^0}\}_{t\ge 0}$ is exponentially stable.

To stabilize the second equation we apply the approach introduced in \cite{Rabah_Sklyar_Rezounenko_2008}
which is based on the abstract theorem on infinite pole assignment.
Below, for the sake of completeness and due to the specific form of the operator $\widehat{\mathcal{A}}_1$,
we give a simplified formulation of the mentioned theorem (Theorem~\ref{thr_inf_p_assignment}).

The theorem on infinite pole assignment holds for a single input system.
However, as it is shown in \cite{Rabah_Sklyar_Rezounenko_2008},
the multivariable case may be reduced to the single input case by the following
considerations (see \cite{Rabah_Sklyar_Rezounenko_2008} and also the classical result on
finite-dimensional control systems in \cite{Wonham_1985}).

Let us consider the vector ${\bf b}$ given by
\begin{equation}\label{eq_stab_432}
{\bf b}=
\left(
\begin{array}{c}
b\\
0
\end{array}
\right)
=
\left(
\begin{array}{c}
c_1b_1+\ldots+c_p b_p\\
0
\end{array}
\right)
=
\left(
\begin{array}{c}
B\\
0
\end{array}
\right)c
=
{\mathcal B}c, \quad
c=
\left(
\begin{array}{c}
c_1\\
\vdots\\
c_p
\end{array}
\right),
\end{equation}
where $\{b_1, \ldots, b_p\}\subset\mathbb{C}^n$ are the columns of the matrix $B$.
The exist numbers $c_i\in\mathbb{C}$ such that the following relations hold (Lemma~\ref{stz_multi-single}):

\noindent (a) $\langle b, y_m\rangle_{\mathbb{C}^n}\not=0$ for all eigenvectors $y_m$ of the matrix $A_{-1}^*$ corresponding
to the eigenvalues $\overline{\mu_m}$, where $\mu_m\in\sigma_1$, $m=1,\ldots,\ell_1$.

\noindent (b) $\langle b, y_m^k\rangle_{\mathbb{C}^n}\not=0$ for all such $y_m^k$
which satisfy $y_m^k\in \Ker \triangle^*(\overline{\lambda_m^k})$ and ${\rm Re} \lambda_m^k\ge 0$.

Let us  denote the projection of ${\bf b}$ onto the subspace $M_2^1$  by ${\bf b}_1$ and consider the
single input system
\begin{equation}\label{eq_stab_432dsf}
\dot{x_1}=\widehat{\mathcal{A}}_1 x_1 + {\bf b}_1 \widehat{u},\quad \widehat{u}\in\mathbb{C}.
\end{equation}

Due to the construction, the eigenvalues of $\widehat{\mathcal{A}}_1$ are simple
and each eigenvalue $\lambda_m^k$, $m=1,\ldots,\ell_1$, $|k|\ge N_1$ belongs to the circle $L_m^{k}(r^{(k)})$
centered at $\widetilde{\lambda}_m^{k}={\rm i}(\arg \mu_m +2\pi k)$.
The corresponding eigenvectors $\varphi_m^k$ form a Riesz basis (Proposition~\ref{lm_lm333}).

Thus, the single input system~(\ref{eq_stab_432dsf}) satisfies the conditions (H1)--(H4) of Theorem~\ref{thr_inf_p_assignment}.
According to that theorem the spectrum $\Lambda_1(\mathcal A)$
of the operator $\widehat{\mathcal{A}}_1$ may be moved by the regular feedback~(\ref{eq_control})
as follows.

Let us chose scalars $\widehat{\lambda}_m^{k}$ such that ${\rm Re}\widehat{\lambda}_m^{k}<0$
and which are located inside the circles $L_m^{k}(r^{(k)})$.
There exists a feedback $\widehat{u}=\widehat{\mathcal F}_1x$,
$\widehat{\mathcal F}_1: M_2\rightarrow \mathbb{C}$
such that $\widehat{\lambda}_m^{k}$ are eigenvalues of the operator $\widehat{\mathcal A}_1+{\bf b}_1 \widehat{\mathcal F}_1$.
Taking into account~(\ref{eq_stab_432}), we define the feedback $v=\mathcal F_1 x$, $\mathcal F_1: M_2\rightarrow \mathbb{C}^p$ as follows
$$F_1 x{\stackrel{\rm def}{=}} c \widehat{\mathcal F}_1 x.$$
Due to the theorem on stability the semigroup $\{{\rm e}^{t (\widehat{\mathcal A}+{\mathcal B}{\mathcal F}_1)}|_{M_2^1}\}_{t\ge 0}$
is asymptotically stable. Thus, the feedback
$$
u={\mathcal F}_1 x_1+{\mathcal F}_2 x_2
$$
transforms the original system into a system where all the conditions of Theorem~\ref{thr_stability} on asymptotic stability are verified.
The last completes the proof of Theorem~\ref{thr_stabilizability}.
\end{proof}

\begin{rem}
We would like to emphasize that in the present paper to prove the result on stabilizability for the case $\det A_{-1}=0$
we have contributed the ideas which are mainly of technical character: the direct decomposition of the state space and
the proof of the resolvent boundedness on some subspace.
However, the main contribution from the stabilizability point of view is the abstract theorem on infinite pole assignment
which had been proved in~\cite{Rabah_Sklyar_Rezounenko_2008}.
\end{rem}

\begin{rem}
The stabilizability of the restriction of the system onto the subspace $M_2^2(\mathcal A)$ follows
also from some classical results. Let us denote by $\Gamma_\delta$ a rectifiable, simple, closed curve which
surround the spectral set $\Lambda_2(\mathcal A)$ and $\Gamma_\delta\subset\{\lambda: \; {\rm Re} \lambda\ge \varepsilon\}$.
By $P_2=\frac{1}{2\pi {\rm i}}\int\limits_{\Gamma_\delta} R(\mathcal A, \lambda) \dd \lambda$ we denote the spectral projector.
The subspaces $P_2 M_2$ and $(I-P_2)M_2$ are $\mathcal A$-invariant and the spectrum of the restriction
$(I-P_2)M_2$ belongs to the half-plane $\{\lambda: \; {\rm Re} \lambda< \varepsilon\}$.
According to \cite{Triggiani_JMAA_1975_2} the spectral set  $\Lambda_2(\mathcal A)$ may be assigned arbitrary
to the half-plane  $\{\lambda: \; {\rm Re} \lambda< \varepsilon\}$
by means of a finite rank input operator (under some controllability conditions).
The development of this approach was given in \cite{Nefedov_Sholokhovich_1986} and \cite{Jacobson_Nett_1988}
(see e.g. \cite{Curtain_Zwart_1995} as a survey).
\end{rem}

\begin{rem}
Let us discuss the assumptions (3)--(4) of Theorem~\ref{thr_stabilizability_intro}:

\noindent {\em (3)} ${\rm rank}(\triangle(\lambda),\; B)=n$ for all $\lambda: {\rm Re} \lambda \ge 0$.

\noindent {\em (4)} ${\rm rank}(\mu I - A_{-1},\; B)=n$ for all $\mu\in\sigma_1$.

The assumptions (3) may include an infinite number of relations (obviously, it should be verified only for $\lambda$ which are
eigenvalues of $\mathcal A$, i.e. $\det \triangle_{\mathcal A}(\lambda)=0$;
there may exist an infinite number of eigenvalues of $\mathcal A$ belonging to the right half-plane).
However, only a finite number of these relations have to be verified.
More precisely, there exists $M>0$ such that for any eigenvalue $\lambda$ such that $|{\rm Im}\lambda| > M$
the condition (3) follows from the condition (4).
\end{rem}
\begin{proof}
Let $A_{-1}$ be of the form~(\ref{eq_eq34}) and we prove the proposition for eigenvalues located inside the
circles $L_1^{k}(r^{(k)})$, i.e. for $\lambda_1^{k}$ (for all other indices $m=1,\ldots,\ell_1$ the idea of the proof remains the same).
The matrix $A_{-1}-\mu_1 I=A_{-1}-{\rm e}^{-\widetilde{\lambda}_1^k} I$ is of the form:
$$
A_{-1}-\mu_1 I=
\left(
\begin{array}{cc}
0 & 0\\
0 & \widehat{A}
\end{array}
\right),
\qquad \widehat{A}\in \mathbb{C}^{(n-1)\times (n-1)},\; \det \widehat{A}\not=0.
$$
Due to the assumption (4) there exists $i\in\overline{1,p}$ such that the column $b=(b_{1i},\ldots,b_{n i})^T$ of thee
matrix $B$ possesses a nonzero first component: $b_{1i}\not=0$.
Besides, we note that since $\det \widehat{A}\not=0$ then there exists the unique $\alpha=(\alpha_{2},\ldots, \alpha_{n})^T\in \mathbb{C}^{n-1}$
such that $\widehat{b}=\widehat{A}\alpha$, where $\widehat{b}=(b_{2i},\ldots,b_{ni})^T\in \mathbb{C}^{n-1}$.

Let us analyze the structure of $\triangle(\lambda_1^k)$, we rewrite it as follows:
$$\triangle(\lambda_1^k)=\lambda_1^k {\rm e}^{-\lambda_1^k} (A_{-1}-\mu_1 I + D_k),$$
where
$$
D_k=({\rm e}^{\widetilde{\lambda}_1^k}-{\rm e}^{\lambda_1^k})I
+ {\rm e}^{\lambda_1^k}\int_{-1}^0 {\rm e}^{\lambda_1^k \theta} \left(A_2(\theta) + \frac{1}{\lambda_1^k} A_3(\theta)\right) \dd \theta.
$$
From the last formula we conclude, that $D_k\rightarrow 0$, when $k\rightarrow\infty$.

Let us show that there exists $N\in\mathbb{N}$ such that for any $k:\; |k|\ge N$:
\begin{equation}\label{eq_rem1}
{\rm rank}(\triangle(\lambda_1^k),\; b)=n,
\end{equation}
which is equivalent to the statement of the proposition.

Since $\det \triangle(\lambda_1^k)=0$ then $\det (A_{-1}-\mu_1 I + D_k)=0$.
Let us denote by $\widehat{D}_k=\{(D_k)_{ij}\}_{i,j=2}^n$. Since $\det \widehat{A}\not=0$, then
there exists $N\in\mathbb{N}$ such that for any $k:\; |k|\ge N$ we have: $\det (\widehat{A}+\widehat{D}_k)\not=0$.
Therefore, there exists the unique vector $\alpha_k\in\mathbb{C}^n$ such that $\widehat{b}=(\widehat{A}+\widehat{D}_k)\alpha_k$.
From the last we conclude
$$
\alpha_k=(\widehat{A}+\widehat{D}_k)^{-1}\widehat{b}= (I+\widehat{A}^{-1}\widehat{D}_k)^{-1} \widehat{A}^{-1}\widehat{b}=
(I+\widehat{A}^{-1}\widehat{D}_k)^{-1} \alpha.
$$
Since $(I+\widehat{A}^{-1}\widehat{D}_k)^{-1}\rightarrow I$ when $k\rightarrow\infty$,
we conclude that $\alpha_k\rightarrow \alpha$.

If we suppose that (\ref{eq_rem1}) does not hold, then $b_1=b_{1i}$ allows the following representation:
$$
b_{1i}=(\alpha_k)_2 (D_k)_{12} + (\alpha_k)_{n} (D_k)_{1n}.
$$
However, since $(D_k)_{1i}\rightarrow 0$, then the right-hand side of the last relation tends to zero when $k\rightarrow\infty$.
We have come to the contrary, which completes the proof of the proposition.
\end{proof}

As it was mentioned above, for the sake of completeness
we give further the formulations of the results on infinite pole assignment from \cite{Rabah_Sklyar_Rezounenko_2008}.
We formulate them in a form, which takes into account the specific form of the operator $\widehat{\mathcal{A}}_1$.

\begin{thr}[On infinite pole assignment \cite{Rabah_Sklyar_Rezounenko_2008}]\label{thr_inf_p_assignment}
Let $H$ be a complex Hilbert space, $\mathcal A$ be an infinitesimal generator of a $C_0$-semigroup in $H$,
and the control system if given by $\dot{x}={\mathcal A}x + {\mathcal B}u$, $x\in \mathcal{D}(\mathcal A)\subset H$.
Let $\mu_1,\ldots, \mu_\ell$ be some nonzero complex numbers and we introduce complex numbers
$$
\widetilde{\lambda}_m^{k}=\ln |\mu_m|+ {\rm i}(\arg \mu_m +2\pi k), \quad m=1,\ldots,\ell, k\in\mathbb{Z},
$$
and the circles $L_m^{k}(r^{(k)})$ centered at $\widetilde{\lambda}_m^{k}$ with radii $r^{(k)}$
satisfying the relation $\sum\limits_{k\in\mathbb{Z}}(r^{(k)})^2 < \infty$.

Let the following assumptions hold:
\begin{enumerate}
\item[{\em (H1)}]  The spectrum of  $\mathcal A$ consists only of eigenvalues which are located in the circles $L_m^{k}(r^{(k)})$.
Moreover, all the eigenvalues of $\mathcal A$ are simple (i.e. its algebraic multiplicity equals $1$)
and there exists $N_1\in\mathbb{N}$ such that for any $k: |k|\ge N_1$ the total multiplicity
of the eigenvalues contained in the circles $L_m^{k}(r^{(k)})$ equals $1$.
\item [{\em (H2)}] The corresponding eigenvectors, which we denote by $\varphi_m^k$,
constitute a Riesz basis in $H$.
\item[{\em (H3)}] The system $\dot{x}={\mathcal A}x + {\mathcal B}u$ is of a single input, i.e.
the operator $\mathcal B: \mathbb{C}\rightarrow H$ is the operator of multiplication by ${\bf b}\in H$.
\item[{\em (H4)}] We assume the following controllability condition: ${\bf b}$ is not orthogonal to eigenvectors $\psi_m^k$
of the operator ${\mathcal A}^*$: $\langle {\bf b}, \psi_m^k \rangle \not=0$ and
\begin{equation}
\lim_{k\rightarrow\infty} k|\langle {\bf b}, \psi_m^k \rangle|= c_m, \qquad 0<c_m<+\infty.
\end{equation}
\end{enumerate}
Then there exists $N_2\ge N_1$ such that for any family of complex numbers $\widehat{\lambda}_m^{k}\in L_m^{k}(r^{(k)})$, $|k|\ge N_2$
there exists a linear control ${\mathcal F}: \mathcal{D}(\mathcal A)\rightarrow \mathbb{C}$, such that
\begin{enumerate}
\item[{\em (1)}] the complex numbers $\widehat{\lambda}_m^{k}$ are eigenvalues of the operator $\mathcal A+\mathcal B \mathcal F$;
\item[{\em (2)}] the operator $\mathcal B \mathcal F: \mathcal{D}(\mathcal A)\rightarrow H$ is relatively $\mathcal A$-bounded.
\end{enumerate}
\end{thr}

\begin{lem}[\cite{Rabah_Sklyar_Rezounenko_2008}]\label{stz_multi-single}
If the assumptions (3) and (4) of Theorem~\ref{thr_stabilizability} holds, then there exists a vector $b\in {\rm Im} B$,
say $b=c_1b_1+\ldots+c_p b_p$, $c_i\in\mathbb{C}$, such that the following relations hold:

\noindent (a) $\langle b, y_m\rangle_{\mathbb{C}^n}\not=0$ for all eigenvectors $y_m$ of the matrix $A_{-1}^*$ corresponding
to the eigenvalues $\overline{\mu_m}$, where $\mu_m\in\sigma_1$, $m=1,\ldots,\ell_1$.

\noindent (b) $\langle b, y_m^k\rangle_{\mathbb{C}^n}\not=0$ for all such $y_m^k$
which satisfy $y_m^k\in \Ker \triangle^*(\overline{\lambda_m^k})$ and ${\rm Re} \lambda_m^k\ge 0$.
\end{lem}
\begin{proof}
Let us prove the relations (a). Consider the subspaces
$$
L_m=\{y:\; \langle y, y_m\rangle_{\mathbb{C}^n}=0\},\quad m=1,\ldots,\ell_1.
$$
These subspaces are of dimension $n-1$. Let us consider also $M_m=L_m\cap {\rm Im} B$.
By the assumption (4) we have: $\dim M_m<p$. Indeed, if $\dim M_m=p=\dim {\rm Im} B$, then $B^* y=0$
and the condition (4) is not satisfied. Thus, each $M_m$ is nowhere dense in ${\rm Im} B$ and
due to Baire theorem we have:
$$
\mathop\bigcup\limits_{m=1}^{\ell_1}M_m\not={\rm Im} B.
$$
The last means that there exists $b\in {\rm Im} B$ such that $b\not\in L_m$ for all $m=1,\ldots,\ell_1$,
i.e. the relations  $\langle b, y_m\rangle_{\mathbb{C}^n}\not=0$ hold.

The inequalities (b) follows from the relation
$$
|\langle b, y_m^k\rangle_{\mathbb{C}^n}|\thicksim C_m (|k|+1)^{-1}
$$
which is proved in \cite{Rabah_Sklyar_Rezounenko_2008} independently of the condition $\det A_{-1}\not=0$.
\end{proof}
\section{Conclusion and perspectives}
In the present paper we have generalized the results on strong asymptotic non-exponential stability
and regular stabilizability for the case of mixed retarded-neutral type systems.
The proofs of the mentioned generalizations are technically complicated and requires subtle estimates.
We combine the Riesz basis technic with the analysis of the boundedness of the resolvent on some ${\mathcal A}$-invariant subspaces.

As a perspective, we consider systems with the difference operator $K$ given by
$$
Kf=\sum_{i=1}^r A_{h_i}f(h_i), \qquad h_i\in [-1,0].
$$
Besides, the dilemma of item~(iii) of Theorem~\ref{thr_stability_intro} on stability may be investigated more precisely.

\vskip2ex

\noindent
{\bf Acknowledgments.}

This work was supported in part by \'Ecole  Centrale  de Nantes.

\end{document}